\documentclass{amsart}

\usepackage{amsmath,amssymb,latexsym,geometry}
\usepackage{graphicx}
\usepackage{color}
\usepackage{subfigure}
\usepackage{multirow}
\usepackage{mathptmx}
\usepackage{url,bm,overpic}
\usepackage{enumerate}

\newtheorem{theorem}{Theorem}[section]
\newtheorem{corollary}{Corollary}[section]
\newtheorem{proposition}{Proposition}[section]
\newtheorem{lemma}{Lemma}[section]
\newtheorem{remark}{Remark}[section]
\newtheorem{Assumption}{Assumption}[section]

\newtheorem{definition}{Definition}[section]

\newcommand{\Second}{\textup{I}\!\textup{I}}
\numberwithin{equation}{section}
\newcommand\descitem[1]{\item{\bfseries #1}\\}

\begin{document}

\title{WHEN LOCALLY LINEAR EMBEDDING HITS BOUNDARY}

\author{Hau-Tieng Wu}
\address{Courant Institute of Mathematical Sciences \\
New York University, New York, NY, 10012  United States}
\email{hauwu@cims.nyu.edu}

\author{Nan Wu}
\address{Department of Mathematical Sciences\\
The University of Texas at Dallas, Richardson, TX 75080, United States}
\email{nan.wu@utdallas.edu}

\begin{abstract}{Based on the Riemannian manifold model, we study the asymptotic behavior of a widely applied unsupervised learning algorithm, {\em locally linear embedding (LLE)}, when the point cloud is sampled from a compact, smooth manifold  {\em with boundary}. We show several peculiar behaviors of LLE near the boundary that are different from those diffusion-based algorithms. In particular, we show that LLE pointwisely converges to a {\em mixed-type differential operator with degeneracy} and we calculate the convergence rate. The impact of the hyperbolic part of the operator is discussed and we propose a clipped LLE algorithm which is a potential approach to recover the Dirichlet Laplace-Beltrami operator.}
\end{abstract}


\maketitle

\section{Introduction}
Arguably, unsupervised learning is the holy grail of artificial intelligence. While a lot of challenges are on different fronts, many attempts have been explored, including ISOMAP \cite{Tenenbaum_deSilva_Langford:2000}, locally linear embedding (LLE) \cite{Roweis_Saul:2000}, Hessian LLE \cite{Donoho_Grimes:2003}, eigenmap \cite{Belkin_Niyogi:2003}, diffusion map (DM) \cite{Coifman_Lafon:2006}, vector diffusion map (VDM) \cite{Singer_Wu:2012}, t-distributed stochastic neighboring embedding \cite{VanderMaaten_Hinton:2008}, maximal variation unfolding \cite{Weinberger_Saul:2006}, to name but a few. In this paper, based on the Riemannian manifold model, we study the {asymptotic} behavior of LLE when the point cloud is sampled from a compact, smooth manifold with boundary.

LLE is an algorithm based on a rudimentary idea -- by well parametrizing the dataset locally, we can patch all local information to recover the global one.
It has been widely applied in different fields and has been cited more than 15,800 times according to Google Scholar. However, its theoretical justification for data points sampled on compact manifolds without boundary was only made available at the end of 2017 \cite{Wu_Wu:2017,2018arXiv180402811M}. 
Essentially, the established theory says that under the manifold without boundary setup, LLE has several peculiar behaviors that are different from those of diffusion-based algorithms, including eigenmap and DM.
First, unlike DM, LLE may not behave like a diffusion process since the associated kernel function is not always positive.
Second, it is sensitive to the regularization, and different regularizations lead to different differential operators. If the regularization is chosen properly, LLE asymptotically converges to the Laplace-Beltrami operator without extra probability density function (p.d.f.) estimation, even if the p.d.f. is not uniform. However, when the regularization is not chosen properly, LLE converges to a fourth order differential operator in the cases like the spheres.
Third, when the regularization is chosen properly, the convergence of LLE to the Laplace-Beltrami operator is comparable to that of DM with a proper normalization \cite{Coifman_Lafon:2006,Singer_Wu:2016, cheng2022eigen}. 
Fourth, the kernel associated with LLE is in general not symmetric, and this asymmetric kernel depends on the curvature and p.d.f. information. 
Fifth, the kernel depends on the local covariance matrix analysis and the Mahalanobis distance, since it is the {mix-up} of the ordinary kernel and a special kernel depending on the Mahalanobis distance \cite{2018arXiv180402811M}.

While several theoretical properties have been discussed in \cite{Wu_Wu:2017} and \cite{2018arXiv180402811M}, there are more open problems about LLE left. In this paper, we are interested in exploring the asymptotic behavior of LLE when the manifold has a boundary. 
First, we show that asymptotically LLE pointwisely converges to a mixed-type differential operator with degeneracy and we calculate the  convergence rate. 
Second, after showing that the asymptotic operator near the boundary involves singular coefficients, we study the 1-dim manifold case and relate the eigenvalue problem of LLE to a Sturm-Liouville equation.
Third, through a series of numerical simulations, we explore the impact of the hyperbolic part of the operator. In those simulations, we modify the LLE by clipping certain points that are close to the boundary, which asymptotically is equivalent to eliminating the hyperbolic part of the operator, then we obtain an algorithm that is potential to recover the Laplace-Beltrami operator with the Dirichlet boundary condition. This enlightens a new approach to recovering the Dirichlet Laplace-Beltrami operator.
Fourth, we compare LLE with DM to explain the differences between their behaviors on the boundary.

The paper is organized in the following way. In Section \ref{Section Review}, we review the LLE algorithm and provide some spectral properties of LLE on the linear algebra level. In Section \ref{Section Boundary}, we provide the manifold model when the boundary is not empty, and develop the asymptotic theory for the LLE matrix, particularly the associated kernel behavior and its relationship with the geometrical structure of the manifold. In Section \ref{section numerics}, we discuss the clipped LLE which potentially leads to the Laplace-Beltrami operator with the Dirichlet boundary condition. Numerical simulations of the clipped LLE are provided. %
The paper is closed with the discussion in Section \ref{Section Concludion}. Technical proofs are postponed to the Appendices.
For reproducibility purposes, the Matlab code to reproduce figures in this paper can be downloaded from \url{http://hautiengwu.wordpress.com/code/}.

\section{Review locally linear embedding}\label{Section Review}

{We start with some matrix notations.  For $p,r\in \mathbb{N}$ so that $r\leq p$, let $I_r \in \mathbb{R}^{r \times r}$ be the identity matrix.  Denote $J_{p,r}=\begin{bmatrix}
I_r \\
0  \\
\end{bmatrix} \in \mathbb{R}^{p\times r}$, i.e. the $(i,i)$-th entry of $J_{p,r}$ is $1$ for $i=1,\ldots,r$, and the other entries are zero. Denote $\bar J_{p,r}=\begin{bmatrix}
0 \\
I_r \\
\end{bmatrix}\in \mathbb{R}^{p\times r}$,  i.e. the $(p-r+i,i)$ entry of $J_{p,r}$  is $1$ for $i=1,\ldots,r$, and the other entries are zero. Denote $I_{p,r}:=J_{p,r}J_{p,r}^\top=\begin{bmatrix}
I_r &  0 \\
0 & 0  \\
\end{bmatrix}\in \mathbb{R}^{p\times p}$ and $\bar I_{p,r}:=\bar J_{p,r}\bar J_{p,r}^\top =\begin{bmatrix}
0 & 0 \\
0 & I_r \\
\end{bmatrix}\in \mathbb{R}^{p\times p}$. Finally, for $d \leq r \leq p$, define
$\mathfrak{J}_{p,r-d}:=\bar{J}_{p,p-d} J_{p-d,r-d} \in \mathbb{R}^{p\times(r-d)}$. }

We quickly recall necessary information about LLE and refer readers with interest in more discussion to \cite{Roweis_Saul:2000,Wu_Wu:2017}.  The key ingredient of LLE is the {\em barycentric coordinate}, which is a quantity shown in \cite{Wu_Wu:2017} to be parallel to the kernel chosen in the graph Laplacian. Suppose we have the point cloud {$\mathcal{X}=\{z_i\}_{i=1}^n \subset \mathbb{R}^p$.} There are two nearest neighbor search schemes to proceed. The first one is the {{\em $\epsilon$-radius ball scheme}}. Fix $\epsilon>0$. For $z_k \in \mathcal{X}$, assume there are $N_k$ data points, excluding $z_k$, in the $\epsilon$-radius ball centered at $z_k$. The second one is the {{\em$K$-nearest neighbor (KNN) scheme}} used in the original LLE algorithm \cite{Roweis_Saul:2000}; that is, for a fixed $K\in\mathbb{N}$, find the $K$ neighboring points. 

Fix one nearest neighbor search scheme, and denote the nearest neighbors of $z_k\in \mathcal{X}$ as $\mathcal{N}_k=\{z_{k,i}\}_{i=1}^{N_k}$.
Then the barycentric coordinate of $z_k$ associated with $\mathcal{N}_k$, denoted as $w_k$, is defined as the solution of the following optimization problem:

\begin{equation} \label{Definition:FindBarycentric}
w_{k}= \mathop{\arg\min}_{w \in \mathbb{R}^{N_k},\, w^\top  \boldsymbol{1}_{N_k}=1} \Big\|z_k-\sum_{j=1}^{N_k}w(j)z_{k,j} \Big\|^2 = \mathop{\arg\min}_{w \in \mathbb{R}^{N_k},\, w^\top  \boldsymbol{1}_{N_k}=1} w^\top G_{n,k}^\top G_{n,k}w  \in \mathbb{R}^{N_k},
\end{equation}
where $\boldsymbol{1}_{N_k}$ is a vector in $\mathbb{R}^{N_k}$ with all entries $1$ and
\begin{equation} \label{def:local data matrix}
G_{n,k}:=\begin{bmatrix}
| & & |\\
z_{k,1}-z_k & \ldots & z_{k,N_k}-z_k\\
| & & | 
\end{bmatrix}\in \mathbb{R}^{p \times {N_k}}
\end{equation}
is called the \textit{local data matrix}.
In general, $G_{n,k}^\top G_{n,k}$ might be singular, and it is suggested in \cite{Roweis_Saul:2000} to stabilize the algorithm by regularizing the equation and solve
\begin{align}
(G_{n,k}^\top G_{n,k}+c I_{N_k\times N_k})y_k=\boldsymbol{1}_{N_k}\,, \quad w_k=\frac{y_k}{y_k^\top \boldsymbol{1}_{N_k}},  \label{Section2:wk}
\end{align}
where $c>0$ is the {\em regularizer} chosen by the user. As is shown in \cite{Wu_Wu:2017}, the regularizer plays a critical role in LLE. With the barycentric coordinate of $x_k$ for $k=1,\ldots,n$, the {\em LLE matrix}, which is a $n \times n$ matrix denoted as $W$, is defined as
\begin{equation}\label{Section2:Wki}
W_{ki}= \left\{ \begin{array}{ll} w_k(j) & \mbox{if $z_i=z_{k,j} \in \mathcal{N}_k$};\\ 0 & \mbox{otherwise}. \end{array} \right.
\end{equation}
The barycentric coordinates are invariant under rotation and translation, because the matrix $G_{n,k}$ is invariant under translation, and $G_{n,k}^\top G_{n,k}$ is invariant under rotation.  As discussed in \cite{Wu_Wu:2017}, the barycentric coordinates can be understood as the projection of $\boldsymbol{1}_{N_k}$ onto the null space of $G^\top_{n,k} G_{n,k}$ .

Suppose $r_n=\texttt{rank}(G_{n,k}^\top G_{n,k})$. {Note that $r_n=\texttt{rank}(G_{n,k})=\texttt{rank}(G_{n,k}^\top G_{n,k})=\texttt{rank}(G_{n,k} G_{n,k}^\top) \leq min(N_k, p) \leq p$} and $G_{n,k}G_{n,k}^\top$ is positive (semi-)definite. Denote the eigen-decomposition of the matrix $G_{n,k} G_{n,k}^\top $ as  $U_n \Lambda_n U_n^\top$, where 
$\Lambda_n=\texttt{diag}(\lambda_{n,1},{\lambda}_{n,2},\ldots,{\lambda}_{n,p})$, 
${\lambda}_{n,1} \geq {\lambda}_{n,2} \geq \cdots \geq {\lambda}_{n,r_n}> {\lambda}_{n,r_n+1}= \cdots ={\lambda}_{n,p}=0$, and $U_n \in O(p)$. Denote
\begin{equation}\label{Definition:Irho:Soft}
\mathcal{I}_c(G_{n,k}G_{n,k}^\top):=U_nI_{p,r_n} (\Lambda_n+c I_{p\times p})^{-1} U_n^\top,
\end{equation}
and
{\begin{align}
\mathbf{T}_{n,x_k}:= \mathcal{I}_c(G_{n,k}G_{n,k}^\top)  G_{n,k}\boldsymbol{1}_{N_k} \label{Definition:Tn}\,.
\end{align}}
Then, it is shown in \cite[Section 2]{Wu_Wu:2017} that the solution to (\ref{Section2:wk}) is 
\begin{align}
y_{k}^\top =&\,c^{-1}\boldsymbol{1}_{N_k}^\top- c^{-1} \mathbf{T}^\top_{n,x_k}  G_{n,k} \label{solution y_n},
\end{align}
and hence 
\begin{align}
w^\top _k&\,=\frac{\boldsymbol{1}_{N_k}^\top-\mathbf{T}^\top_{n,x_k}  G_{n,k}}
{N_k-\mathbf{T}^\top_{n,x_k}  G_{n,k}\boldsymbol{1}_{N_k}}\,.\label{Expansion:LLEweightedKernel}
\end{align}
Note that $N_k-\mathbf{T}^\top_{n,x_k}  G_{n,k}\boldsymbol{1}_{N_k}$ in the denominator of (\ref{Expansion:LLEweightedKernel}) is the sum of entries of $\boldsymbol{1}_{N_k}^\top-\mathbf{T}^\top_{n,x_k}  G_{n,k}$ in the  numerator, so we could view $y^\top _k$ as the ``kernel function'' associated with LLE, and $w^\top_k$ as the normalized kernel. 

{To reduce the dimension of $\mathcal{X}$, it is suggested in \cite{Roweis_Saul:2000} to embed $\mathcal{X}$ into a low dimension Euclidean space via
\begin{equation}
z_k\mapsto Y_k=[v_1(k) ,\cdots, v_{\ell}(k)]^\top  \in \mathbb{R}^{\ell}
\end{equation}
for each $z_k\in \mathcal{X}$, where $\ell\in\mathbb{N}$ is the dimension of the embedded points chosen by the user and $v_1, \cdots ,v_{\ell}\in \mathbb{R}^n$ are eigenvectors of $(I-W)^\top (I-W)$ corresponding to the $\ell$ smallest eigenvalues. }

\subsection{Spectral properties of the LLE matrix}\label{spectral LLE matrix}

We provide some spectral properties of the LLE matrix. 
Unlike the graph Laplacian (GL), in general, $W$ is not a symmetric matrix or a Markov transition matrix, according to the analysis shown in \cite{Wu_Wu:2017}.
For $A\in \mathbb{R}^{n\times n}$, let $\sigma(A)\subset \mathbb{C}$ be the spectrum of $A$ and define $\rho(A)$ to be the spectral radius of $A$. 

\begin{proposition}\label{Proposition 1}
The LLE matrix $W\in \mathbb{R}^{n\times n}$ satisfies $\rho(W) \geq 1$. 
\end{proposition}

{The proof of the proposition is straightforward. Since $W\boldsymbol{1}=\boldsymbol{1}$, where $\boldsymbol{1}$ is an $n$-dim vector with all entries $1$, $1\in \sigma(W)$. Thus we have that $\rho(W) \geq 1$. In Appendix \ref{Appendix examples for propostion 1}, we construct an example to show that it is possible for $\rho(W) > 1$.}

\

Since in general, the LLE matrix $W$ may not be symmetric, the eigenvalues might be complex and can be complicated. For example, in the null case that $400$ points are sampled independently and identically from a $200$-dim Gaussian random vector with $0$ mean and identity covariance, the eigenvalue distribution of $W$ spreads on the complex plane. See Figure \ref{Figure:RandomNoise} for the distribution of such a dataset.
\begin{figure}[ht]
\center
\includegraphics[width=.4\textwidth]{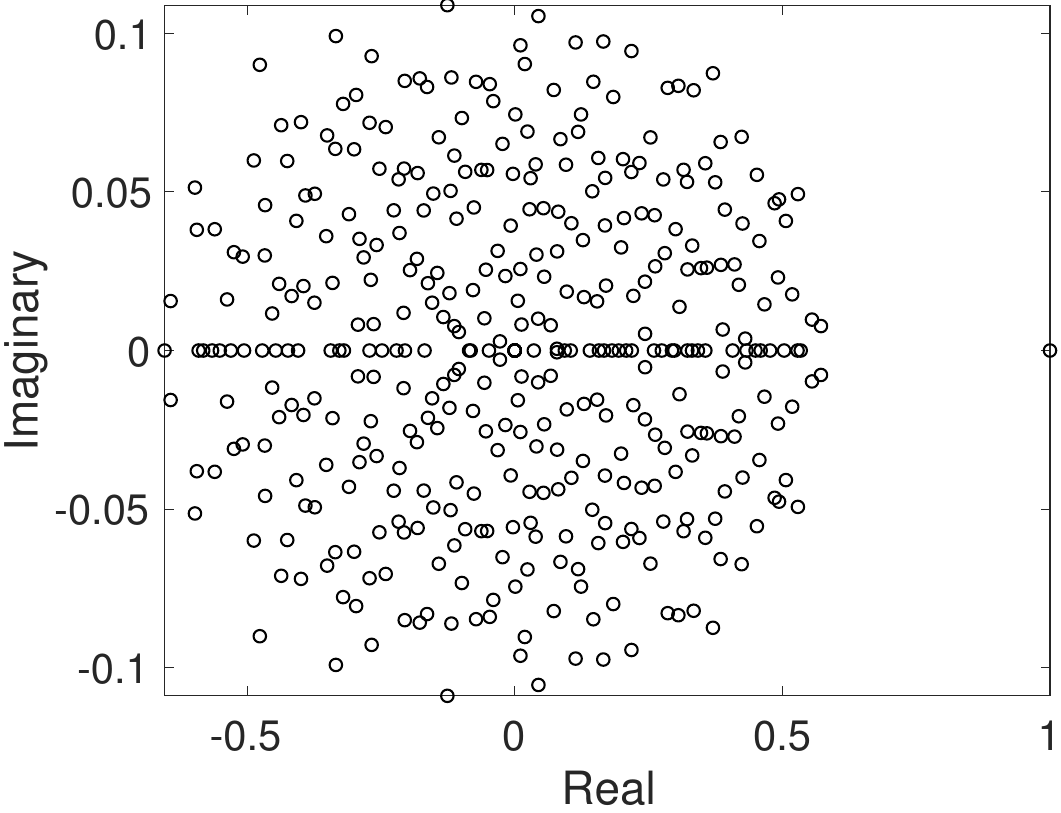}
\caption{The distribution of eigenvalues of the LLE matrix, where $W$ is constructed with $50$ nearest neighbors. In this example, the top eigenvalue is $1$.}
\label{Figure:RandomNoise}
\end{figure}

However, in some special cases, we can well control the imaginary part of the distribution. Consider the symmetric and anti-symmetric parts of $W$,
 $W^+=(W+W^\top)/2$ and $W^-=(W-W^\top)/2$,
so that $W=W^++W^-$. 
By applying the Bauer-Fike theorem with the $L^2$ norm and Holder's inequality, for any eigenvalue $\lambda$ of $W$, there is a real eigenvalue $\mu$ of $W^+$ such that $|\lambda-\mu| \leq \|W^-\|_2 \leq \sqrt{\|W^-\|_1\|W^-\|_\infty}$. Below we show that the imaginary part of eigenvalues of the LLE matrix $W$ is well controlled under some conditions. 

\begin{proposition}\label{Proposition:ImaginaryControl}
Denote $N=\max_k N_k$, where $N_k=|\mathcal{N}_k|$.
If $\max_{i,j}|W_{ij}-{W_{ji}}| \leq \frac{C\epsilon}{N}$ for some $C\geq 0$, the imaginary part of eigenvalues of the LLE matrix $W$ is of order $\epsilon$. 
\end{proposition}

\begin{proof}
{ Note that $W_{ij}^-=0$ if $\|z_i-z_j\|_2 \geq \epsilon$ and $W_{ij}^-$ might be nonzero if $\|z_i-z_j\|_2< \epsilon$.} Since $\sqrt{\|W^-\|_1\|W^-\|_\infty} \leq N \max_{i,j}|W_{ij}-W_{ji}|$, based on the assumption, the imaginary part of eigenvalues of $W$ is bounded by $O(\epsilon)$.
\end{proof}

Note that $\max_{i,j}|W_{ij}-{W_{ji}}|$ measures the similarity of different $\epsilon$-neighborhood $\mathcal{N}_k$. Thus, the assumption that $\max_{i,j}|W_{ij}-{W_{ji}}| \leq \frac{C\epsilon}{N}$ for some $C\geq 0$ means that the affinity graph is ``not too imbalanced''. This assumption holds asymptotically under the manifold setup.

Since the KNN scheme and the $\epsilon$-radius ball scheme are directly related under a suitable manipulation as is shown in \cite[Section 5]{Wu_Wu:2017}, from now on we fix to the $\epsilon$-radius ball scheme in the rest of the paper for the sake of theoretical analysis.

\section{Preliminaries for LLE under the manifold with boundary setup}\label{Section Boundary}
\subsection{Main assumptions}
{In this subsection, we summarize the major assumptions that we need in this paper.  First, we have the following assumption about the manifold $M$. 
\begin{Assumption}\label{major assumption 1}
Let $(M,g)$ be a d-dimensional compact, smooth Riemannian manifold with boundary isometrically embedded in $\mathbb{R}^p$ via $\iota:M \hookrightarrow \mathbb{R}^p$. We assume the boundary of $M$ is smooth. 
\end{Assumption}

Next, we make the following assumption about the sample points on the manifold $M$. 
\begin{Assumption}\label{major assumption 2}
Suppose $(\Omega, \mathcal{F}, \texttt{P})$ is a probability space, where $\texttt{P}$ is a  probability measure defined on the Borel sigma algebra $\mathcal{F}$ on $\Omega$. Let  $X$ be a random variable on $(\Omega, \mathcal{F}, \texttt{P})$ with the range on $(M,g)$.  We assume $\mathsf{P}:=X_*\texttt{P}$ is absolutely continuous with respect to the volume measure on $M$ associated with $g$ so that $d\mathsf{P}=P dV$ by the Radon-Nikodym theorem, where $dV$ is the volume form of $M$ and $P$ is a non-negative function defined on $M$. We call $P$ the probability density function (p.d.f.) associated with $X$. We further assume $P \in C^2(M)$  and $0 <P_m \leq P(x) \leq P_M$ for all $x \in M$.  We assume $\{x_1 \cdots, x_n\} \subset M$ are i.i.d. sampled from $P$. 
\end{Assumption}

\begin{remark}
Under the regularity assumption of the boundary and the density function in this model, in general, the chance to sample a point on the boundary is zero, unless we further assume the knowledge of the boundary and sample on the boundary. Without the knowledge of the boundary, an estimate of the boundary is therefore needed. {Such estimate has wide applications including the distance to the boundary estimation and kernel density estimation on a manifold with boundary. We refer the readers to \cite{berry2017density} for a discussion. }
\end{remark}

\begin{remark}
We refer the readers to \cite{lee2012smooth} for a discussion of the differentiability of a function on the boundary of the manifold. Compared with the $P\in C^5(M)$ requirement imposed in \cite{Wu_Wu:2017}, in this work we only assume $P\in C^2(M)$. In \cite{Wu_Wu:2017}, we need $P\in C^5(M)$ to explore the regularization effect on the whole algorithm. In this work, since we will fix the regularization and focus on the boundary, $P\in C^2(M)$ is sufficient. 
\end{remark}

We adopt the notations in Section \ref{Section Review}. Let $\mathcal{X}=\{z_i =\iota(x_i)\}_{i=1}^n$. Fix $\epsilon>0$, we propose the choice of the parameter for regularization: 
$$c=n \epsilon^{d+3}.$$ 
 Then, we construct the LLE matrix $W$ by using $\mathcal{X}$ and $c=n \epsilon^{d+3}$ as shown in \eqref{Section2:wk} and \eqref{Section2:Wki}.
}

\subsection{Manifold with boundary setup}\label{Section Manifold Setup}
The { manifold} setup is nowadays standard and has been considered to study several algorithms, including Eigenmap \cite{Belkin_Niyogi:2007}, DM \cite{Coifman_Lafon:2006,trillos2018error}, VDM \cite{Singer_Wu:2012,Singer_Wu:2016}, LLE \cite{Wu_Wu:2017} and several others, like the gradient estimation \cite{mukherjee2010learning}, diffusion on the fiber structure \cite{Lin_Minasian_Wu:2016,Gao:2016}, Bayesian regression \cite{yang2016bayesian}, extrinsic local regression \cite{lin2017extrinsic}, image processing model \cite{osher2017low}, sensor fusion algorithm \cite{shnitzer2018recovering}, to name but a few. Although the manifold model is standard, when the boundary is non-empty, it is less discussed in the literature. We introduce the following setup for manifold with boundary. {See \cite{vaughn2019diffusion} for a different treatment. 

Denote $d_g(\cdot,\cdot)$ to be the geodesic distance associated with $g$.
For $\epsilon>0$, define the $\epsilon$-neighborhood of $\partial M$ as 
\begin{equation}
M_{\epsilon}=\{x \in M | d_g(x, \partial M) < \epsilon\}.
\end{equation}
 
For the tangent space $T_xM$ on $x\in M$, denote $\iota_*T_{x}M$ to be the embedded tangent space in $\mathbb{R}^p$ and $(\iota_*T_{x}M)^\bot$ be the normal space at $\iota(x)$. Let $\Second_{x}$ be the second fundamental form of $\iota(M)$ at $\iota(x)$.  Denote $S^{d-1}$ to be the $(d-1)$-dim unit sphere embedded in $\mathbb{R}^p$, and $|S^{d-1}|$ be its volume. Denote $\{e_i\}_{i=1}^p$ to be the canonical basis of $\mathbb{R}^p$, where $e_i$ is a unit vector with $1$ in the $i$-th entry. Since the barycentric coordinate is rotational and translational invariant, without loss of generality, when we analyze local behaviors around $x\in M$ in this paper, we implicitly assume that the manifold has been properly translated and rotated so that $\iota_*T_xM$ is spanned by $e_1,\ldots,e_d$.

 We consider the following extension of $M$ and $\iota(M)$ \cite{wong2008extension}.  By the Whitney extension Theorem \cite{whitney1992analytic}, there is a compact manifold with boundary $\tilde{M}$ and $\delta>0$ satisfying the following properties:
\begin{enumerate}
\item $\tilde{M}$ is an isometric extension of $M$.

\item $d_{\tilde{M}}(\partial M,\,\partial\tilde{M}) \geq \delta$, where $d_{\tilde{M}}$  is the geodesic distance measured in $\tilde{M}$. 

\item $\tilde{M}$ is isometrically embedded in $\mathbb{R}^p$ via $\tilde{\iota}$ such that $\tilde{\iota}|_M=\iota$. Thus, $\iota(M)\subset \tilde{\iota}(\tilde{M})$. 
\end{enumerate}
Due to the above extension, we abuse the notation and use $\exp_x$ to denote the exponential map of $\tilde{M}$ at $x \in \tilde{M}$ and the exponential map of ${M}$ at $x \in {M}$. Note that when $x\in M_{\epsilon}$, the exponential map of $M$ at $x$ may not be well defined on $\iota^{-1}(B_{\epsilon}^{\mathbb{R}^p}(\iota(x))\cap \iota(M))$. For example, consider an annulus in the plane and $\iota(x)$ is a point on the inner circle. However, when $\epsilon<\frac{\delta}{2}$ and $\epsilon$ is small enough, we can make sure that $B_{\epsilon}^{\mathbb{R}^p}(\iota(x))\cap \tilde{\iota}(\tilde{M})$ is contained in the interior of $\tilde{\iota}(\tilde{M})$ and $\exp_x$ is well defined over $\tilde{\iota}^{-1}(B_{\epsilon}^{\mathbb{R}^p}(\iota(x))\cap \tilde{\iota}(\tilde{M}))$.  Since $B_{\epsilon}^{\mathbb{R}^p}(\iota(x))\cap \iota(M) \subset B_{\epsilon}^{\mathbb{R}^p}(\iota(x))\cap \tilde{\iota}(\tilde{M})$, $\exp_x$ is well defined over $\tilde{\iota}^{-1}(B_{\epsilon}^{\mathbb{R}^p}(\iota(x))\cap \iota(M))$. Hence, we conclude that $\exp_x$ is well defined over $\tilde{\iota}^{-1}(B_{\epsilon}^{\mathbb{R}^p}(\iota(x))\cap \iota(M))$ for any $x$ in both  $M_\epsilon$ and $M \setminus M_\epsilon$.

Now, we can handle the $\epsilon$-ball near the boundary.
For $x \in M_{\epsilon}$, define 
\begin{equation}
D_{\epsilon}(x)=(\tilde{\iota} \circ \exp_x)^{-1}(B_{\epsilon}^{\mathbb{R}^p}(\iota(x))\cap \iota(M) ) \subset T_{x}\tilde{M}\,,\nonumber
\end{equation}
where $T_x \tilde{M}$ is identified with $ \mathbb{R}^d$.
Denote $x_\partial:=\arg\min_{y \in  \partial M} d(y,x)$ and 
\begin{equation}\label{Definition tildeepsilon gamma}
\tilde{\epsilon}_x =\min_{y \in \partial M} d(y,x). 
\end{equation}
Due to the smoothness assumption of the boundary, if $\epsilon$ is sufficiently small, such $x_\partial$ is unique. 
Clearly, we have $0\leq \tilde{\epsilon}_x  \leq \epsilon$ when $x \in M_{\epsilon}$. Choose the normal coordinates $\{\partial_i\}_{i=1}^d$ around $x$ so that $x_\partial=\tilde{\iota} \circ \exp_x (\tilde{\epsilon}_x  \partial_d)$. Denote $\gamma_{x}(t)$ to be the unique geodesic with $\gamma_{x}(0)=x_\partial$ and $\gamma_{x}(\tilde{\epsilon}_x )=x$.  We further rotate $\iota(M)$ so that 
$$e_d=\iota_* \frac{d}{dt}\gamma_{x}|_{\tilde{\epsilon}_x}.$$
Hence, when $x=x_\partial \in \partial M$, $e_d$ is the inward normal direction of $\iota(\partial M)$ at $\iota(x)$.

Since $\tilde{M}$ is an isometric extension of $M$ and $\tilde{\iota}$ is an extension of $\iota$, for $x \in M$, we use $\Second_x$ to denote both the second fundamental form of $\iota(M)$ at $\iota(x)$ and  the second fundamental form of $\tilde{\iota}(M)$ at $\tilde{\iota}(x)$.  Recall that the second fundamental form at $x$ is a symmetric bilinear map from $T_{x}M \times T_{x}M$ to $(\iota_*T_{x}M)^\bot$. We define $\Second_{ij}(x)=\Second_x(\partial_i, \partial_j)$ for $i,j=1, \cdots, d$.

When $x$ is close to the boundary, $(\tilde{\iota} \circ \exp_x)^{-1}(B_{\epsilon}^{\mathbb{R}^p}(\iota(x))\cap \iota(\partial M))$ is not empty and can be regarded as the graph of a function. 
Denote $a_{ij}(x_\partial)$, $i,j=1,\ldots,d-1$, to be the second fundamental form of the embedding of $\partial M$ into $M$ at $x_\partial$. 
Then there is a domain $K \subset \mathbb{R}^{d-1}$ and a smooth function $q$ defined on $K$, such that
\begin{align}
&(\tilde{\iota} \circ \exp_x)^{-1}(B_{\epsilon}^{\mathbb{R}^p}(\iota(x))\cap \iota(\partial M)  ) \nonumber\\
=&\,\Big\{\sum_{l=1}^du^l\partial_l \in T_x  M \Big|\, (u^1, \cdots, u^{d-1}) \in K,\,  u^d=q(u^1, \cdots, u^{d-1})\Big\}\,,\nonumber
\end{align}
where $q(u^1, \cdots, u^{d-1})$ can be approximated by
\[
\tilde{\epsilon}_x + \sum_{i,j=1}^{d-1} a_{ij}(x_\partial) u^iu^j
\]  
up to an error depending on a cubic function of $u^1,\ldots,u^{d-1}$. For the sake of self containedness, we provide a proof of this fact in Lemma \ref{Lemma:3.5}.

Note that in general the region $D_{\epsilon}(x)$ may not be symmetric with respect to $x$. We define the {\em symmetrized region} associated with $D_{\epsilon}(x)$.

\begin{definition}\label{coordinates near boundary}
For $x \in M_{\epsilon}$ and $\epsilon>0$ sufficiently small, the symmetrized region associated with $D_{\epsilon}(x)$ is defined as
\begin{equation}
\tilde{D}_{\epsilon}(x)=\Big\{(u_1, \cdots u_d) \in T_x\tilde{M}\Big| \sum_{i=1}^d u_i^2 \leq \epsilon^2 \,\,\, \mbox{and} \,\,\, u_d \leq \tilde{\epsilon}_x +  \sum_{i,j=1}^{d-1} a_{ij}(x_\partial) u_iu_j \Big\}\,.\nonumber
\end{equation}
When $x\in M_\epsilon$, $\tilde{D}_{\epsilon}(x)$ is symmetric across $\partial_1,\ldots,\partial_{d-1}$ since if $(u_1, \cdots, u_i, \cdots u_d) \in \tilde{D}_{\epsilon}(x)$, then $(u_1, \cdots, -u_i, \cdots, u_d) \in \tilde{D}_{\epsilon}(x)$ for $i=1, \cdots, d-1$ by definition. Clearly, the volume of $\tilde{D}_{\epsilon}(x)$ is an approximation of that of $D_{\epsilon}(x)$ up to the third order error term. See Corollary \ref{Lemma:4} for details. 
For $x \not \in M_{\epsilon}$ and $\epsilon$ sufficiently small, define the symmetric region associated with $D_{\epsilon}(x)$ as 
\begin{equation}
\tilde{D}_{\epsilon}(x)=\Big\{(u_1, \cdots u_d) \in T_x\tilde{M}\Big|\, \sum_{i=1}^d u_i^2 \leq \epsilon^2\Big\}\subset T_{x}\tilde{M}.
\end{equation}
\end{definition}

}

\subsection{The augmented vectors and the kernels associated with LLE}

{Before we define the augmented vectors and the kernels associated with LLE, we recall the definition of the {\em local covariance matrix}.} For $x\in M$, we call 
\begin{equation}
C_x:=\mathbb{E}[(\iota(X)-\iota(x))(\iota(X)-\iota(x))^{\top}\chi_{B_{\epsilon}^{\mathbb{R}^p}(\iota(x))}(\iota(X))]\in\mathbb{R}^{p\times p}
\end{equation}
the local covariance matrix at $\iota(x)\in \iota(M)$, which is the covariance matrix considered for the {\em local principal component analysis (PCA)} \cite{Singer_Wu:2012,Cheng_Wu:2013}. 
In this paper, we use the following symbols for the local covariance matrix. For $x\in M$, suppose $\texttt{rank}(C_x)=r\leq p$. Clearly, $r$ depends on $x$, but we ignore $x$ for simplicity. Denote the eigen-decomposition of $C_x$ as $C_x=U_x \Lambda_x U_x^\top$, where $U_x\in O(p)$ is composed of eigenvectors and $\Lambda_x$ is a diagonal matrix with the associated eigenvalues $\lambda_1 \geq \lambda_2 \geq \cdots \geq  \lambda_r>\lambda_{r+1}= \cdots = \lambda_p=0$.
{ The theoretical property of local PCA on the manifold without boundary has been studied in a sequence of works, like \cite{nadler2008finite, Singer_Wu:2012, tyagi2013tangent, Cheng_Wu:2013, kaslovsky2014non, little2017multiscale, Wu_Wu:2017, dunson2021inferring}, and the companion property near the boundary will be discussed in Appendix \ref{AppendixSectionCovariance}.  

In this work,  the regularizer in (\ref{Section2:wk}) we have interest in is  $c=n\epsilon^{d+3}$. For the sake of self-containedness,  we provide an intuitive explanation for the choice of the regularizer when $M$ has no boundary based on the results in \cite{Wu_Wu:2017}.  Under Assumptions \ref{major assumption 1} and \ref{major assumption 2}, let $\mathcal{X}=\{z_i =\iota(x_i)\}_{i=1}^n$. Let $G_{n,k}$ be the local data matrix constructed from $\mathcal{X}$ as defined in \eqref{def:local data matrix}. Then, as $n \rightarrow \infty$, we expect 
$$\frac{1}{n}G_{n,k}G_{n,k}^\top \rightarrow C_{x_k},$$
where $C_{x_k}$ is the local covariance matrix at $\iota(x_k)$. In  \cite{Wu_Wu:2017}, the authors show that the first $d$ eigenvalues of $C_{x_k}$ are of order $\epsilon^{d+2}$, while the rest $p-d$ eigenvalues are bounded by $\epsilon^{d+4}$ when $\epsilon$ is sufficiently small.  Those $p-d$ eigenvalues include the extrinsic geometric information, e.g. the second fundamental form of $\iota(M)$.  Hence, if $\epsilon$ is chosen properly based on $n$, we expect that the first  $d$ eigenvalues of $G_{n,k}G_{n,k}^\top$ are of order $n\epsilon^{d+2}$, while the rest $p-d$ eigenvalues containing the extrinsic geometric information are bounded by $n\epsilon^{d+4}$. Suppose we choose $c=n\epsilon^{d+3}$ in \eqref{Definition:Irho:Soft}. Then last $p-d$ diagonal terms of $\Lambda_n+c I_{p\times p}$ are dominated by $c$ but less than the first $d$ diagonal terms.  Hence, we can eliminate the impact of the extrinsic geometry of $\iota(M)$ in \eqref{Definition:Irho:Soft}. Since the Laplace Beltrami operator only depends on the intrinsic geometry of $M$, we choose $c=n \epsilon^{d+3}$ to construct the LLE matrix in \eqref{Section2:Wki} and \eqref{Expansion:LLEweightedKernel} in order to recover the operator. 
When the boundary is non-empty, since our focus is the boundary effect, we fix this regularizer so that points away from the boundary have a good control.
}

\begin{definition}\label{DefAugmented}
Define the augmented vector at $x\in M$ as 
\begin{align}\label{Augmented vector formula}
\mathbf{T}(x)^\top & = \mathbb{E}[(\iota(X)-\iota(x))\chi_{B_{\epsilon}^{\mathbb{R}^p}(\iota(x))}(\iota(X))] ^\top U_xI_{p,r}(\Lambda_x+\epsilon^{d+3}I_{p\times p})^{-1} U_x^\top \in \mathbb{R}^p\,,
\end{align}
which is a $\mathbb{R}^p$-valued vector field on $M$.
\end{definition}
The nomination of $\mathbf{T}(x)$ comes from analyzing the kernel associated with LLE. It has been shown in \cite[Corollary 3.1]{Wu_Wu:2017} that the kernel associated with LLE 
is not symmetric and is defined as
\begin{align}\label{kernel formula}
K_{\epsilon}(x,y):=\chi_{B_{\epsilon}^{\mathbb{R}^p}(\iota(x))}(\iota(y))-[(\iota(y)-\iota(x))\chi_{B_{\epsilon}^{\mathbb{R}^p}(\iota(x))}(\iota(y))]^\top \mathbf{T}(x)\,.
\end{align}
We call $\mathbf{T}(x)$ the {\em augmented vector} since it augments the symmetric $0-1$ kernel $K(x,y)=\chi_{B_{\epsilon}^{\mathbb{R}^p}(\iota(x))}(\iota(y))$ by the inner product of $\mathbf{T}(x)$ and $[(\iota(y)-\iota(x))\chi_{B_{\epsilon}^{\mathbb{R}^p}(\iota(x))}(\iota(y))]$. 
Notice that the vector $\mathbf{T}_{n,x_k}$ defined in (\ref{Definition:Tn}) is a discretization of $\mathbf{T}(x)$ and the theoretical justification is provided in  Appendix \ref{proof of theorem t0}.

{\subsection{Empirical estimation of the regularizer}
In \cite{Wu_Wu:2017}, the authors propose to choose the regularizer in LLE as $c=n\epsilon^{d+3}$ based on the asymptotic analysis. However, in practice, the dimension of the manifold $d$ is unknown. While it is possible to estimate the dimension, it is usually a challenging mission. We thus need a practical way to determine the regularizer without estimating the dimension. In this subsection, we provide empirical estimators of $c$ {\em without} estimating the dimension of the underlying manifold given a finite sample $\mathcal{X}=\{z_i =\iota(x_i)\}_{i=1}^n$ satisfying Assumptions \ref{major assumption 1} and \ref{major assumption 2}, and the estimator is asymptotically equal to $n\epsilon^{d+3}$ up to a constant.  Suppose $z, z' \in \mathbb{R}^p$.  Define $\mathcal{K}_{\epsilon}(z,z')= \exp(-\frac{\|z-z'\|_{\mathbb{R}^p}}{\epsilon})$, where $\epsilon$ is the same bandwidth in the $\epsilon$ radius ball scheme of LLE.  We suggest considering the following empirical regularizer  
\begin{align}\label{definition:ctilde no dim}
\tilde{c}=\epsilon^3 \texttt{median}\left(\sum_{j=1}^n \mathcal{K}_{\epsilon}(z_i,z_j)\right)\,,
\end{align}
where the median is evaluated over all $z_i \in \mathcal{X}$. The construction of the above estimator is motivated by the kernel density estimation.  Suppose $z=\iota(x)$ for $x \in M$. It is shown in \cite{berry2017density} that if $n \epsilon^d \rightarrow \infty$ and $\epsilon \rightarrow 0$ as $n \rightarrow \infty$, then $\tilde{c}(z) \rightarrow C(x) P(x) n \epsilon^{d+3}$, where $C(x)$ depends on $d$ and $d_g(x,\partial M)$ if $x$ is sufficient close to the boundary, and  $C(x)$ only depends on $d$ if $x$ is away from the boundary. 
This estimator is easy to implement and does not require an estimation of $d$. We shall mention that although we choose the squared exponential kernel for the density estimation, more general kernels can be applied to construct the estimator. We refer the readers to \cite{wu2022strong} for a discussion. 
}

\section{Asymptotic analysis of LLE under the manifold with boundary setup}\label{Asymptotic Analysis}

{In this section, we provide the asymptotic analysis of LLE under the manifold with boundary setup. With the $\epsilon$-radius ball scheme, the asymptotic analysis is achieved in the following 4 steps. We also  summarize the main results of  each step as follows.

\smallskip

Step 1: In subsection \ref{Asymptotic Analysis 1}, we study the augmented vector $\mathbf{T}(x)$ for $x \in M$. We show that the vectors $\mathbf{T}(x)$ form a smooth vector field on $\iota(M)$.  The geometry of the vector field can be intuitively described as follows.   The vector $\mathbf{T}(x)$ almost points toward the normal direction of $\iota(M)$ in the interior region $x \in M \setminus M_{\epsilon}$. However, in $M_{\epsilon}$, $\mathbf{T}(x)$ leans towards the tangent direction of $\iota(M)$ gradually.  It is worth noting that the restriction of the tangent components of $ \mathbf{T}(x)$ on $\iota(\partial M)$ forms an inward normal vector field of $\iota(\partial M)$. 

\smallskip

Step 2:  Since $\mathbf{T}(x)$ is the major ingredient in the definition of the kernel function $K_{\epsilon}(x,y)$, where $x,y \in M$, in subsection \ref{Asymptotic Analysis 2}, we explore the properties of $K_{\epsilon}(x,y)$ by using the properties of $\mathbf{T}(x)$ that we derive in the previous subsection.  Obviously, based on the definition, $K_{\epsilon}(x,y)=0$ when $\iota(y) \not \in B_{\epsilon}^{\mathbb{R}^p}(\iota(x))$. We show that $K_{\epsilon}(x,y)$ is approximately equal to $1$ when $x \in M \setminus M_{\epsilon}$ and $\iota(y) \in B_{\epsilon}^{\mathbb{R}^p}(\iota(x))$. However, when $x \in M_{\epsilon}$, $K_{\epsilon}(x,y)$ is not symmetric and may be negative.  

\smallskip

Step 3: In subsection \ref{Asymptotic Analysis 3}, we provide the variance analysis. First, we use $K_{\epsilon}(x,y)$ to define an integral operator $Q_{\epsilon}$ on $C(M)$. For $f \in C^2(M)$, let $\vec{f}$ be the discretization of $f$ over $\{x_1 ,\cdots, x_n\} \subset M$. Let $W$ be the LLE matrix. Then, we show that $[(W-I)\vec{f}](k) $ converges to $Q_{\epsilon}f (x_k)$ at the rate $O\Big(\frac{\sqrt{\log (n)}}{n^{1/2}\epsilon^{d/2-1}}\Big)$ regardless the $x_k$ in  $M \setminus M_{\epsilon}$ or $M_{\epsilon}$. Hence, the result implies that the pointwise convergence rate of LLE is the same for manifolds with or without boundary. 

\smallskip

Step 4:  In subsection \ref{Asymptotic Analysis 4}, we provide the bias analysis. We define a second order mixed-type differential operator $\mathcal{D}_\epsilon$. We show that for $f \in C^3(M)$, $Q_{\epsilon}f (x) / \epsilon^2$ can be approximated by $\mathcal{D}_\epsilon f(x)$ for all $x \in M$.  The final result connecting $[(W-I)\vec{f}](k)$ and $\mathcal{D}_\epsilon f(x_k)$ is achieved by combining the variance and the bias analysis. }

\subsection{Properties of the augmented vector on manifold with boundary}\label{Asymptotic Analysis 1}

The main challenge to analyze LLE is dealing with the augmented vector. It involves three main players in the data structure, the p.d.f., the curvature, and the boundary if the boundary is not empty. Clearly,  when $x$ is close to the boundary, {the  term $\mathbb{E}[(\iota(X)-\iota(x))\chi_{B_{\epsilon}^{\mathbb{R}^p}(\iota(x))}(\iota(X))]$ in $\mathbf{T}(x)$} includes the geometry of the boundary, and the integration will depend on the p.d.f.. On the other hand, while a manifold can be locally well approximated by an affine space, the curvature {appears in the eigenvalues of the local covariance matrix. Hence, the term $(\Lambda_x+\epsilon^{d+3}I_{p\times p})^{-1}$ in $\mathbf{T}(x)$ involves the curvature.} Dealing with these terms requires a careful asymptotic analysis. To alleviate the heavy notation toward this goal, we consider the following functions, and their role will become clear along the theory development.

\begin{definition} \label{Lemma:5:summary} 
Suppose $\epsilon$ is sufficiently small. We define the following functions on $[0,\infty)$, where $\frac{|S^{d-2}|}{d-1}$ is defined to be $1$ when $d=1$.
\begin{align}
\sigma_{0}(t)&:= \left\{
\begin{array}{ll}
\frac{|S^{d-1}|}{2d}+\frac{|S^{d-2}|}{d-1}\int_{0}^{\frac{t}{\epsilon}} (1-x^2)^{\frac{d-1}{2}}dx& \mbox{ for }0 \leq t \leq \epsilon\\
\frac{|S^{d-1}|}{d} & \mbox{ for }t > \epsilon
\end{array}\right.\\
\sigma_{1,d}(t)&:=\left\{\begin{array}{ll}
-\frac{|S^{d-2}|}{d^2-1}(1-(\frac{t}{\epsilon})^2)^{\frac{d+1}{2}}&\mbox{ for }0 \leq t \leq \epsilon\\
0&\mbox{ otherwise}
\end{array}\right.\nonumber\\
\sigma_{2}(t)&:=
\left\{
\begin{array}{ll}
\frac{|S^{d-1}|}{2d(d+2)}+\frac{|S^{d-2}|}{d^2-1}\int_{0}^{\frac{t}{\epsilon}} (1-x^2)^{\frac{d+1}{2}} dx& \mbox{ for }0 \leq t \leq \epsilon\\
\frac{|S^{d-1}|}{d(d+2)} &\mbox{ otherwise}
\end{array}\right.\nonumber\\
\sigma_{2,d}(t)&:=\left\{
\begin{array}{ll}
\frac{|S^{d-1}|}{2d(d+2)}+\frac{|S^{d-2}|}{d-1}\int_{0}^{\frac{t}{\epsilon}} (1-x^2)^{\frac{d-1}{2}}x^2 dx&\mbox{ for } 0 \leq t \leq \epsilon\\
\frac{|S^{d-1}|}{d(d+2)} & \mbox{ otherwise}
\end{array}\right.\nonumber\\
\sigma_{3}(t)&:=
\left\{
\begin{array}{ll}
-\frac{|S^{d-2}|}{(d^2-1)(d+3)}(1-(\frac{t}{\epsilon})^2)^{\frac{d+3}{2}}&\mbox{ for }0 \leq t \leq \epsilon\\
0&\mbox{ otherwise} 
\end{array}\right.\nonumber\\
\sigma_{3,d}(t)&:=
\left\{
\begin{array}{ll}
-\frac{|S^{d-2}|}{(d^2-1)(d+3)}(2+(d+1)(\frac{t}{\epsilon})^2)(1-(\frac{t}{\epsilon})^2)^{\frac{d+1}{2}}&\mbox{ for }0 \leq t \leq \epsilon\\
0& \mbox{ otherwise}
\end{array}\right.\nonumber
\end{align}
\end{definition}

Note that these functions are of order $1$ when $t\leq \epsilon$. 
These seemingly complicated formulas share a simple geometric picture. If $\mathcal{R}$ is the region between the unit sphere and the hyperspace $x_d=\frac{t}{\epsilon}$ in $\mathbb{R}^d$ with coordinates $\{x_1, \cdots, x_d\}$, where $0\leq t\leq \epsilon$, then $\sigma_{0}(t)$, $\sigma_{1,d}(t)$, $\sigma_{2}(t)$, $\sigma_{2,d}(t)$, $\sigma_{3}(t)$  and $\sigma_{3,d}(t)$ are expansions of the integrals of $1$, $x_d$, $x^2_1$, $x^2_d$, $x^2_1 x_d$ and $x^3_d$ over $\mathcal{R}$ respectively.
All the above functions are differentiable of all orders except when $t=\epsilon$.  The regularity of the functions at $t=\epsilon$  depends on $d$. For example $\sigma_{0}(t)$ is at least $C^0$ at $t=\epsilon$  and the other functions are at least $C^1$ at $t=\epsilon$.

With these notations, the behavior of $\mathbf{T}(x)$, particularly when $x$ is near the boundary, can be fully described.

\begin{proposition} \label{augmented vector behavior}
Decompose $\mathbf{T}(x)=\mathbf{T}^{tan}(x)+\mathbf{T}^{per}(x)$, where $\mathbf{T}^{tan}(x)$ is the tangential component of $\mathbf{T}(x)$ and $\mathbf{T}^{per}(x)$ is the normal component of $\mathbf{T}(x)$; that is, $\mathbf{T}^{tan}(x)\in\iota_*T_{x}M$ and $\mathbf{T}^{per}(x)\in(\iota_*T_{x}M)^\bot$. If $x \in M_{\epsilon}$, then
\begin{align}
\mathbf{T}^{tan}(x)=\,&\frac{\sigma_{1,d}(\tilde{\epsilon}_x )}{\sigma_{2,d}(\tilde{\epsilon}_x )}\frac{1}{\epsilon} e_d +O(1)\nonumber\\
\mathbf{T}^{per}(x)=\,&\frac{P(x)}{2} \Bigg[ \Big(\sigma_{2}(\tilde{\epsilon}_x )-\frac{\sigma_{1,d}(\tilde{\epsilon}_x )}{\sigma_{2,d}(\tilde{\epsilon}_x )}\sigma_{3}(\tilde{\epsilon}_x )\Big) \sum_{j=1}^{d-1} \Second_{jj}(x)\nonumber\\
&\qquad\qquad+\Big(\sigma_{2,d}(\tilde{\epsilon}_x )-\frac{\sigma_{1,d}(\tilde{\epsilon}_x )}{\sigma_{2,d}(\tilde{\epsilon}_x )}\sigma_{3,d}(\tilde{\epsilon}_x )\Big) \Second_{dd}(x)\Bigg]  \frac{1}{\epsilon}+O(1).\nonumber
\end{align}
 If $x \in M \setminus M_{\epsilon}$, then
 \begin{align}
\mathbf{T}^{tan}(x)=\,&J_{p,d} \frac{\nabla P(x)}{P(x)}+O(\epsilon)\nonumber\\
\mathbf{T}^{per}(x)=\,&\frac{P(x)}{2}  \bigg[\frac{|S^{d-1}|}{d(d+2)} \sum_{j=1}^{d} \Second_{jj}(x)\bigg] \frac{1}{\epsilon}+O(1).\nonumber
\end{align}
\end{proposition}

{The proof is postponed to Appendix \ref{proof of vector T}.  The above proposition says that when $x\in M_\epsilon$, both the tangent and normal components of $\mathbf{T}(x)$ are of order $\frac{1}{\epsilon}$, and the normal component depends on the extrinsic curvature of the manifold at $\iota(x)$. An interesting observation is that a construction of the {\em inward} normal vector field on $\iota(\partial M)$ is naturally encoded in the LLE algorithm. In particular, $\mathbf{T}^{tan}(x)$ on $\iota(\partial M)$ forms an inward normal vector field on $\iota(\partial M)$ with an order $O(1)$ perturbation.  Moreover, the magnitude of $ \mathbf{T}^{tan}(x)$ on $M_{\epsilon}$ only depends on the distance from $x$ to the boundary and it is independent of the p.d.f.. On the other hand, when $x \in M \setminus M_{\epsilon}$, $\mathbf{T}(x)$ is of order $\frac{1}{\epsilon}$ in the normal direction of $M \setminus M_{\epsilon}$ with an order $O(1)$ perturbation in the tangential direction. With the theorem developed in \cite{Wu_Wu:2017} for the augmented vector field away from the boundary, we have the full knowledge of the augmented vector field.}

At last, in Figure \ref{Figure:Tvec}, we provide a visualization of the augmented vector field in a 2-dim manifold parametrized by $(x,y,x^2-y^3)$, where $x^2+y^2\leq 1$. We sample the manifold in the following way. First, uniformly sample 
$20,000$ points independently on $[-1,1]\times[-1,1]$, and keep points with norm less and equal to $1$. The $i$-th point is then constructed by the parametrization. Clearly, the sampling is not uniform. The LLE matrix is constructed with the $\epsilon$-radius ball nearest neighbor search scheme with $\epsilon=0.2$.

\begin{figure}[ht]
\center
\includegraphics[width=.48\textwidth]{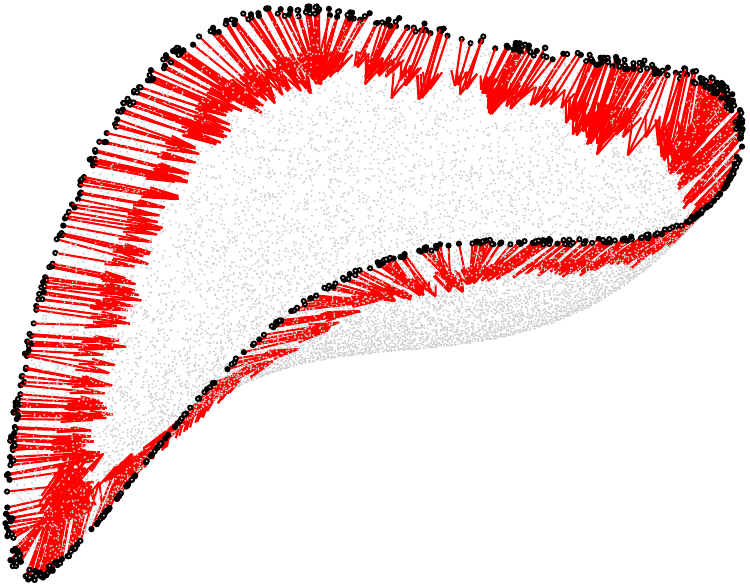}
\includegraphics[width=.48\textwidth]{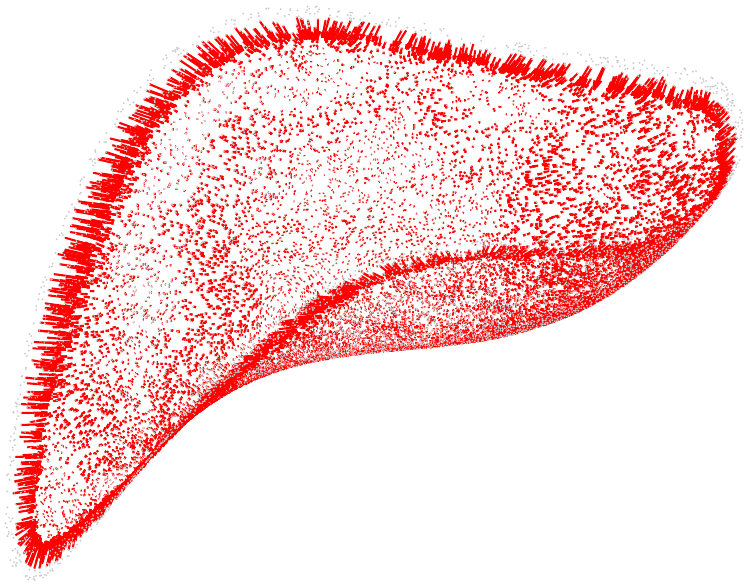}
\caption{The $\mathbf{T}$ vector field. The sampled point cloud is plotted in gray. Left: the black points indicates points satisfies $0.98\leq x^2+y^2\leq 1$, and the $\mathbf{T}$ on those points are marked in red. Right: the $\mathbf{T}$ on points with $x^2+y^2< 0.98$ are marked in red.}
\label{Figure:Tvec}
\end{figure}

{\begin{remark}
Another work that constructs a normal vector field on the boundary of a manifold is \cite{berry2017density}. Inspired by the kernel density estimation, the authors propose an inward normal vector field by using $\vec{F}(\iota(x))=\mathbb{E}[(\iota(X)-\iota(x))K(\frac{\|\iota(X)-\iota(x)\|_{\mathbb{R}^p}}{h})] \in \mathbb{R}^p$, where $K:[0, \infty) \rightarrow [0, \infty)$ has an exponential decay and $h$ is the bandwidth.  Since the construction originates from the kernel density estimation,  when $x \in \partial M$, the magnitude of the vector field depends on the p.d.f..
\end{remark}}

\subsection{Properties of the kernel on manifold with boundary}\label{Asymptotic Analysis 2}

With the above knowledge of the augmented vector field near the boundary, the behavior of the kernel near the boundary can be well quantified. 

\begin{proposition} \label{proposition 1}
{ Let $K_{\epsilon}(x,y)$ be the kernel function defined in \eqref{kernel formula}.  Fix $x \in M$, we summarized the properties of $K_{\epsilon}(x,y)$  as follows.}
\begin{enumerate}
\item
Suppose $x \not\in M_{\epsilon}$. When $\iota(y) \in B_{\epsilon}^{\mathbb{R}^p}(\iota(x))$, $K_{\epsilon}(x,y)=1-O(\epsilon)$. Otherwise $K_{\epsilon}(x,y)=0$. Hence, $K_{\epsilon}(x,y)\geq 0$, when $\epsilon$ is sufficiently small. The implied constant in $O(\epsilon)$ depends on the minimum and $C^1$ norm of $P$ and the maximum of second fundamental form of the manifold.
\item
If $x \in M_{\epsilon}$ and $\iota(y) \in B_{\epsilon}^{\mathbb{R}^p}(\iota(x))$, when $\epsilon$ is sufficiently small,  
\begin{equation}
K_{\epsilon}(x,y)=1- \frac{\sigma_{1,d}(\tilde{\epsilon}_x )u_d}{\sigma_{2,d}(\tilde{\epsilon}_x )\epsilon}+O(\epsilon),
\end{equation}
where the coordinate $u_d$ of $y$ is defined in Definition \ref{coordinates near boundary}. The implied constant in $O(\epsilon)$ depends on the minimum and $C^1$ norm of $P$ and the maximum of second fundamental form of the manifold. Otherwise $K_{\epsilon}(x,y)=0$.  Hence, we have
\[
\inf_{x,y} K_{\epsilon}(x,y)=1-\frac{|S^{d-2}|}{d-1}\frac{2d(d+2)}{(d+1)|S^{d-1}| }+O(\epsilon)<0
\]
when $\epsilon$ is sufficiently small, where $\frac{|S^{d-2}|}{d-1}$ is defined to be $1$ when $d=1$. 
\item
For any $x\in M$, we have
\begin{equation}\label{Equation Ker expection}
\mathbb{E}K_{\epsilon}(x,X)=C(x)\epsilon^d+O(\epsilon^{d+1}), 
\end{equation}
where $C(x)>C>0$, and $C$ is a constant depending only on $d$ and $P$. Hence, $ \mathbb{E}K_{\epsilon}(x,X)>0$ for all $x \in M$ when $\epsilon$ is sufficiently small. The implied constant in $O(\epsilon^{d+1})$ depends on the $C^1$ norm of $P$ and the maximum of second fundamental form of the manifold.
\end{enumerate}
\end{proposition}

This proposition provides several facts about LLE. First, the assumption of Proposition \ref{Proposition:ImaginaryControl} is satisfied when the manifold is boundary free, since the higher order error terms depend on various curvatures of $M$ and $M$ is smooth and compact. So, the eigenvalues of the LLE matrix in the {\em boundary-free} manifold setup have well controlled imaginary parts. However, when the boundary is not empty, we may lose this control.
Second, the kernel function behaves differently when $x$ is near the boundary and away from the boundary. When $x$ is away from the boundary, the kernel is non-negative. However, when $x$ is close to the boundary, then it is possible that $K_{\epsilon}(x,y)$ is negative. In particular, when $x \in \partial M$, $\iota(y) \in B_{\epsilon}^{\mathbb{R}^p}(\iota(x))$, the geodesic distance between $x$ and $y$ is $\epsilon+O(\epsilon^2)$ and the minimizing geodesic between $x$ and $y$ is perpendicular to $\partial M$, then $K_{\epsilon}(x,y)=1-\frac{|S^{d-2}|}{d-1}\frac{2d(d+2)}{(d+1)|S^{d-1}| }+O(\epsilon)<0$. Although it is possible that $K_{\epsilon}(x,y)$ is negative, $\mathbb{E}K_{\epsilon}(x,X)$ is always positive if $\epsilon$ is small enough. See Figure \ref{Figure:Kernel} for an illustration of the kernel associated with LLE, where the manifold, the sampling scheme, and the LLE matrix are the same as that in Figure \ref{Figure:Tvec}, expect $\epsilon=0.1$.

\begin{figure}[ht]
\center
\includegraphics[width=.98\textwidth]{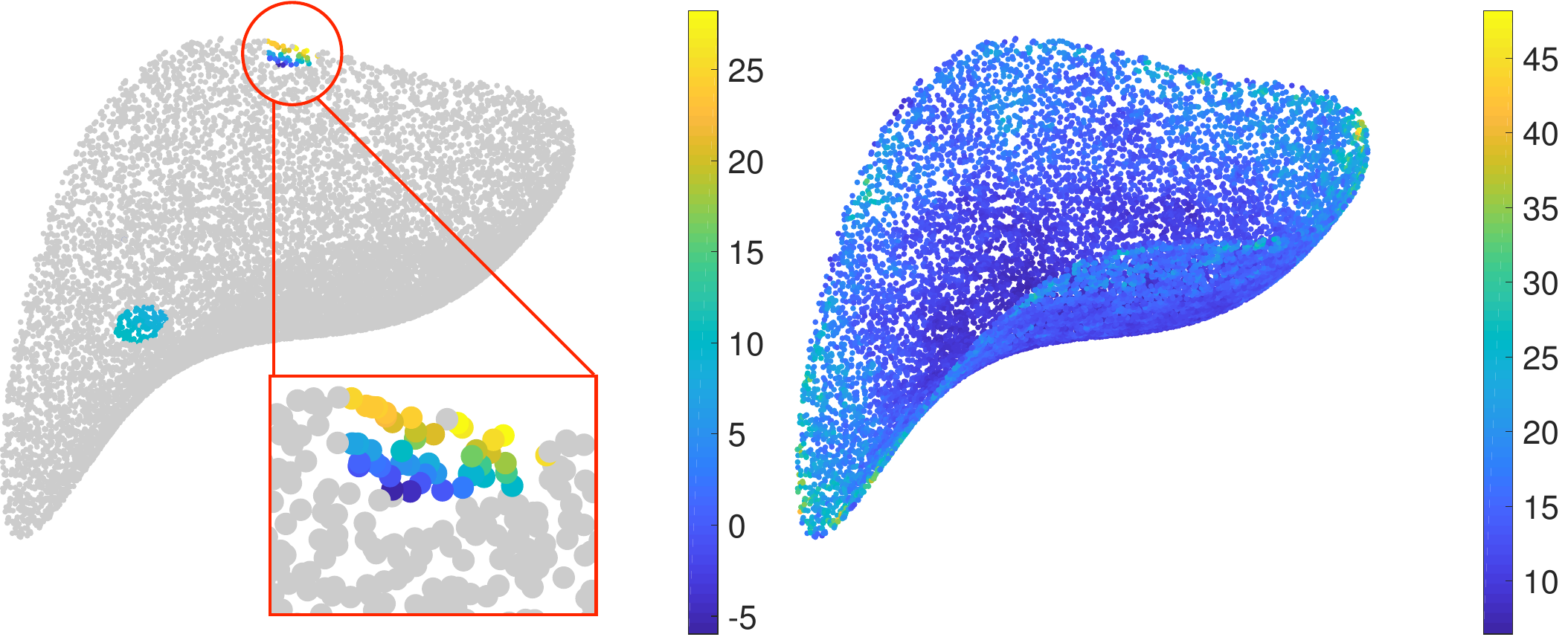}
\caption{The kernel function associated with LLE. The sampled point cloud is plotted in gray. Left: the kernel function $K_\epsilon$, where $\epsilon = 0.1$, on two points, one is close to the boundary (indicated by the red circle, with the zoomed in enhanced visualization), and one is away from the boundary. It is clear that the kernel close to the boundary changes sign, while the kernel away from the boundary is positive. Right: the $\mathbb{E}K_\epsilon(x,X)$. It is clear that the expectations of the kernel at all points are positive.}
\label{Figure:Kernel}
\end{figure}

\subsection{Variance analysis of LLE on manifold with boundary}\label{Asymptotic Analysis 3}

Define the integral operator from $C(M)$ to $C(M)$:
\begin{align}\label{operator Q epsilon}
Q_\epsilon f(x) :=&\, \frac{\mathbb{E}[K_{\epsilon}(x,X)f(X)]}{\mathbb{E}K_{\epsilon}(x,X)}- f(x)\,,
\end{align}
where $f \in C(M)$.
We now show that when the boundary is not empty, the LLE matrix $W$ converges to the integral operator $Q_\epsilon$ when $n\to \infty$. The proof of the theorem is postponed to Appendix \ref{proof of theorem t0}.

\begin{theorem}[Variance analysis] \label{Theorem:t0} Suppose $f \in C^2(M)$. Suppose $\epsilon=\epsilon(n)$ so that $\frac{\sqrt{\log(n)}}{n^{1/2}\epsilon^{d/2+1}}\to 0$ and $\epsilon\to 0$ as $n\to \infty$. We have with probability greater than $1-n^{-2}$ that for all $k=1,\ldots,n$, 
\begin{align}
& \sum_{j=1}^{N_k} y_k(j)= \frac{\mathbb{E}K_{\epsilon}(x_k,X)}{\epsilon^{d+3}}+O\Big(\frac{\sqrt{\log (n)}}{n^{1/2}\epsilon^{d/2+3}}\Big)\label{Nonnegative Kernel} \\
& \sum_{j=1}^n [W-I_{n \times n}]_{kj} f(x_j) = Q_\epsilon f(x_k)+O\Big(\frac{\sqrt{\log (n)}}{n^{1/2}\epsilon^{d/2-1}}\Big)\,,
\end{align}
where $y_k$ is defined in \eqref{Section2:wk}.
The implied constants in the error terms depend on $C^2$ norm of $f$, $C^1$ norm of $P$ and the $L^\infty$ norm of $\max_{i,j=1,\ldots,d}\|\Second_{ij}(x)\|$. 
\end{theorem}

Note that the order of the variance does not depend on the location of $x_j$. By combining \eqref{Nonnegative Kernel} and \eqref{Equation Ker expection} in Proposition \ref{proposition 1}, we know that if $n$ is sufficiently large, the sum of all components of $y_k$ is positive. This result restates the fact that $w_k$ defined in \eqref{Expansion:LLEweightedKernel} does not blow up.

\subsection{Bias analysis of LLE on manifold with boundary and the main result}\label{Asymptotic Analysis 4}
{In this subsection, we study the integral operator $ Q_\epsilon$ by relating it to the following differential operator.}

\begin{definition}\label{Definition D operator}
Fix $\epsilon>0$. Define a differential operator on $C^2(M)$ as
\begin{align}
\mathcal{D}_\epsilon f(x)=\phi_{1} (\tilde{\epsilon}_x ) \sum_{i=1}^{d-1}\partial_{ii}^2f(x) +\phi_{2} (\tilde{\epsilon}_x ) \partial_{dd}^2f(x)+V(x) \partial_d f(x)\,,
\end{align}
where $\phi_{1}$ and $\phi_{2}$ are functions defined on $[0,\infty)$ by
\begin{align}
\phi_{1} (t) =&\,\frac{1}{2} \frac{\sigma_{2,d}(t) \sigma_{2}(t)-\sigma_{3}(t) \sigma_{1,d}(t)}{\sigma_{2,d}(t) \sigma_{0}(t) -\sigma_{1,d}^2(t)},\\
\phi_{2} (t) =&\,\frac{1}{2} \frac{\sigma^2_{2,d}(t)-\sigma_{3,d}(t) \sigma_{1,d}(t)}{\sigma_{2,d}(t) \sigma_{0}(t) -\sigma_{1,d}^2(t)},\label{Equation phi2}
\end{align}
and $V$ is a function on $M$ defined by
\begin{align}
V(x) =\frac{\sigma_{1,d}(\tilde{\epsilon}_x )}{P(x)\big(\sigma_{2,d}(\tilde{\epsilon}_x ) \sigma_{0}(\tilde{\epsilon}_x )-\sigma_{1,d}^2(\tilde{\epsilon}_x )\big)}\,.
\end{align}
\end{definition}

Next, we take a closer look at coefficients of $\mathcal{D}_\epsilon$.

\begin{proposition}\label{properties of the coeffecients}
Fix $\epsilon>0$. We have the following properties of coefficients of the differential operator $\mathcal{D}_\epsilon$.
\begin{enumerate}
\item
$\phi_{1} (t)>0 $. When $t \geq \epsilon$,  
\begin{equation}
\phi_{1} (t)=\frac{1}{2(d+2)}.
\end{equation}
Moreover, $\phi_{1} (t)$ is differentiable of all orders at all $t>0$ except at $t=\epsilon$, where it is at least first order differentiable.
\item 
$\phi_{2} (0)<0$.
 If $t \geq \epsilon$, then 
\begin{equation}
\phi_{2} (t)=\frac{1}{2(d+2)}.
\end{equation}
Moreover, $\phi_{2} (t)$ is differentiable of all orders at all $t>0$ except at $t=\epsilon$, where it is at least first order differentiable. Hence, there is a set $\mathcal{S} \subset M_{\epsilon}$ diffeomorphic to $\partial M$ and $\phi_{2} (\tilde{\epsilon}_x )$ vanishes on $\mathcal{S}$.  
Denote the geodesic distance from $x \in \mathcal{S}$ to $\partial M$ as $t^{*}(x)$. $t^*(x)$ depends only on $\epsilon$ and $d$. In fact, $\delta_1 \epsilon <t^*(x)<\delta_2 \epsilon$, where
\begin{align}
& \delta_1=\Bigg(1-\bigg[\frac{1+\frac{(d^2-1)|S^{d-1}|}{2d(d+2)|S^{d-2}|}}{1+\sqrt{\frac{2}{d+3}}}\bigg]^{\frac{2}{d+1}}\Bigg)^{\frac{1}{2}}, \\
& \delta_2=\Bigg(1-\bigg[\frac{(d^2-1)|S^{d-1}|}{4d(d+2)|S^{d-2}|}+\frac{1}{d+3}\bigg]^{\frac{2}{d+1}}\Bigg)^{\frac{1}{2}}<1
\end{align}
and $\delta_2 \rightarrow 0$ as $d \rightarrow \infty$.
\item
 $V(x)\leq 0$. Moreover, $V(x)=O(\epsilon^2)$ is differentiable of all orders at all $x$ except when $\tilde{\epsilon}_x =\epsilon$, where it is at least differentiable of the first order. If $x\in M_\epsilon$ satisfies $\tilde{\epsilon}_x  \geq \epsilon$, in other words, $x\in M\backslash M_\epsilon$, then $V(x)=0$. In particular, if $P(x)$ is constant, then $V(x)$ is an increasing function of  $\tilde{\epsilon}_x$.
\end{enumerate}
\end{proposition}

Denote $M_w$ to be the interior subset of the region between $\mathcal{S}$ and $\partial M$ on $M$. 
Denote $M_{e}$ to be the interior subset of $M\backslash M_w$ on $M$. 
Clearly, $M_w$ is a strict subset of $M_\epsilon$. According to Proposition \ref{properties of the coeffecients}, $\mathcal{D}_\epsilon$ is of hyperbolic type over $M_w$, of elliptic type over $M_e$, and degenerate over $\mathcal{S}$. We thus call $M_w$ the {\em wave region}, $M_e$ the {\em elliptic region}, and $\mathcal{S}$ the degenerate region. We conclude that the operator $\mathcal{D}_\epsilon$ is a mixed-type differential operator with degeneracy.

\begin{remark}
In fact, $t^*$ is the solution of the following nonlinear equation of $t$:
\begin{align}
&\Big(\frac{|S^{d-1}|}{2d(d+2)}+\frac{|S^{d-2}|}{d-1}\int_{0}^{\frac{t}{\epsilon}} (1-x^2)^{\frac{d-1}{2}}x^2 dx \Big)^2\nonumber\\
=&\,\frac{|S^{d-2}|^2}{(d^2-1)^2(d+3)}\Big[2+(d+1)\Big(\frac{t}{\epsilon}\Big)^2\Big]\Big[1-\Big(\frac{t}{\epsilon}\Big)^2\Big]^{d+1},
\end{align}
where $t>0$
\end{remark}

We have the following theorem describing how $Q_\epsilon$ is related to $\mathcal{D}_\epsilon$ when $\epsilon$ is sufficiently small. The proof is postponed to Appendix  \ref{proof of t1 and t2}.

\begin{theorem}[Bias analysis]\label{Theorem:t1} 
Suppose $f \in C^3(M)$ and $P\in C^2(M)$. We have
\begin{equation}
Q_\epsilon f(x)=\mathcal{D}_\epsilon f(x)\epsilon^2+O(\epsilon^3)\,.
\end{equation}
\end{theorem}
By combining the bias and variance analyses, we have the following pointwise convergence result. 
\begin{theorem}\label{Theorem:final} 
Suppose $f \in C^3(M)$ and  $P\in C^2(M)$. Suppose $\epsilon=\epsilon(n)$ so that $\frac{\sqrt{\log(n)}}{n^{1/2}\epsilon^{d/2+1}}\to 0$ and $\epsilon\to 0$ as $n\to \infty$. We have with probability greater than $1-n^{-2}$ that for all $k=1,\ldots,n$, 
\begin{align}
\sum_{j=1}^n [W-I_{n \times n}]_{kj} f(x_j) = \mathcal{D}_\epsilon f(x)\epsilon^2+O(\epsilon^3)+O\Big(\frac{\sqrt{\log (n)}}{n^{1/2}\epsilon^{d/2-1}}\Big)\,,
\end{align}
where the implied constants in the error terms depend on the $C^2$ norm of $f$, the $C^1$ norm of $P$, and the $L^\infty$ norm of $\max_{i,j=1,\ldots,d}\|\Second_{ij}(x)\|$. 
\end{theorem}

Note that when $M$ is a manifold without boundary, then $\mathcal{D}_\epsilon=\frac{1}{2(d+2)}\Delta$.  Hence, the above theorem is consistent with the result in  \cite{Wu_Wu:2017} when $M$ has no boundary.  However, the regularizer in \cite{Wu_Wu:2017} is chosen as $c=n \epsilon^{d+\rho}$, where $\rho$ may be different than $3$.  The main result there is presented in different cases capturing the interaction between $\rho$ and different local covariance matrix structures. In contrast, the above theorem is much simpler as we only focus on the case when $\rho=3$.  

\subsection{Relationship between $W-I$ and $(W-I)^\top(W-I)$}\label{sec (W-I)T(W-I) compr}

{
In this paper and \cite{Wu_Wu:2017}, we focus on studying the asymptotic behavior of $W-I$. However, in the original LLE algorithm, it is the eigenvector of the matrix $(W-I)^\top(W-I)$ that is used to reduce the dimension of $\mathcal{X}$. We shall clarify the relationship between $W-I$ and $(W-I)^\top(W-I)$ from two aspects -- spectral geometry and linear algebra.

Recall that when a manifold is compact without boundary, based on the spectral embedding results \cite{berard1994embedding,bates2014embedding,portegies2016embeddings}, the eigenfunctions of $-\Delta$ can be applied to construct an embedding of the manifold in a Euclidean space, and the embedding is almost isometric if both eigenfunctions and eigenvalues are properly used. Therefore, if we could obtain eigenpairs of $\Delta$ from the database $\mathcal{X}$ sampled from a manifold, we could recover the manifold via this spectral embedding.
However, it is not clear what the eigenvectors of $(W-I)^\top(W-I)$ or $W-I$ mean directly from the algorithm. An immediate approach to answering  this question is via the pointwise convergence.  In \cite{Wu_Wu:2017}, if $M$ is a compact manifold without boundary, $W-I$ pointswisely converges to the operator $\frac{\epsilon^2}{2(d+2)}\Delta$.  In Theorem \ref{Theorem:final} of the current work, we show that the same result holds over the region $M \setminus M_{\epsilon}$ when $\partial M\neq \emptyset$, and near the boundary the asymptotic behavior is complicated with degeneracy. 
Hence, one may guess that $(W-I)^\top(W-I)$ will pointwisely converge to $\frac{\epsilon^4}{4(d+2)^2}\Delta^2$ over the data points in the region $M \setminus M_{\epsilon}$. However, it is in general {\em not} true, particularly when the sampling is nonuniform. It is because the nonsymmetry of $W-I$ plays an important role in eliminating the impact of nonuniform distribution of the point cloud on $M$ \cite{Wu_Wu:2017}, and we lose such property if we consider $(W-I)^\top(W-I)$.} {Consider the following analysis of $(W-I)^\top(W-I)$ in a simple manifold for an illustration. In contrast to the operator $Q_\epsilon$, we define a new integral operator from $C(M)$ to $C(M)$:
\begin{align}
\bar{Q}_\epsilon f(x) :=&\, \frac{\mathbb{E}[K_{\epsilon}(X,x)f(X)]}{\mathbb{E}K_{\epsilon}(X,x)}- f(x)\,,
\end{align}
where $f \in C(M)$.
By the law of large number, we would have
$$\sum_{j=1}^n [W-I_{n \times n}]^\top_{kj} f(x_j) \rightarrow \bar{Q}_\epsilon f(x_k)$$ 
and 
$$
\sum_{i=1}^n [W-I_{n \times n}]^\top_{ki} \sum_{j=1}^n [W-I_{n \times n}]_{ij} f(x_j) \rightarrow (\bar{Q}_\epsilon (Q_\epsilon f))(x_k)
$$  
as $n \rightarrow \infty$. Now,
suppose $\iota(M)=[-1,1] \subset \mathbb{R}$ and let $\mathcal{X}=\{\iota(x_i)\}_{i=1}^n$, where $\{x_1, x_2, \cdots, x_n\}$ are i.i.d. sampled following a p.d.f. $P \in C^2(M)$. Then, for any $\epsilon$ small enough,  suppose $\iota(x_k) \in [-1+\epsilon, 1-\epsilon]$, for any $f \in C^5(M)$, 
we have
\begin{align}
(\bar{Q}_\epsilon (Q_\epsilon f))(x_k)=\frac{1}{9}\Big(\frac{1}{4}f ''''(x_k)+ \frac{f '''(x_k) P'(x_k)}{P(x_k)} \Big) +O(\epsilon)\,,\label{theorem (W-I)T(W-I)}
\end{align}
 where $\vec{f}=[f(x_1),\ldots,f(x_n)]^\top$.
%
The detailed calculation of \eqref{theorem (W-I)T(W-I)}} is postponed to Appendix \ref{proof (W-I)T(W-I)}. {In this result, although the manifold is flat and there is no extrinsic geometric information involved, asymptotically $(W-I)^\top(W-I)$ involves not only the desired $\Delta^2$ but also the sampling distribution. This says that even if we do not consider the boundary, the asymptotic differential operator is more complicated than the bi-Laplacian $\Delta^2$. A similar result can be derived when $M$ is a general manifold without boundary, but we omit details here. 

Based on the above discussion, we could reasonably conjecture that the eigenvectors of $W-I$ approximate the eigenfunctions of $\Delta$ when $\partial M=\emptyset$ and the eigenfunctions of more complicated second order differential operator with degeneracy when $\partial M\neq \emptyset$, and the eigenvectors of $(W-I)^\top(W-I)$ approximate the eigenfunctions of more complicated fourth order differential operator. To prove these conjectures, we need to establish the spectral convergence results, which is out of the scope of this paper.
Note that previous work on spectral convergence of graph Laplacian \cite{dunson2021spectral, calder2020lipschitz, wormell2020spectral} mainly focuses on symmetric kernel matrices, except the work discussing the kNN kernel construction \cite{calder2020lipschitz}. However, these kernels are a priori assigned, so their approaches cannot be directly applied to study $W-I$. Specifically, the LLE matrix is not only nonsymmetric but also determined by the dataset. We need different analysis tools to establish the spectral convergence. On the other hand, a complete understanding of the original LLE is certainly via understanding $(W-I)^\top(W-I)$. 
We shall mention that even for bi-Laplacian to which $(W-I)^\top(W-I)$ pointwisely converges under special conditions, it is still challenging to derive the spectral convergence. Note that it is still an active research field to study bi-Laplacian and its spectral behavior \cite{chang1999regularity,cuccu2009maximization}. To sum up, with the help of pointwise convergence results, we could conjecture the spectral behavior of $W-I$ and $(W-I)^\top(W-I)$, and their behaviors are different in general.

From the linear algebra perspective, there are several interesting relationships between $(W-I)^\top(W-I)$ and $W-I$.  First, $W-I$ is a sparse matrix and in general sparser than $(W-I)^\top(W-I)$. Since $(W-I)^\top(W-I)$ is symmetric, its eigendecomposition always exists, while $W-I$ is not always diagonalizable.
Second, eigenvalues of $(W-I)^\top(W-I)$ are the square of the singular values of $W-I$ and the eigenvectors of $(W-I)^\top(W-I)$ are the same as the right singular vectors of the $W-I$. 
While $W-I$ is in general not diagonalizable, based on Proposition \ref{proposition 1} under the manifold setup, the conditions in Proposition \ref{Proposition:ImaginaryControl} are satisfied, which says that when the sample size of the data is large enough, $W-I$ is close to the symmetric matrix $\frac{W+W^\top}{2}-I$ and the imaginary parts of the eigenvalues of $W-I$ are small. 
Note that even if $W-I$ is diagonalizable, the right singular vectors of $W-I$ are different from the right eigenvectors of $W-I$. It echos what we discuss above --- under the boundary-free manifold setup, we conjecture that the eigenvectors of $W-I$ approximate the eigenfunctions of $\Delta$, while the eigenvectors of  $(W-I)^\top(W-I)$, and hence the singular vectors of $W-I$, approximate the eigenfunctions of a fourth order differential operator that involves the nonuniform sampling information. 
Third, we shall mention that numerically we consistently found that under the manifold setup, when $n$ is sufficiently large, the leading eigenvectors of $W-I$ recover the corresponding eigenfunctions of $\Delta$. Note that a theoretical justification of this numerical finding is part of the spectral convergence conjecture listed above. Therefore, if we consider $W-I$ for the dimension reduction purpose, we propose to use the real parts of the top eigenvectors of $W-I$ corresponding to the leading eigenvalues listed in the decreasing order of their real parts to define the embedding. We provide a numerical comparison of the embedding by $W-I$ and $(W-I)^\top(W-I)$ in Figure \ref{Figure:comparison}, where we nonuniformly sample 3513 points from the unit disk on $\mathbb{R}^2$ (shown on the top left subfigure superimposed with the radius as the color), construct $W$ with the $\epsilon=0.04$ radius ball and the regularizer $n \epsilon^5$, and embed the data using the top 2 non-trivial eigenvectors of $W-I$ (shown on the top middle subfigure) and $(W-I)^\top (W-I)$ (shown on the top right subfigure). Note that in the top middle and right subfigures, the original radius of each point is superimposed as the color, which indicates how the embedding behaves. The top three nontrivial eigenvectors of $W-I$ are shown in the middle panel, and the top three nontrivial eigenvectors of $(W-I)^\top(W-I)$ are shown in the bottom panel. In this example, we see that the embeddings by $W-I$ and $(W-I)^\top(W-I)$ are different. Based on the analysis in the current work, it is not surprising that the embedding by $W-I$ is impacted by the boundary, but it is interesting to see that $(W-I)^\top(W-I)$ is less impacted by the boundary. The nonuniform sampling also plays a role here. While we do not show it here, we found that when the sampling from the unit disk is uniform, the embeddings by $W-I$ and $(W-I)^\top(W-I)$ are similar, which suggests the interaction between the sampling and boundary. The above interesting findings warrant further study of the behavior of $(W-I)^\top(W-I)$ from various aspects to fully understand how LLE functions.

\begin{figure}[htb!]
\center
\includegraphics[trim=0 25 0 25, width=0.92\textwidth]{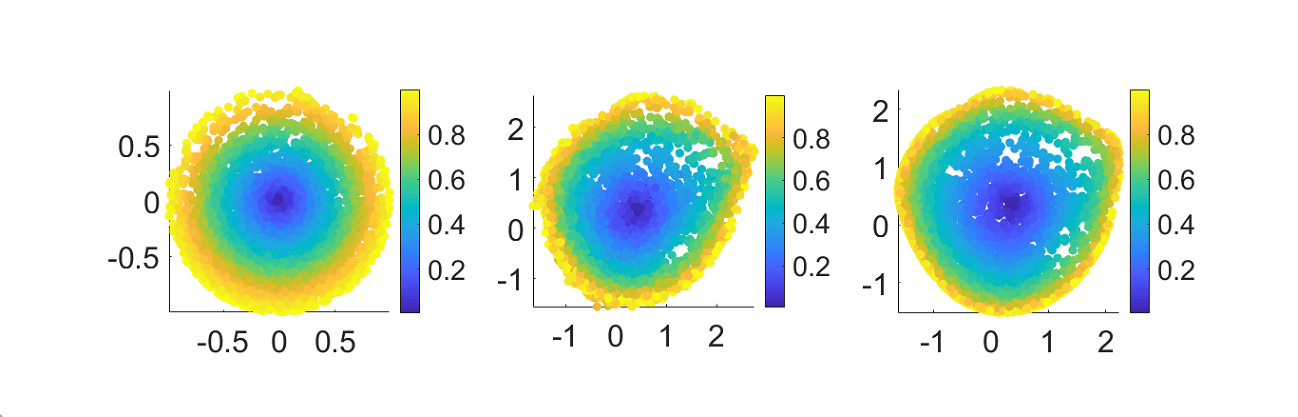}
\includegraphics[trim=0 25 0 25, width=0.92\textwidth]{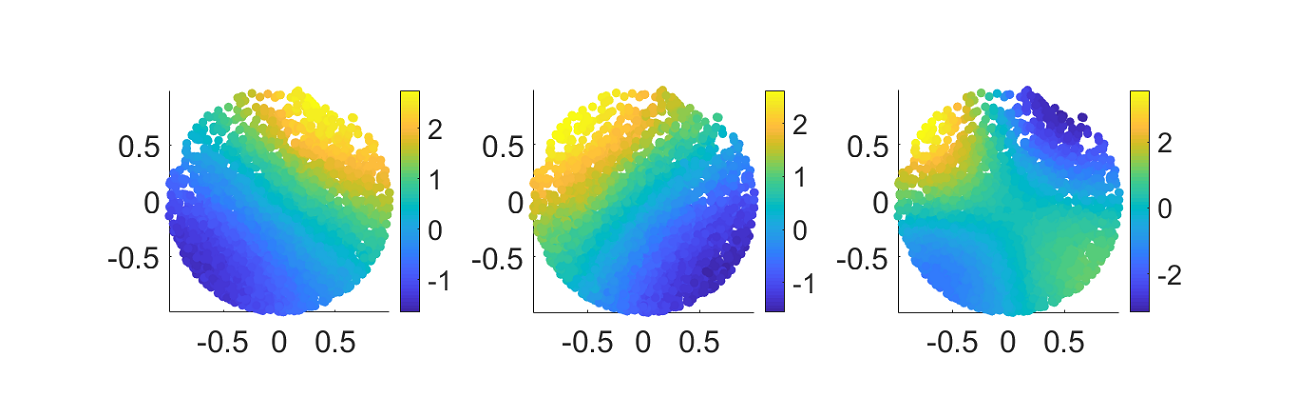}
\includegraphics[trim=0 25 0 25, width=0.92\textwidth]{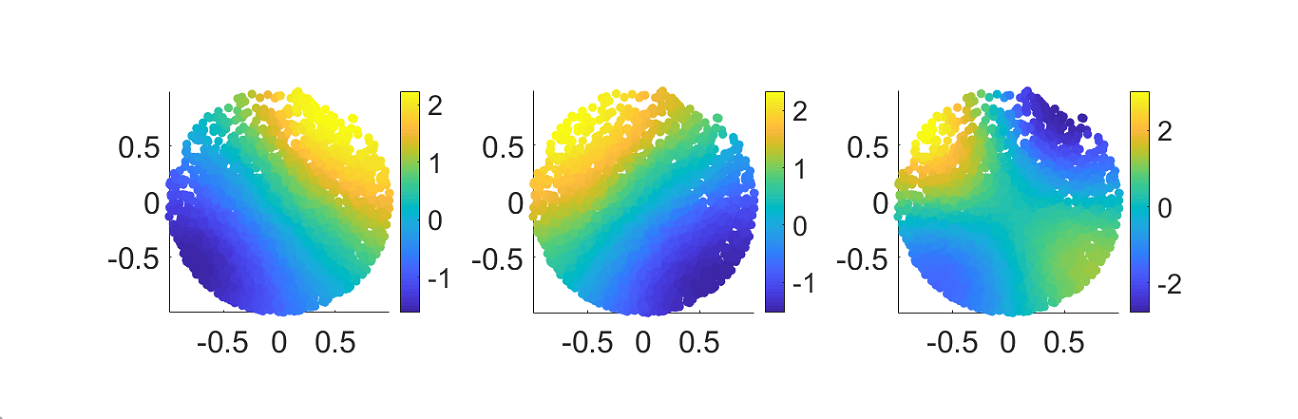}
\caption{Top left subfigure: the original dataset uniformly sampled from the unit disk superimposed with the radius as the color. Top middle (right respectively) subfigure: the embedding by the top two nontrivial eigenvectors of $W-I$ ($(W-I)^\top(W-I)$ respectively), where the radius of the original radius of each point is superimposed as the color. The top three nontrivial eigenvectors of $W-I$ are shown in the middle panel, and the top three nontrivial eigenvectors of $(W-I)^\top(W-I)$ are shown in the bottom panel.}
\label{Figure:comparison}
\end{figure}

}

\section{Exploration of LLE on 1-dim manifold with boundary}\label{OpenQuestions}
{In this section, we further explore the LLE matrix $W$. In Theorem \ref{Theorem:final}, we show that the  matrix $W-I$ converges pointwisely to the differential operator $\mathcal{D}_\epsilon$. Then, we illustrate how the differential operator $\mathcal{D}_\epsilon$ looks like in the $1$ dimensional case.}

\begin{corollary}\label{Corollary one dim}
Let $M$ be a regular smooth curve in $\mathbb{R}^p$. Let $\gamma(t):[0,a] \rightarrow \mathbb{R}^p$ be the arclength parametrization. Let $P(t)$ be the probability density function. Then, we have, for $f\in C^2(M)$,
\begin{equation*}
\mathcal{D}_\epsilon f(t) = \left\{
\begin{array}{lll} 
 -\frac{1}{12}(1-4(\frac{t}{\epsilon})+(\frac{t}{\epsilon})^2)f''(t)+\frac{6\epsilon^2(\epsilon-t)}{P(t)(\epsilon+t)^3}f'(t) & \mbox{if $t \in [0, \epsilon]$};\\ \frac{1}{6} f''(t) & \mbox{if $t \in [\epsilon, a-\epsilon]$};\\ 
 -\frac{1}{12}(1-4(\frac{a-t}{\epsilon})+(\frac{a-t}{\epsilon})^2)f''(t)-\frac{6\epsilon^2(\epsilon+t-a)}{P(t)(\epsilon+a-t)^3}f'(t) & \mbox{if $t \in [a-\epsilon, a]$}. \end{array} \right.
\end{equation*}
Specifically, $\mathcal{D}_\epsilon f(t)$ degenerates to $\frac{2\sqrt{3}}{P(t)}f'(t)$ at $t=(2-\sqrt{3})\epsilon$ and $t=a-(2-\sqrt{3})\epsilon$.
\end{corollary}
This corollary comes from a direct expansion of the formula in Definition \ref{Definition D operator}. Note that $e_d$ is in the outward normal direction by definition. Therefore, $\partial_d f(t)=-f'(t)$, when $t \in [0, \epsilon]$. And $\partial_d f(t)=f'(t)$, when $t \in [a-\epsilon, a]$.
To study the spectral property of $\mathcal{D}_\epsilon$, it is natural to consider converting $\mathcal{D}_\epsilon$ into the Sturm-Liouville (SL) form by the integrating factor. However, due to the degeneracy of $a_\epsilon$, several technical details need to be taken care of. 
Here we provide a summary of known facts about the SL theory \cite[Chapter V]{naimark1967linear}. 

Fix $a>0$. The SL problem on $(0,a)$ is finding a complex function $f(x)$ defined on $(0,a)$ that solves
\begin{equation}\label{SL problem}
-(p(x)f(x)')'+q(x)f(x)=\lambda w(x) f(x),
\end{equation}
where $p(x)$, $q(x)$ and $w(x)$ are measurable real functions on $(a,b)$ and $\lambda$ is a complex number. 
The SL problem is called {\em regular} if $1/p$, $q$, and $w$ are all functions in $L^1(a,b)$; otherwise, it is called {\em singular}.
A complex function $f(x)$ is a {\em solution} of the SL problem (\ref{SL problem}), if $f^{[0]}(x)$ and $f^{[1]}(x)$ exist,  
where $f^{[0]}(x):=f(x)$ (respectively $f^{[1]}(x):=p(x)f'(x)$) is the zero (respectively first) order quasi-derivative of $f$,
and are {absolutely continuous} on any compact subinterval of $(a,b)$.
It is worth noting that in general $f'(x)$ may not be {absolutely continuous}. 

It is stated in \cite[Chapter V]{naimark1967linear} that for $x_0 \in (0,a)$ and complex numbers $c_0$ and $c_1$, there exists a unique solution to the regular SL problem with $f^{[0]}(x_0)=c_0$  and $f^{[1]}(x_0)=c_1$.
As an eigenvalue problem, by \cite{atkinson176141discrete}, given boundary conditions $A_1f^{[0]}(0)+A_2f^{[1]}(0)=0$ and $B_1f^{[0]}(a)+B_2f^{[1]}(a)=0$ with $A_1^2+A_2^2>0$ and $B_1^2+B_2^2>0$, if the SL problem is regular and $p>0$ and $w>0$ on $(0,a)$, then the eigenvalues are discrete and bounded from below; that is, we have eigenvalues  $-\infty<\lambda_0<\lambda_1<\lambda_2<\ldots$ so that $\lambda_n \rightarrow \infty$ as $n \rightarrow \infty$. Moreover, if $f_n$ is the corresponding eigenvalues of $\lambda_n$, then $f_n$ has exactly $n$ zeros in $(0,a)$.

Now we come back to the challenge. Suppose $P(t)=1/a$; that is, the sampling is uniform. By Corollary \ref{Corollary one dim}, when $\epsilon$ is sufficiently small, the second order ordinary differential equation
\begin{equation} \label{1-d operator}
\mathcal{D}_\epsilon f(t)=A_\epsilon(t)f''(t)+B_\epsilon(t)f'(t)\,
\end{equation}
dominates.
Note that $A_\epsilon(t)>0$ on the elliptic region $((2-\sqrt{3})\epsilon, a-(2-\sqrt{3})\epsilon)$, and $A_\epsilon(t)<0$ on the wave region $[0,(2-\sqrt{3})\epsilon) \cup (a-(2-\sqrt{3})\epsilon,a]$, while $B_\epsilon(t)\geq 0$ on $[0,a]$.

To convert $\mathcal{D}_\epsilon$ into the SL form, we define two more functions over $t \in [0,\epsilon]$. First,
\[
g(t):=\big|t-(2-\sqrt{3})\epsilon\big|^{(4+2\sqrt{3})a\epsilon}\big|t-(2+\sqrt{3})\epsilon\big|^{(4-2\sqrt{3})a\epsilon}(t+\epsilon)^{8a\epsilon}e^{[\frac{12a\epsilon^3}{(\epsilon+t)^2}+\frac{12a\epsilon^2}{\epsilon+t}]}, 
\]
over $t \in [0,\epsilon]$. Clearly, $g(t)>0$. Moreover, $g$ is continuous and is smooth except at $t=(2-\sqrt{3})\epsilon$.
Second, 
\[
h(t):=\big|t-(2-\sqrt{3})\epsilon\big|^{(4+2\sqrt{3})a\epsilon-1}\big|t-(2+\sqrt{3})\epsilon\big|^{(4-2\sqrt{3})a\epsilon-1}(t+\epsilon)^{8a\epsilon}e^{[\frac{12a\epsilon^3}{(\epsilon+t)^2}+\frac{12a\epsilon^2}{\epsilon+t}]}, 
\]
over $t \in [0,\epsilon]$. 
By a direct check, we know that $h(t)>0$ on $[0,(2-\sqrt{3})\epsilon)\cup ((2-\sqrt{3})\epsilon,\epsilon]$ and {$h(t)\to \infty$ when $t\to (2-\sqrt{3})\epsilon$ since $(4+2\sqrt{3})a\epsilon-1<0$}.
With $g$ and $h$, define 
\begin{align}\label{p(t)}
p(t) := \left\{
\begin{array}{lll} 
 -g(t) & \mbox{if $t \in [0, (2-\sqrt{3})\epsilon]$};\\ g(t) & \mbox{if $t \in [(2-\sqrt{3})\epsilon, \epsilon]$};\\ 
g(\epsilon) & \mbox{if $t \in [\epsilon, a-\epsilon]$}; \\ g(a-t)  & \mbox{if $t \in [a-\epsilon,a-(2-\sqrt{3})\epsilon]$};\\ -g(a-t)  & \mbox{if $t \in [a-(2-\sqrt{3})\epsilon, a]$} \end{array} \right.
\end{align}
on $[0,a]$ and
\begin{align}
w(t) := \left\{
\begin{array}{lll} 
h(t) & \mbox{if $t \in [0, \epsilon]$};\\
h(\epsilon) & \mbox{if $t \in [\epsilon, a-\epsilon]$}; \\h(a-t)  & \mbox{if $t \in [a-\epsilon,a]$}. \end{array} \right.
\end{align}
on $[0,a]$. 
We have the following proposition summarizing the behavior of $p$ and $w$.
\begin{proposition}
Suppose $\epsilon$ is sufficiently small. The defined function $p$ satisfies the following properties.
\begin{enumerate}
\item
$p(t)>0$ on the elliptic region $((2-\sqrt{3})\epsilon, a-(2-\sqrt{3})\epsilon)$, $p(t)<0$ on the wave region $[0,(2-\sqrt{3})\epsilon) \cup (a-(2-\sqrt{3})\epsilon,a]$, and $p(t)=0$ when $t=(2-\sqrt{3})\epsilon$ or $t=a-(2-\sqrt{3})\epsilon$. 
\item
$p(t)$ is $C^1$ on $[0,a]$ except at $t=(2-\sqrt{3})\epsilon$ and $t=a-(2-\sqrt{3})\epsilon$. In particular, $p'(t) \rightarrow \infty$ as $t \rightarrow (2-\sqrt{3})\epsilon$ from right or $t \rightarrow a-(2-\sqrt{3})\epsilon$ from left; $p'(t) \rightarrow -\infty$ as $t \rightarrow (2-\sqrt{3})\epsilon$ from left or $t \rightarrow a-(2-\sqrt{3})\epsilon$ from right. 
\item $p(t)$ is {absolutely continuous}. 
\item 
$1/p \in L^1$ on $[0,a]$.
\end{enumerate}
The defined function $w$ satisfies the following properties.
\begin{enumerate}
\item
$w(t)$ is $C^1$ on $[0,a]$ except at $t=(2-\sqrt{3})\epsilon$ and $t=a-(2-\sqrt{3})\epsilon$.
\item $w(t)  \rightarrow \infty$ as $t \rightarrow (2-\sqrt{3})\epsilon$ or $t \rightarrow a-(2-\sqrt{3})\epsilon$.
\item 
$w \in L^1$ on $[0,a]$. 
\end{enumerate}
The defined functions $p$ and $w$ are related to $A_\epsilon$ and $B_\epsilon$ and satisfy the following properties.
\begin{enumerate}
\item
$\frac{p(t)}{w(t)}=A_{\epsilon}(t)$ and $\frac{p'(t)}{w(t)}=B_{\epsilon}(t)$, when $t \neq (2-\sqrt{3})\epsilon$ and $t\neq a-(2-\sqrt{3})\epsilon$.
\item
$\frac{p(t)}{w(t)}\rightarrow A_{\epsilon}(t)$ and $\frac{p'(t)}{w(t)} \rightarrow B_{\epsilon}(t)$ when $t \rightarrow (2-\sqrt{3})\epsilon$ or $t \rightarrow a-(2-\sqrt{3})\epsilon$. 
\end{enumerate}
\end{proposition}

With this proposition, we conclude that except at $t=(2-\sqrt{3})\epsilon$ and $t=a-(2-\sqrt{3})\epsilon$, the following relationship holds:
\begin{align}
A_\epsilon(t)f''(t)+B_\epsilon(t)f'(t)=\frac{p(t)}{w(t)}f''(t)+\frac{p'(t)}{w(t)}f'(t)=\frac{(p(t)f'(t))'}{w(t)}\,.
\end{align}
Also, the corresponding eigenvalue problem
\begin{align}
-(p(t)f'(t))'=\lambda w(t)f(t),
\end{align} 
is regular when $\epsilon$ is small enough. If $f$ is a solution to the above problem, then $f$ is {absolutely continuous}, and hence it is differentiable almost everywhere. Moreover, $p(t)f'(t)$ is {absolutely continuous} and $p(t)$ is differentiable and nonzero except at $t=(2-\sqrt{3})\epsilon$ and $t=a-(2-\sqrt{3})\epsilon$. By the quotient rule $f$ is twice differentiable almost everywhere. 

With $p(t)$ defined in (\ref{p(t)}), define 
\begin{equation} \nonumber
H_p:=\{\mbox{$f^{[0]}(t)$ and $f^{[1]}(t)$ are {absolutely continuous} on $[0,a]$}\}.
\end{equation}
With the above discussion, we have the following corollary based on \cite{atkinson176141discrete} that are related to understanding the spectrum of the LLE matrix.

\begin{corollary} 
Suppose we impose the Dirichlet boundary condition for the eigenvalue problem, $D_\epsilon f(t)=\lambda f(t)$, over the elliptic region $[(2-\sqrt{3})\epsilon, a-(2-\sqrt{3})\epsilon]$; that is, 
\[
f((2-\sqrt{3})\epsilon)=f(a-(2-\sqrt{3})\epsilon)=0. 
\]
Then the eigenvalues are discrete and bounded from above; that is, the eigenvalues are $\infty>\lambda_0>\lambda_1>\lambda_2>\ldots$ so that $\lambda_n \rightarrow -\infty$ as $n \rightarrow \infty$. If $f_n \in H_p$ is the corresponding eigenfunctions of $\lambda_n$, then $f_n$ has exactly $n$ zeros in $((2-\sqrt{3})\epsilon, a-(2-\sqrt{3})\epsilon)$.
\end{corollary}
{
We mention that this corollary is only for theoretical interest but not for practical interest since we need extra steps to ``clip'' the wave region and impose the Dirichlet boundary condition when we only have a point cloud.
Also, the above conclusion may not hold when $P(t)$ is not uniform. In fact, it is not hard to show that if $P(t)$ behaves like $t+\epsilon$ in $[0, \epsilon]$, then the corresponding SL problem is singular.

}

\section{A comparison of LLE and DM}
We provide a comparison of LLE and DM \cite{Coifman_Lafon:2006} on a manifold with smooth boundary. Recall that, unlike LLE, when we run DM, the affinity matrix is defined by composing a {\em fixed} kernel function chosen by the user with the distance between pairs of sampled points. Below we summarize the bias analysis result of DM using our notations for a further comparison when the manifold has a non-empty boundary. A full calculation can be found in \cite{Coifman_Lafon:2006,Singer_Wu:2016}.
To simplify the comparison, we consider the Gaussian kernel $H(t)=e^{-t^2}$. More general kernels can be considered, and we refer the reader with interest to, e.g.,  \cite{Coifman_Lafon:2006,Singer_Wu:2016}. Also, see \cite{vaughn2020diffusion} for more relevant results. 
For $x,y \in M$, we define $H_{\epsilon}(x,y)=\exp(-\frac{\|\iota(x)-\iota(y)\|^2}{\epsilon^2})$, where $\epsilon>0$ is the bandwidth. For $0 \leq \alpha \leq 1$, we define the $\alpha$-normalized kernel as
\begin{align*}
H_{\epsilon,\alpha}(x,y)&\,:=\frac{H_{\epsilon}(x,y)}{p^{\alpha}_{\epsilon}(x) p^{\alpha}_{\epsilon}(y)},
\end{align*}
where $p_{\epsilon}(x):=\mathbb{E}[H_\epsilon (x,X)]$.
With the $\alpha$-normalized kernel $H_{\epsilon,\alpha}$, for $f \in C^{3}(M)$, the diffusion operator associated with the $\alpha$-normalized DM is 
\begin{equation}
\mathcal{H}_{\epsilon, \alpha}f(x):=\frac{\mathbb{E}[H_{\epsilon,\alpha} (x,X) f(X)]}{\mathbb{E}[H_{\epsilon,\alpha} (x,X)]}.
\end{equation}
The behavior of the operator  $\mathcal{H}_{\epsilon, \alpha}$ is summarized below.
\begin{theorem}[Bias analysis of Diffusion map]
Let  $(M,g)$ be a d-dimensional compact, smooth Riemannian manifold isometrically embedded in $\mathbb{R}^p$, with a non-empty smooth boundary.  Suppose $f \in C^3(M)$ and $P\in C^2(M)$. If $\alpha=1$, we have
\begin{align}
\mathcal{H}_{\epsilon, \alpha}f(x) =& \frac{\sigma_{1,d}(\tilde{\epsilon}_x )}{\sigma_{0}(\tilde{\epsilon}_x )} \partial_d f(x) \epsilon +\Big[\psi_1(\tilde{\epsilon}_x)\sum_{i=1}^{d-1}\partial^2_{ii} f(x)+ \psi_2(\tilde{\epsilon}_x)\partial^2_{dd} f(x) \Big] \epsilon^2 \\
&+U (\tilde{\epsilon}_x ) \partial_d f(x) \epsilon^2+O(\epsilon^3), \nonumber
\end{align}
where $\psi_1,\psi_2$ and $U$ are scalar value functions defined on $[0,\infty)$ so that
\begin{align*}
\psi_1(t)\,:=\frac{1}{2}\frac{\sigma_{2}(t)}{\sigma_{0}(t)},\quad
\psi_2(t)\,:=\frac{1}{2}\frac{\sigma_{2,d}(t)}{\sigma_{0}(t)},
\end{align*}
$U(t)=0$ if $t \geq \epsilon$, $U(\tilde{\epsilon}_x)$ depends on the second fundamental form of $\partial M$ in $M$ at $x$, and $U(\tilde{\epsilon}_x )$ is independent of $P$.
In fact,
\begin{align}
U(\tilde{\epsilon}_x )=\frac{\int_{\tilde{D}_{\epsilon}(x)} u_d du}{\int_{\tilde{D}_{\epsilon}(x)} 1 du}-\frac{\sigma_{1,d}(\tilde{\epsilon}_x )}{\sigma_{0}(\tilde{\epsilon}_x )} \epsilon\,,
\end{align}
where $\frac{\int_{\tilde{D}_{\epsilon}(x)} u_d du}{\int_{\tilde{D}_{\epsilon}(x)} 1 du}$ is a function depending on $\tilde{\epsilon}_x $ and the second fundamental form of $\partial M$ as a codimension $1$ submanifold embedded in $M$ at $x$ by Definition \ref{coordinates near boundary}.
\end{theorem}

Compared with the differential operator $\mathcal{D}_\epsilon$ associated with LLE, the differential operator associated with DM has a very different behavior. First, in DM, when $\epsilon>0$ is finite, asymptotically the first order differential operator exists in the $\epsilon$ order \cite{Coifman_Lafon:2006}, {which suggests that} the boundary condition is {\em Neumann}. {In \cite{vaughn2020diffusion}, it is shown that the graph Laplacian converges in the weak sense to the Laplace-Beltrami operator with the Neumann boundary condition.} With this boundary condition, the spectral convergence of DM when the boundary is non-empty {without a convergence rate was provided  in \cite{Singer_Wu:2016} as a special case when the considered group action is $SO(1)$. Another spectral convergence result when the boundary exists can be found in \cite{peoples2021spectral}}. Note that when the boundary is empty, more spectral convergence results  with a convergence rate can be found in \cite{von2008consistency,trillos2018error,dunson2021spectral}, while none provide the rate is optimal to our knowledge. Second, the coefficients of the second order differential operator, $\psi_1$ and $\psi_2$, do not change sign and do not degenerate over the whole manifold. Third, in DM, there is an extra first order differential operator in the $\epsilon^2$ order. As a result, in addition to what has been explored in \cite{Wu_Wu:2017}, we see more differences between LLE and DM.

{
While LLE and DM are different, they are intimately related. Here we provide a brief exploration for this relationship from the kernel perspective. Observe that we can rewrite the kernel as
\begin{align}
\frac{\mathbb{E}[K(x,X)f(X)]}{\mathbb{E}K(x,X)}&\,=\frac{\mathbb{E}[\chi_{B_{\epsilon}^{\mathbb{R}^p}(\iota(x))}(\iota(X))-(\iota(X)-\iota(x))^\top \mathbf{T}(x)\chi_{B_{\epsilon}^{\mathbb{R}^p}(\iota(x))}(\iota(X))]f(X)}{\mathbb{E}[\chi_{B_{\epsilon}^{\mathbb{R}^p}(\iota(x))}(\iota(X))-(\iota(X)-\iota(x))^\top \mathbf{T}(x)\chi_{B_{\epsilon}^{\mathbb{R}^p}(\iota(x))}(\iota(X))]} \nonumber\\
&\,=\frac{\mathbb{E}[\frac{1}{2}\chi_{B_{\epsilon}^{\mathbb{R}^p}(\iota(x))}(\iota(X))-\frac{1}{2}(\iota(X)-\iota(x))^\top \mathbf{T}(x)\chi_{B_{\epsilon}^{\mathbb{R}^p}(\iota(x))}(\iota(X))]f(X)}{\mathbb{E}[\frac{1}{2}\chi_{B_{\epsilon}^{\mathbb{R}^p}(\iota(x))}(\iota(X))-\frac{1}{2}(\iota(X)-\iota(x))^\top \mathbf{T}(x)\chi_{B_{\epsilon}^{\mathbb{R}^p}(\iota(x))}(\iota(X))]}\,,\nonumber 
\end{align}
which means that the kernel function is an ``average'' of two functions, 
\[
K_1(x,y):=\chi_{B_{\epsilon}^{\mathbb{R}^p}(\iota(x))}(\iota(y))
\] 
and 
\[
K_2(x,y):=-(\iota(y)-\iota(x))^\top \mathbf{T}(x)\chi_{B_{\epsilon}^{\mathbb{R}^p}(\iota(x))}(\iota(y)). 
\]
We can thus consider the following kernel generalizing the LLE kernel:
\[
K^{(\alpha)}(x,y)=\alpha K_1(x,y)+(1-\alpha)K_2(x,y),
\]
where $\alpha\in [0,1]$. 
Note that $K_1$ can be viewed as a $0-1$ kernel that is commonly used in DM, so when $\alpha=1$, we recover the DM. When $\alpha=1/2$, it is clear that $K^{(1/2)}$ is the LLE kernel. When $\alpha=0$, we get a different kernel with different behavior. Recall Definition \ref{DefAugmented}. We have
\begin{align}
K_2(x,y) & =- \mathbb{E}[(\iota(X)-\iota(x))\chi_{B_{\epsilon}^{\mathbb{R}^p}(\iota(x))}(\iota(X))] ^\top I_{p,r}(C_x+cI_{p\times p})^{-1}I_{p,r}(\iota(y)-\iota(x))\nonumber\\
&=\int_{B_{\epsilon}^{\mathbb{R}^p}(\iota(x))} (\iota(z)-\iota(x)) ^\top I_{p,r}(C_x+cI_{p\times p})^{-1}I_{p,r}(\iota(y)-\iota(x))dz\,.
\end{align}
As discussed in \cite{2018arXiv180402811M}, since $I_{p,r}(C_x+cI_{p\times p})^{-1}I_{p,r}$ can be viewed as the ``regularized precision matrix'', we can view $(\iota(z)-\iota(x)) ^\top I_{p,r}(C_x+cI_{p\times p})^{-1}I_{p,r}(\iota(y)-\iota(x))$ as the local Mahalanobis distance between $z$ and $y$, or the distance between the latent variables related to $z$ and $y$. Thus, when $\alpha=0$, the kernel comes from averaging out the pairwise local Mahalanobis distance, and hence depends on the local geometric structure. The relationship between the second fundamental form of $M$ at $x$ and the latent space will be explored in the future work. 

It is natural to ask if we can ``alleviate'' the impact of the wave region by choosing different $\alpha$. To answer this question, we briefly discuss the behavior of $\mathbb{E}[K^{\alpha}(x,X)f(X)]$ with different choices of $\alpha$, particularly when $\alpha<1$.
Let us take a more careful look at the case when $\alpha=1/2$; that is, the kernel for LLE; particularly, we look into the reason why LLE does not have the Neurman boundary condition, and why LLE is independent of the nonuniform density function from the kernel perspective. We have
\begin{align}
\mathbb{E}[K^{(1/2)}(x,X)&f(X)]=\frac{1}{2}\mathbb{E}[(f(X)-f(x))\chi_{B_{\epsilon}^{\mathbb{R}^p}(\iota(x))}(\iota(X))]\\
&-\frac{1}{2}\mathbb{E}[(\iota(X)-\iota(x))(f(X)-f(x))\chi_{B_{\epsilon}^{\mathbb{R}^p}(\iota(x))}(\iota(X))] ^\top \mathbf{T}(x). \nonumber
\end{align}
Recall the behavior of the two terms on the right hand side. By a direct calculation, we know $\mathbb{E}[(f(X)-f(x))\chi_{B_{\epsilon}^{\mathbb{R}^p}(\iota(x))}(\iota(X))]$ becomes $P(x)\sigma_{1,d}(\tilde{\epsilon}_x ) \partial_d f(x) \epsilon^{d+1} +O(\epsilon^{d+2})$, when $x \in M_{\epsilon}$, and $[\frac{1}{2}P(x) \Delta f(x)+\nabla f(x) \cdot \nabla P(x) ]\epsilon^{d+2}+O(\epsilon^{d+3})$ when $x \not\in M_{\epsilon}$.
Note that this is the behavior of the kernel $K^{(1)}$. On the other hand, $\mathbb{E}[(\iota(X)-\iota(x))(f(X)-f(x))\chi_{B_{\epsilon}^{\mathbb{R}^p}(\iota(x))}(\iota(X))] ^\top \mathbf{T}(x)$ becomes $P(x)\sigma_{1,d}(\tilde{\epsilon}_x ) \partial_d f(x) \epsilon^{d+1} +O(\epsilon^{d+2})$ when $x \in M_{\epsilon}$, and $\nabla f(x) \cdot \nabla P(x) \epsilon^{d+2}+O(\epsilon^{d+3})$ when $x \not\in M_{\epsilon}$. Note that this is the behavior of the kernel $K^{(0)}$. 
As a result, when $x \in M_{\epsilon}$, since the common term $P(x)\sigma_{1,d}(\tilde{\epsilon}_x ) \partial_d f(x) \epsilon^{d+1}$ cancels, there is no such Neumann boundary behavior when $\alpha=1/2$ as that in DM. When $x \not\in M_{\epsilon}$, then the common term $\nabla f(x) \cdot \nabla P(x) \epsilon^{d+2}$ cancels. Hence, the behavior of LLE in the interior of the manifold is independent of the density function. 

However, it is worth noting that one cannot remove the wave region $M_w$ through adjusting $\alpha$ after the above analysis -- since both $\mathbb{E}[(f(X)-f(x))\chi_{B_{\epsilon}^{\mathbb{R}^p}(\iota(x))}(\iota(X))]$ and $\mathbb{E}[(\iota(X)-\iota(x))(f(X)-f(x))\chi_{B_{\epsilon}^{\mathbb{R}^p}(\iota(x))}(\iota(X))] ^\top \mathbf{T}(x)$ are dominated by $P(x)\sigma_{1,d}(\tilde{\epsilon}_x ) \partial_d f(x) \epsilon^{d+1}$ when $x$ is near the boundary, if $\alpha\neq 1/2$, the first order term remains. 
}


\section{Clipped LLE matrix}\label{section numerics}

{ Based on the above theoretical results, we provide an immediate application. The Laplace-Beltrami operator with the Dirichlet boundary condition is widely used in various fields, like in the analysis of stochastic dynamics. For example,  in \cite{georgiou2017exploration} the eigenfunctions of the Laplace-Beltrami operators with the Neumann and Dirichlet boundary conditions are used together to reconstruct the conformational space of a stochastic gradient system. According to the developed theory in Theorem \ref{Theorem:final}, the {asymptotic} operator in general behaves well away from the boundary. We thus consider the following modification of LLE that echos the theoretical development in Section \ref{Section Boundary} and this algorithm is a potential candidate to recover the Laplace-Beltrami operator with the Dirichlet boundary condition on manifold with boundary.}

For a given sampling set $\mathcal{X}=\{x_i\}_{i=1}^n$, due to the $\epsilon$-radius nearest neighbor scheme, {assume we can} divide the LLE matrix $W\in \mathbb{R}^{n\times n}$ into blocks according to four portions, the interior, transition, wave-boundary and non-wave-boundary portions.
The interior portion $\mathcal{X}_I:=\{x_i\in \mathcal{X}:\,d(x_i,\partial M)>2\epsilon\}$ that includes points far away from the boundary, the wave-boundary portion $\mathcal{X}_w:=\{x_i\in \mathcal{X}:\,x_i\in M_w\cup \partial M\}$ that includes points in the wave region, $\mathcal{X}_B:=\{x_i\in \mathcal{X}:\,x_i\in M_\epsilon\backslash (M_w\cup \partial M)\}$
and the transition portion $\mathcal{X}_T:=\{x_i\in \mathcal{X}:\,\epsilon\leq d(x_i,\partial M)\leq 2\epsilon\}$ that includes the remaining points touching the other three portions.
The $W$ matrix is thus divided into
\[
W=\begin{bmatrix}W_{ww} & W_{wB} &W_{wT}& 0\\ W_{Bw} & W_{BB} &W_{BT}& 0\\  W_{Tw}& W_{TB}& W_{TT} &W_{TI} \\ 0 & 0 &W_{IT}& W_{II}\end{bmatrix}\,,
\]
where $W_{ww}\in \mathbb{R}^{|\mathcal{X}_w|\times |\mathcal{X}_w|}$ represents the wave-boundary portion of the LLE matrix, $W_{BB}\in \mathbb{R}^{|\mathcal{X}_B|\times |\mathcal{X}_B|}$ represents the non-wave-boundary portion of the LLE matrix, $W_{TT}\in \mathbb{R}^{|\mathcal{X}_T|\times |\mathcal{X}_T|}$ represents the transition portion of the LLE matrix, $W_{II}\in \mathbb{R}^{|\mathcal{X}_I|\times |\mathcal{X}_I|}$  represents the interior portion of the LLE matrix, and the other submatrices represent the interaction of the four portions of the LLE matrix.
Construct a new matrix $W_r\in \mathbb{R}^{n'\times n'}$, where $n'=n-|\mathcal{X}_w|$, by restricting $W$ to $\mathcal{Y}$; that is, 
\[
W_r=\begin{bmatrix} W_{BB} &W_{BT}& 0\\   W_{TB}& W_{TT} &W_{TI} \\  0 &W_{IT}& W_{II}\end{bmatrix}\,. 
\]
We call $W_r$ the {\em clipped LLE matrix} for simplicity. Note that in general the matrix $W_r$ is not symmetric, not a transition matrix, and the rows may not sum to $1$. 
{We shall emphasize that this algorithm depends on the knowledge of the wave region. In general, we need an algorithm to detect the wave region from a point cloud.}

Below we see some numerical results.
We uniformly and independently sample points from $M_1:=[0,1]\subset \mathbb{R}^1$. The LLE matrix is constructed with the $\epsilon$-radius scheme, where $\epsilon=0.01$. 
The first $5$ eigenfunctions are shown in the top panel of Figure \ref{Figure:CurveEigenfunctions}. Note that the first eigenfunction is constant, and the second eigenfunction is linear, and both are with eigenvalue $1$; that is, these two eigenfunctions form the null space of $I-W$. The other eigenfunctions ``look like'' the eigenfunctions of the Laplace-Beltrami operator with the Dirichlet boundary condition, but higher eigenfunctions become ``irregular'' when getting closer to the boundary. 
Next, consider another 1-dim curve $M_2$ embedded in $\mathbb{R}^3$ that is parametrized by $t\to (t, \log(0.5+t), \cos(\pi t))^\top\in\mathbb{R}^3$, where $t\in[0,1]$. We uniformly and independently sample $8,000$ points from $[0,1]$ and mapped them to $M_2$. Denote the sampled points $\mathcal{X}:=\{x_i\}_{i=1}^{8,000}\subset\mathbb{R}^3$. Note that the sample is not uniform. The LLE matrix $W\in \mathbb{R}^{8,000\times 8,000}$ is constructed with the $\epsilon$-radius scheme, where $\epsilon=0.01$, and hence the clipped LLE matrix $W_r$. 
The first $10$ eigenfunctions of $W_r$ constructed from $M_1$ and $M_2$ are shown in the middle and bottom panels in Figure \ref{Figure:CurveEigenfunctions}. 
According to Corollary \ref{Corollary one dim} and the discussion in Subsection \ref{OpenQuestions}, the asymptotic operator is well behaved in $[(2-\sqrt{3})\epsilon, 1-(2-\sqrt{3})\epsilon]$. This theoretical finding fits the numerical results --
the eigenfunctions are all $0$ at the ``boundary points'' $(2-\sqrt{3})\epsilon$ and $1-(2-\sqrt{3})\epsilon$.

\begin{figure}[ht]
\center
\includegraphics[width=.9\textwidth]{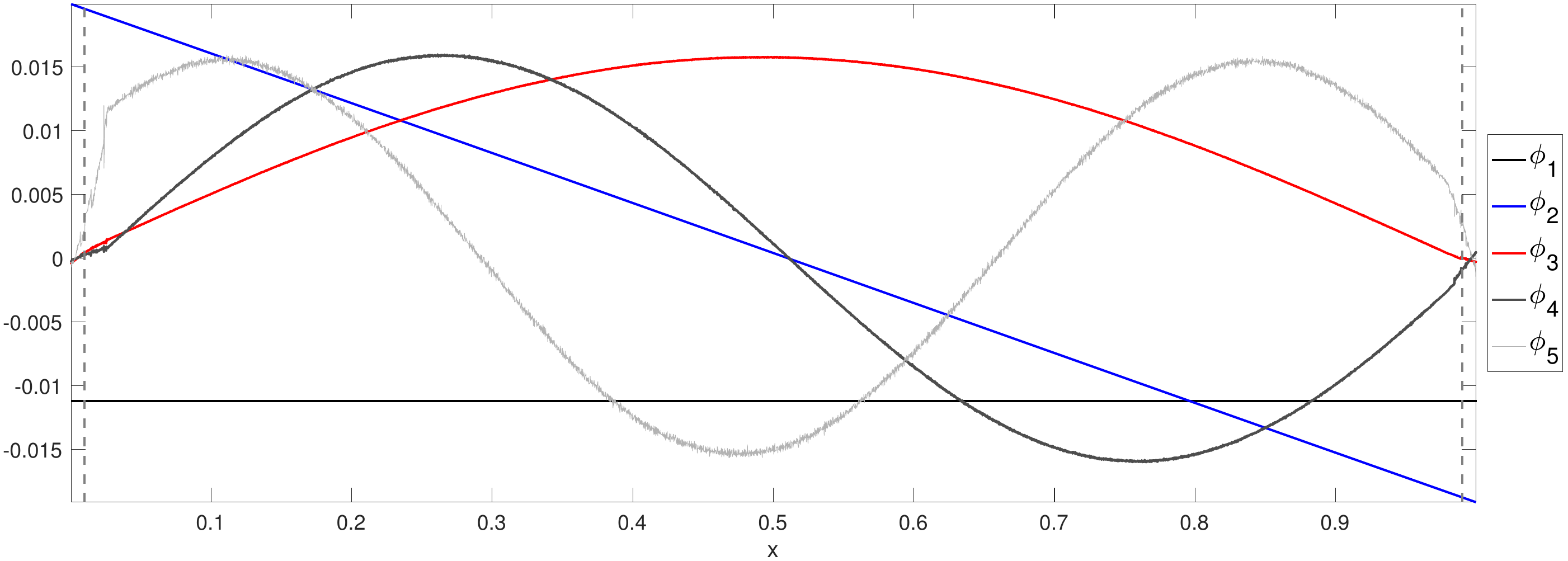}
\includegraphics[width=.9\textwidth]{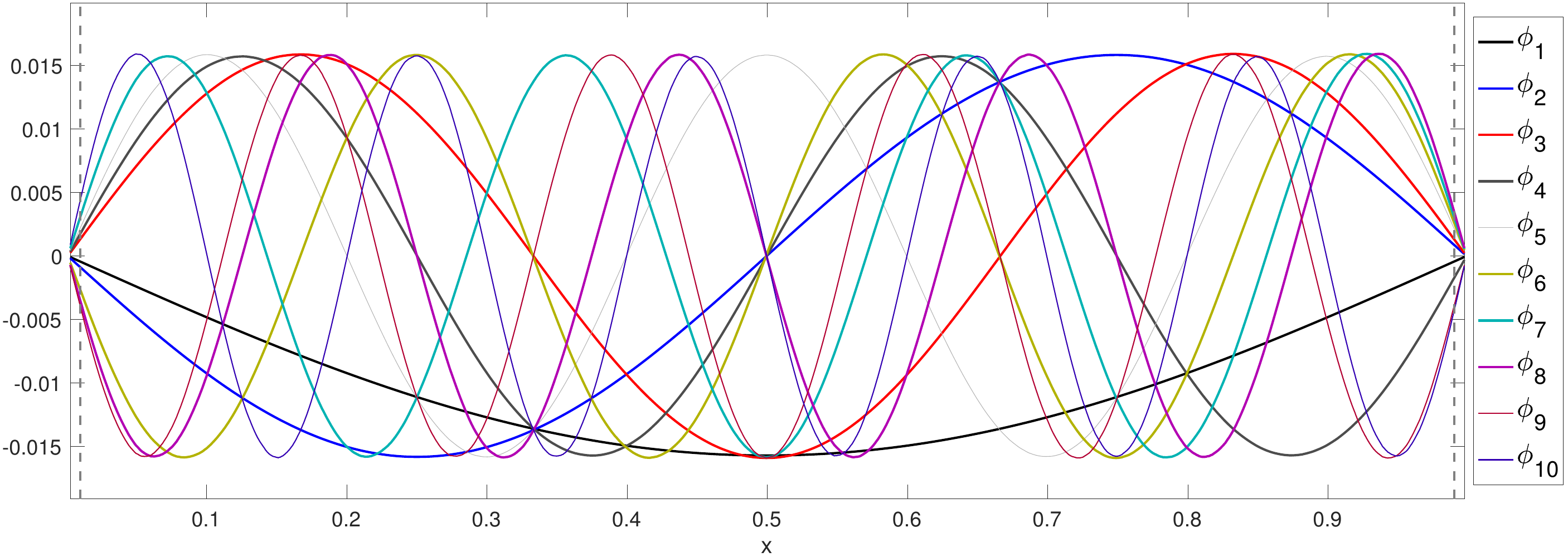}\\
\includegraphics[width=.9\textwidth]{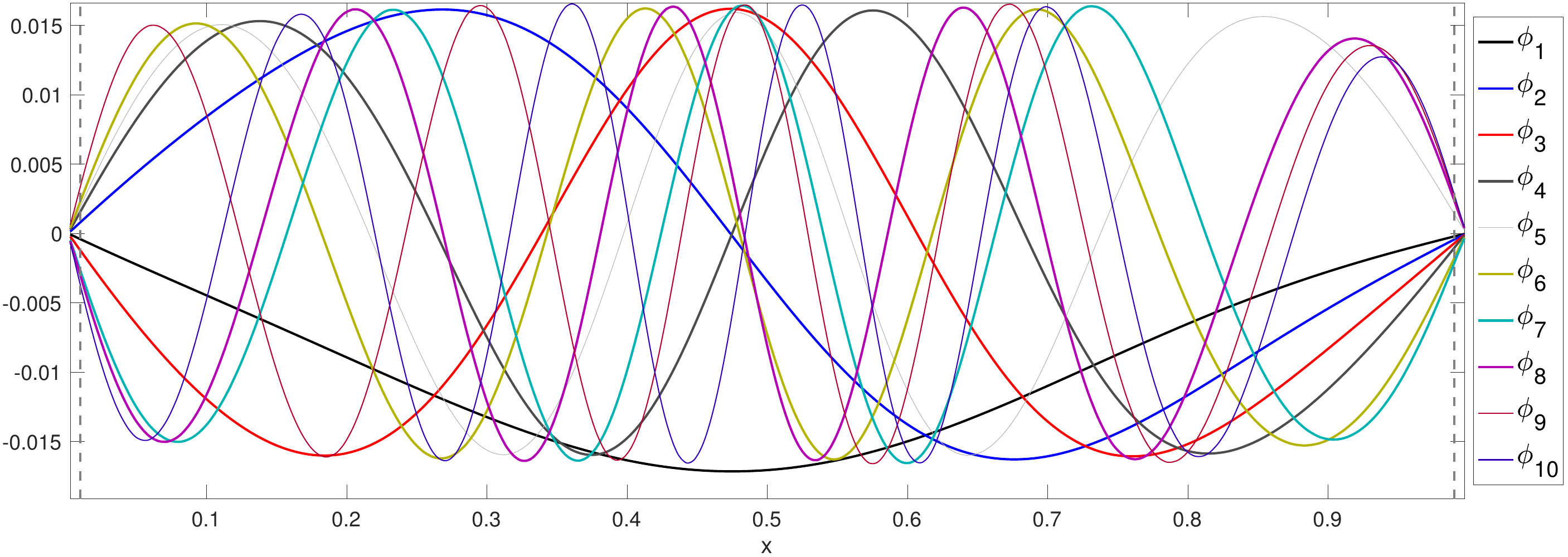}
\caption{Top: The first $5$ eigenfunctions of the LLE matrix for a point cloud sampled from the $[0,1]$ interval are plotted with different colors. The dashed vertical gray lines indicate $\epsilon$ and $1-\epsilon$. It is clear that the third, fourth, and fifth eigenfunctions ``look like'' the eigenfunctions of the Laplace-Beltrami operator with the Dirichlet boundary condition, but higher eigenfunctions become ``irregular'' when getting closer to the boundary.
Middle: the first $10$ eigenfunctions of the clipped LLE matrix $W_r$ for a point cloud sampled from $M_1=[0,1]$ are plotted with different colors. The dashed vertical gray lines indicate $\epsilon$ and $1-\epsilon$.
Bottom: the first $10$ eigenfunctions of the clipped LLE matrix $W_r$ for a point cloud sampled from the curve $M_3\subset \mathbb{R}^3$ are plotted with different colors. The dashed vertical gray lines indicate $\epsilon$ and $1-\epsilon$. Compared with those eigenfunctions of $M_1$, the amplitude of higher eigenfunctions of $M_3$ becomes less constant, which is expected due to the nonuniform sampling effect. 
It is clear that these eigenfunctions ``look like'' the eigenfunctions of the Laplace-Beltrami operator with the Dirichlet boundary condition without the ``irregularity'' behavior close to the boundary observed in the top panel.}
\label{Figure:CurveEigenfunctions}
\end{figure}

Second, we uniformly sample points from a unit disk, $M_3\subset \mathbb{R}^2$, by keeping points with norm {less than or equal to $1$} from $20,000$ points sampled uniformly and independently from $[-1,1]\times[-1,1]$. The LLE matrix is constructed with the $\epsilon$-radius ball nearest neighbor search scheme, where $\epsilon=0.1$. The first $20$ eigenfunctions are shown in Figure \ref{Figure:DiskEigenfunctions}. 
The first eigenfunction is constant, and the second and third eigenfunctions are linear, and these three eigenfunctions are associated with eigenvalue $1$. These three eigenfunctions form the null space of $I-W$. 
We can see three types of eigenfunctions -- those of the first type ``look like'' the eigenfunctions of the Laplace-Beltrami operator with the Dirichlet boundary condition, those of the second type ``look like'' the eigenfunctions of the Laplace-Beltrami operator restricted to the ``rim'' near the boundary, which is topologically a closed manifold $S^1$, and those of the third type ``look like'' the mix-up of the first two types. 
Next, we explore the clipped LLE matrix on the unit disk $M_3\subset \mathbb{R}^2$ with the same uniform sampling scheme. For $M_3$, we remove rows and columns associated with points with norm greater than {$1-(2-\sqrt{3})\epsilon$}, where $\epsilon=0.1$. The first 20 eigenfunctions are shown in Figure \ref{Figure:DiskModifiedEigenfunctions}. It is clear that compared with those shown in Figure \ref{Figure:DiskEigenfunctions}, all eigenfunctions are $0$ at the ``boundary''. 

\begin{figure}[ht]
\center
\includegraphics[width=.98\textwidth]{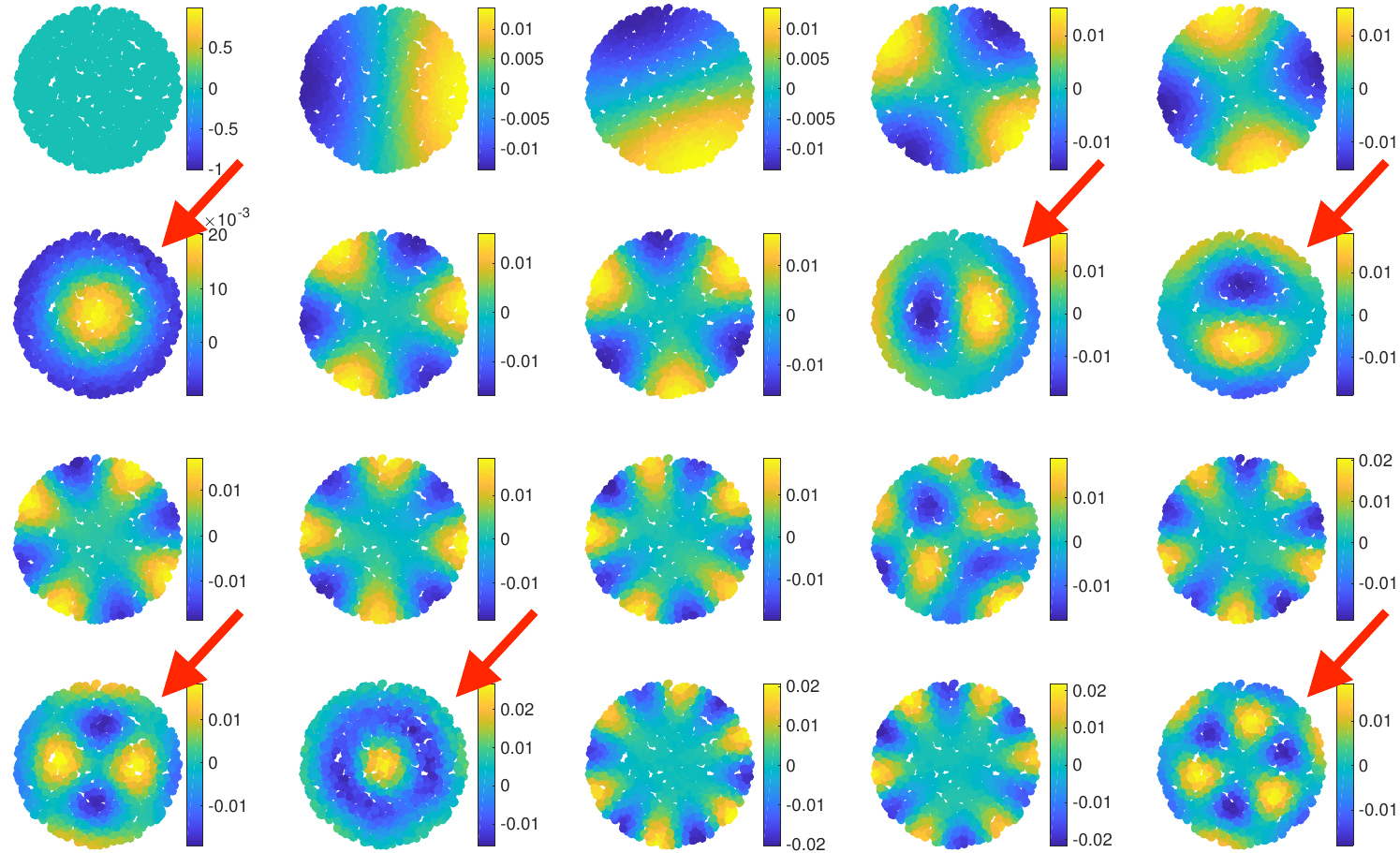}
\caption{The first $20$ eigenfunctions of the LLE matrix for a point cloud sampled from the unit disk are plotted from top left to bottom right. It is clear that some eigenfunctions (indicated by red arrows) ``look like'' the eigenfunctions of the Laplace-Beltrami operator with the Dirichlet boundary condition combined with the eigenfunctions of Laplace-Beltrami operator of the ``rim'' near the boundary. Note that the ``rim'' near the boundary is close to $S^1$ in the Gromov-Hausdorff sense. }
\label{Figure:DiskEigenfunctions}
\end{figure}

Next, the first 20 eigenfunctions of the LLE matrix and the clipped LLE matrix of the surface shown in Figures \ref{Figure:Tvec} and \ref{Figure:Kernel} with the same sampling scheme are shown in Figures \ref{Figure:SurfaceEigenfunctions} and \ref{Figure:SurfaceModifiedEigenfunctions}. It is clear that while the first 3 eigenfunctions of the LLE matrix behave like constant or linear functions, the other eigenfunctions are not easy to describe. However, all eigenfunctions of the clipped LLE matrix are zero on the ``boundary'', which behaves like the eigenfunctions of the Laplace-Beltrami operator with the Dirichlet boundary condition. 

The above examples are all manifolds without interesting topological structure since they can all be parametrized by one chart. In the final example we show a two dimensional manifold with non-trivial topology. Consider a torus embedded in $\mathbb{R}^3$, which is parametrized by 
\[
\Phi:(\theta,\phi)\mapsto ((3+1.2\cos(\theta))\cos(\phi),(3+1.2\cos(\theta))\sin(\phi),1.2\sin(\phi))^\top\in \mathbb{R}^3, 
\]
where $\theta,\phi\in [0,2\pi)$. The manifold $M_4$ is defined as
\[
M_4=\{\Phi(\theta,\phi):\,\theta,\phi\in [0,2\pi)\mbox{ and }(3+1.2\cos(\theta))\cos(\phi)>-3.4\}\,;
\]
that is, $M_4$ is a truncated torus with the boundary diffeomorphic to $S^1$. We sample uniformly $25,000$ points on $[0,2\pi]\times [0,2\pi]$, and remove points associated with $(3+1.2\cos(\theta))\cos(\phi) >-3.4$. Note that this is a nonuniform sampling scheme from $M_4$. Then, establish the LLE matrix and the clipped LLE matrix with $\epsilon=0.3$. The results are shown in Figures \ref{Figure:TorusEigenfunctions} and \ref{Figure:TorusModifiedEigenfunctions}. Again, it is clear that the eigenfunctions of the clipped LLE matrix are zero on the ``boundary''.

With these numerical results, we conjecture that if we clip the wave region, the operator $\mathcal{D}_\epsilon$ over $M\backslash (M_w\cup \partial M)$ asymptotically converges to the Laplace-Beltrami operator with the Dirichlet boundary condition over $M\backslash (M_w\cup \partial M)$ in the spectral sense when $\epsilon\to 0$.\footnote{In the 1-dim case, this is related to a different differential equation, the Kimura equation \cite{epstein2013degenerate}, that shares the same degeneracy on the boundary. In the Kimura equation, the boundary condition is adaptively encoded in the functional space that we search for the eigenfunctions.} We will explore this problem in our future work.

\begin{figure}[ht]
\center
\includegraphics[width=.98\textwidth]{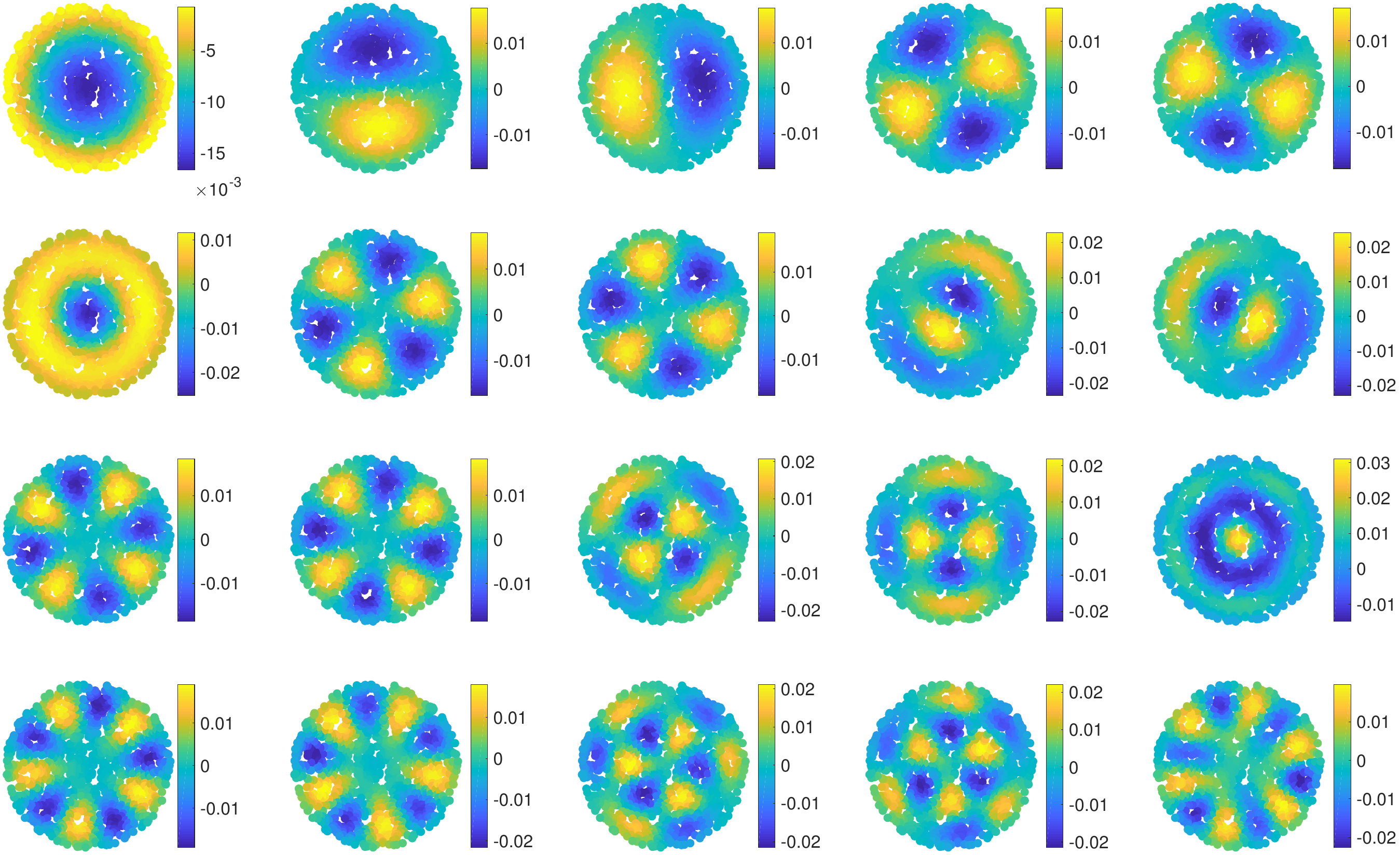}
\caption{The first $20$ eigenfunctions of the clipped LLE matrix for a point cloud sampled from the unit disk are plotted from top left to bottom right. It is clear that all eigenfunctions ``look like'' the eigenfunctions of the Laplace-Beltrami operator with the Dirichlet boundary condition.}
\label{Figure:DiskModifiedEigenfunctions}
\end{figure}

\begin{figure}[ht]
\center
\includegraphics[width=.98\textwidth]{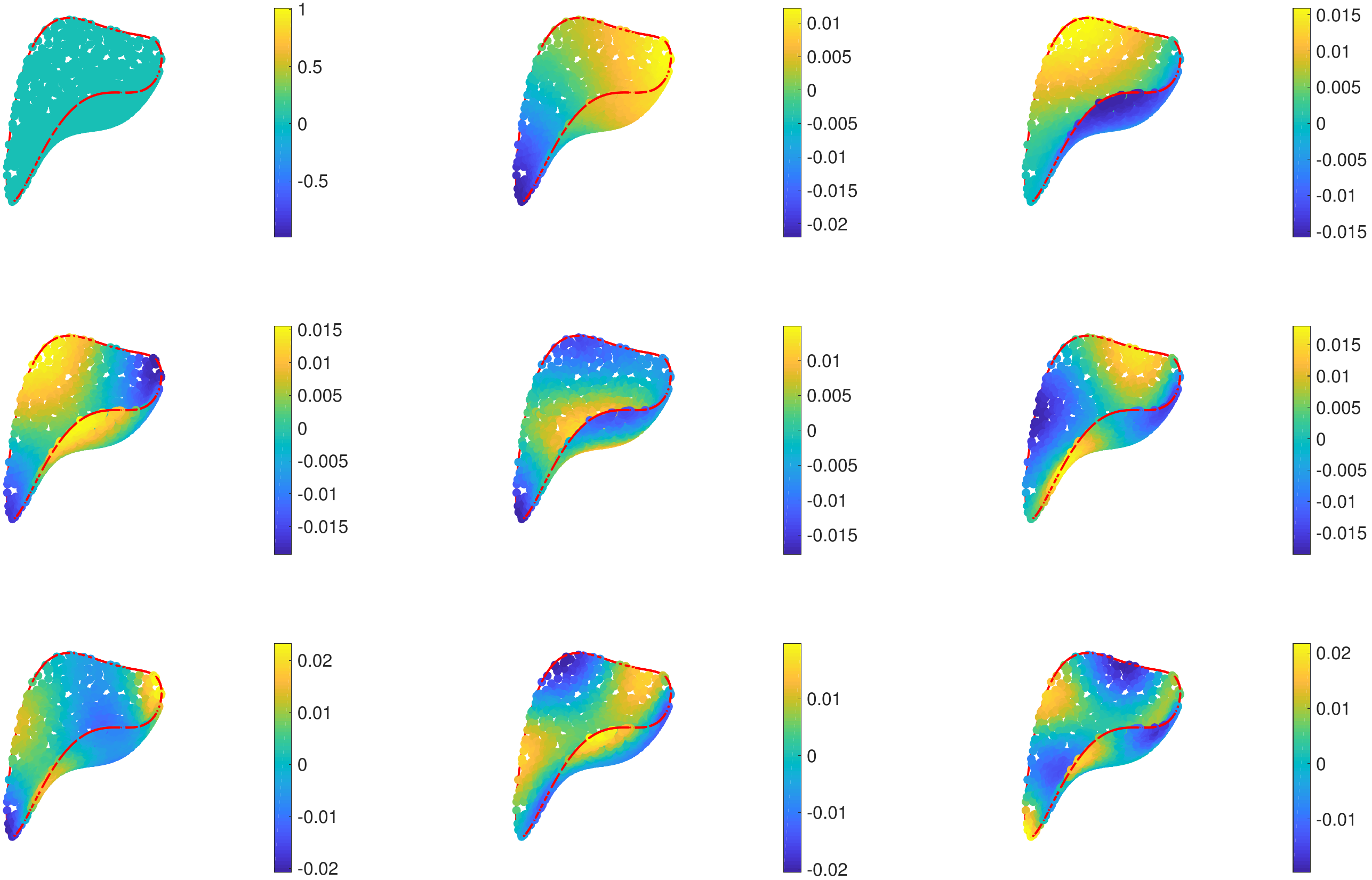}
\caption{The first $9$ eigenfunctions of the LLE matrix for a point cloud sampled from the surface in Figures \ref{Figure:Tvec} and \ref{Figure:Kernel} are plotted from top left to bottom right. To enhance the visualization, the boundary of the surface is colored by red. It is clear that the first three eigenfunctions are either constant or linear, while the behavior of other eigenfunctions is not easy to describe, while compared with those shown in Figure \ref{Figure:DiskEigenfunctions}.}
\label{Figure:SurfaceEigenfunctions}
\end{figure}

\begin{figure}[ht]
\center
\includegraphics[width=.98\textwidth]{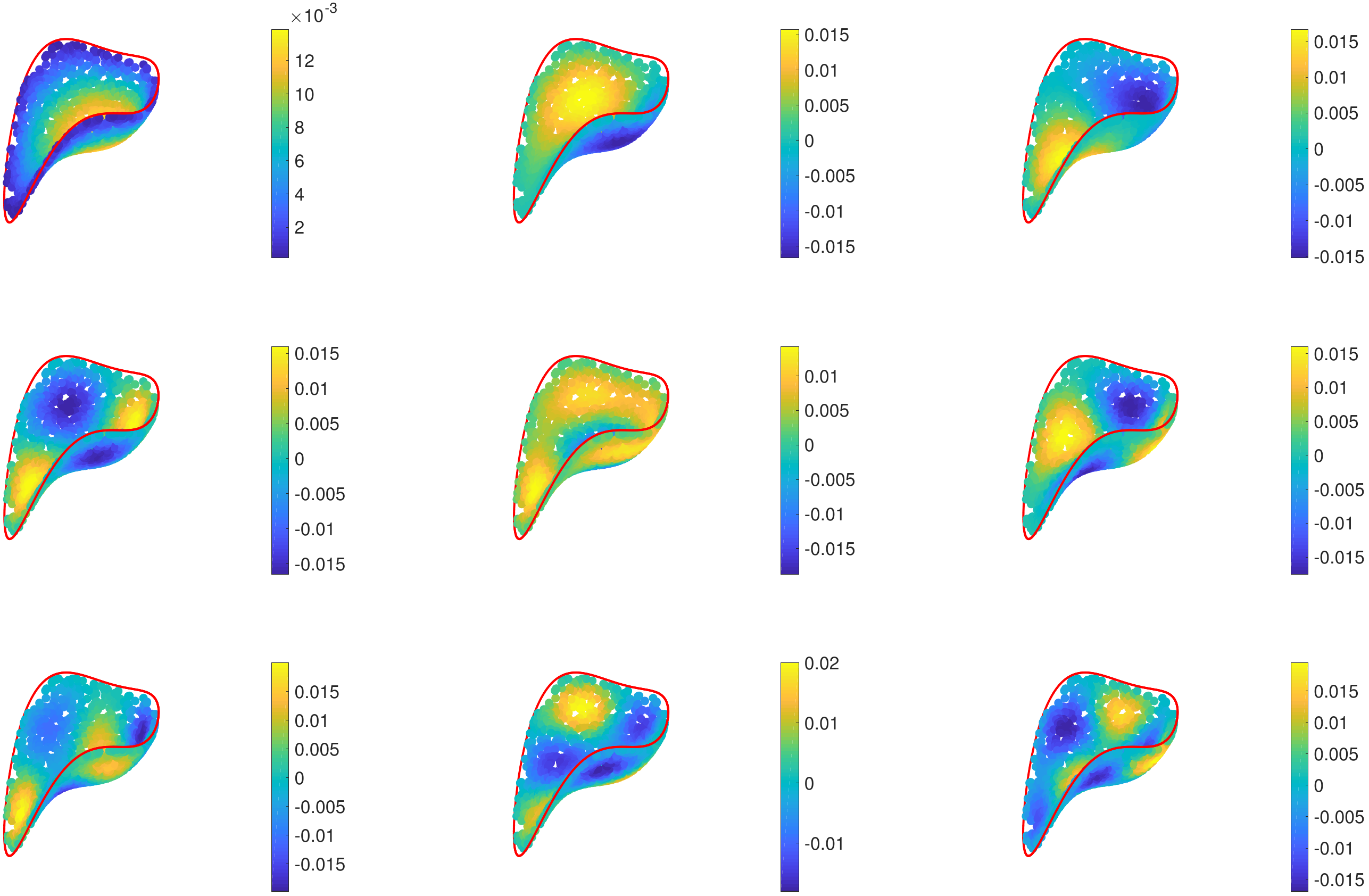}
\caption{The first $9$ eigenfunctions of the clipped LLE matrix for a point cloud sampled from the surface in Figures \ref{Figure:Tvec} and \ref{Figure:Kernel} are plotted from top left to bottom right. To enhance the visualization, the boundary of the surface is colored by red. It is clear that all eigenfunctions are zero on the boundary, and the behavior ``looks like'' the eigenfunctions of the Laplace-Beltrami operator with the Dirichlet boundary condition.}
\label{Figure:SurfaceModifiedEigenfunctions}
\end{figure}

\begin{figure}[ht]
\center
\includegraphics[width=.98\textwidth]{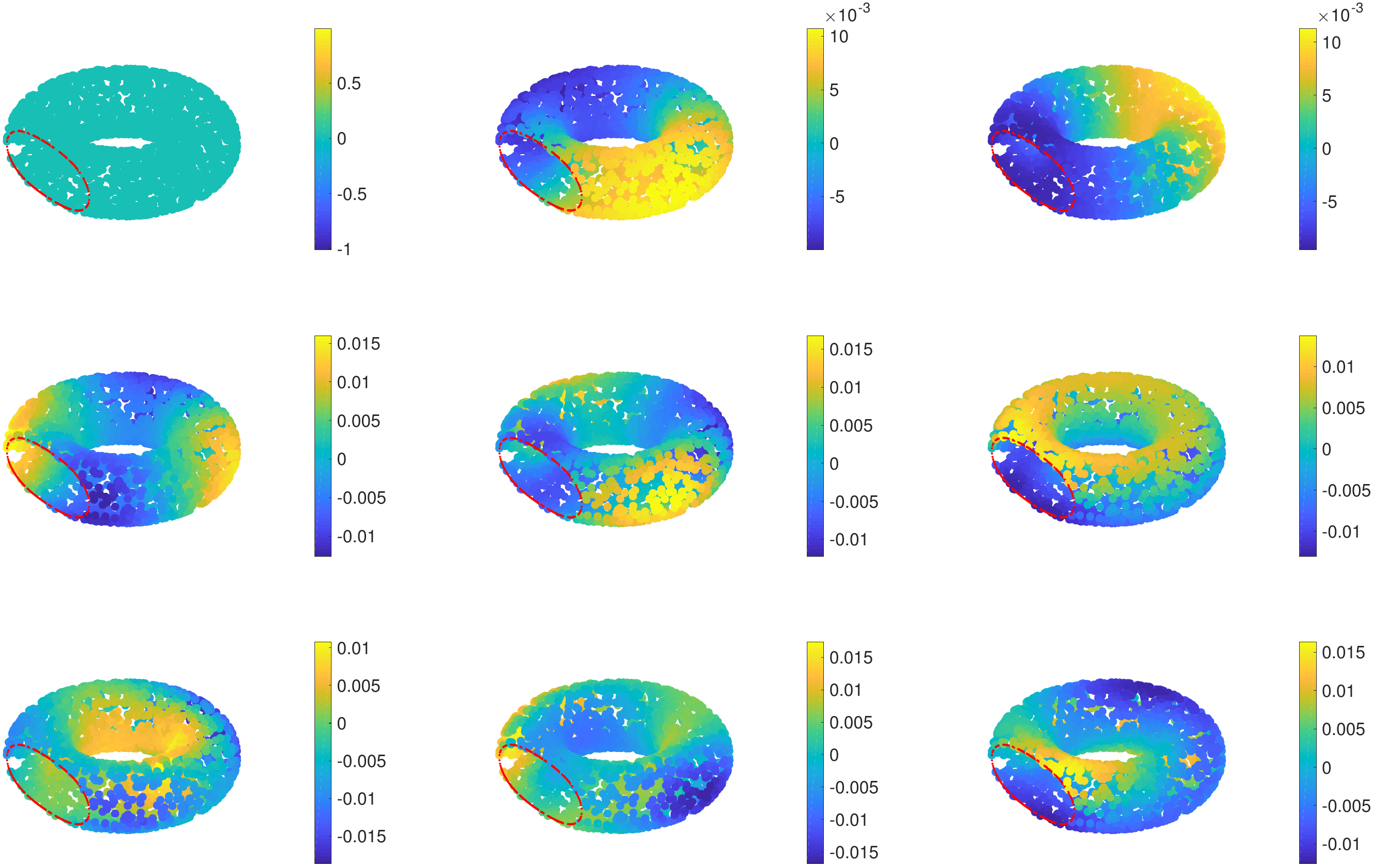}
\caption{The first $9$ eigenfunctions of the LLE matrix for a point cloud sampled from the truncated torus $M_4$ are plotted from top left to bottom right. To enhance the visualization, the boundary of the truncated torus is colored by red. }
\label{Figure:TorusEigenfunctions}
\end{figure}

\begin{figure}[ht]
\center
\includegraphics[width=.98\textwidth]{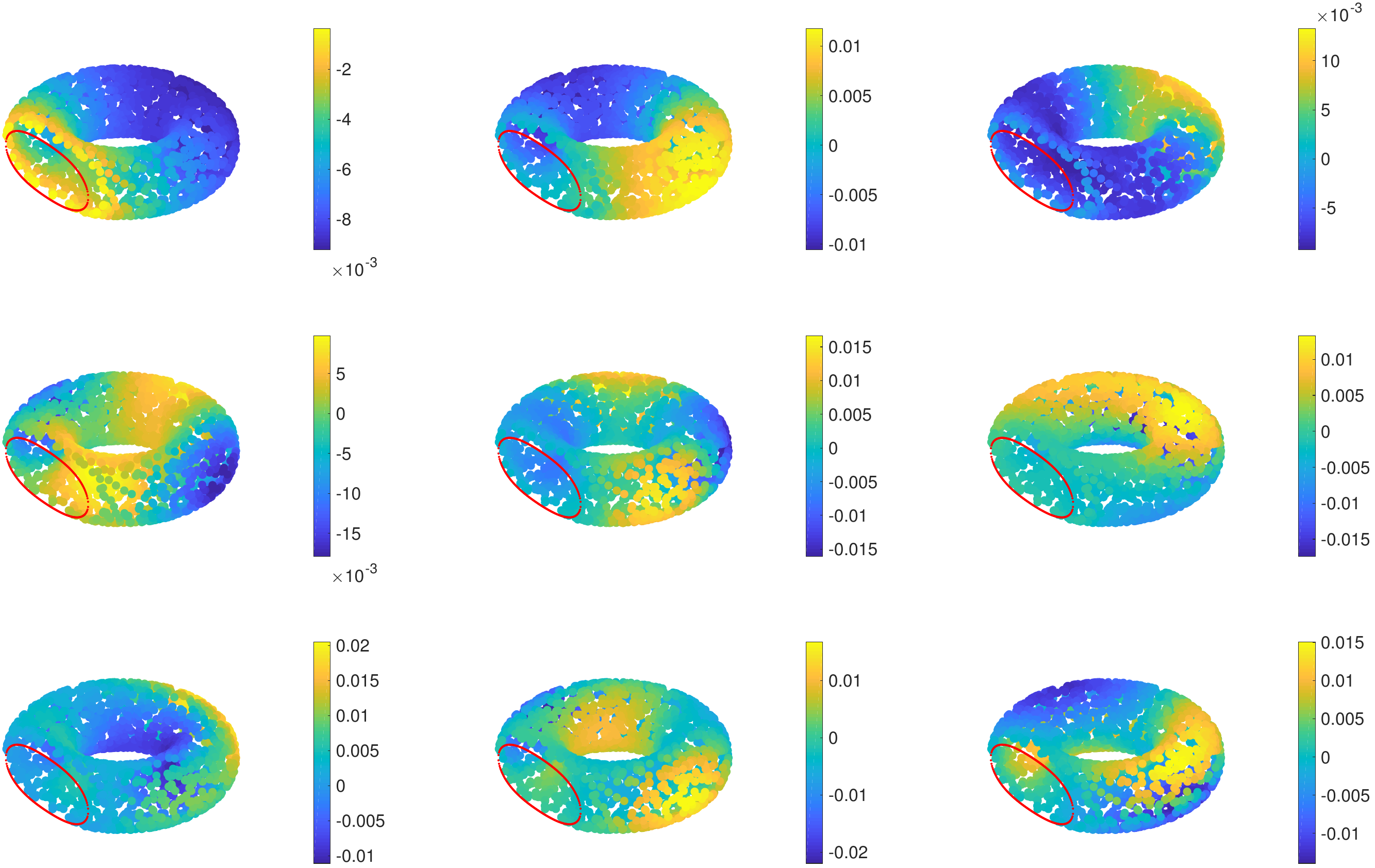}
\caption{The first $9$ eigenfunctions of the clipped LLE matrix for a point cloud sampled from the truncated torus $M_4$ are plotted from top left to bottom right. To enhance the visualization, the boundary of the truncated torus is colored by red. It is clear that all eigenfunctions are zero on the ``boundary''.}
\label{Figure:TorusModifiedEigenfunctions}
\end{figure}

\section{Discussion and Conclusion}\label{Section Concludion}

In this paper, we provide an exploration of LLE when the manifold has boundary.  We mention several interesting problems that we will explore in our future work.

First, the distribution of the LLE matrix eigenvalues under the null case has an interesting distribution behavior, which rings the bell of the interaction between kernel random matrix and random matrix theory. {Recently, the spectral behavior of graph Laplacian has been studied from the random matrix perspective in \cite{ding2020impact}. However, due to the non-symmetric nature of the LLE matrix, such an approach cannot be directly applied.} Understanding the behavior of LLE will pave the road toward statistical inference of unsupervised manifold learning.

Second, the potential Dirichlet boundary condition associated with the clipped LLE matrix and its relationship with the Neuman boundary condition associated with the GL suggest exploring the Dirichlet-to-Neumann map and Schur's complement as a future direction. {Moreover, from the practical perspective, when we do not know where is the wave region, we shall design an effective algorithm to determine all points from the wave region, so that we can clip the LLE matrix.}

Third, as is shown in Theorem \ref{Theorem:final}, LLE converges pointwisely to a mixed-type differential operator with degeneracy, which is an SL equation with a peculiar structure in the one-dimensional case. In other words, we have a degenerate mixed-type differential equation and the boundary condition is not known a priori. 
Understanding the spectral behavior of this operator is of interest on its own from the theoretical perspective and it might be necessary in order to explore the spectral behavior of LLE when there is a boundary. 

{Fourth, the spectral convergence of LLE is so far an open problem to our knowledge. To attack this problem, we shall again compare the original proposed $(W-I)^\top (W-I)$ and $W-I$ considered in this paper and \cite{Wu_Wu:2017}. Assume first that $\partial M=\emptyset$. 
We may learn from what has been done in the literature. To study the spectral convergence of DM \cite{trillos2018error,dunson2021spectral}, we need at least two pieces of information. The first one is the knowledge of the spectral behavior of the asymptotic differential operator, and the second one is the pointwise convergence result. The pointwise convergence result of $W-I$ has been studied in \cite{Wu_Wu:2017} when $\partial M=\emptyset$, which is generalized to the case $\partial M\neq\emptyset$ in this paper.
It might be intuitive to conclude that since the spectral behavior of the Laplace-Beltrami operator has been well known, when $\partial M=\emptyset$, we may easily obtain the spectral convergence of $W-I$. However, since the matrix $W-I$ is not symmetric, the associated integral operator of LLE is not self-adjoint. Thus, we cannot directly apply the same method in \cite{trillos2018error,dunson2021spectral} to prove the spectral convergence of $W-I$, and new tools are needed. 
Moreover, since the fourth order differential operator is involved in $(W-I)^\top (W-I)$ via a pointwise convergence analysis, and its spectral behavior is less well known, it is more challenging to study the spectral convergence of $(W-I)^\top (W-I)$. 
The situation is certainly more complicated when $\partial M\neq\emptyset$ since the associated kernel, the integral operator behavior, and the asymptotic differential operator are all different. Even for $W-I$, it is still an open problem as mentioned above as the third point, not to mention $(W-I)^\top(W-I)$ and if the nearest neighbor scheme is taken into consideration. A further systematic pointwise and spectral convergence study of $(W-I)^\top(W-I)$ is thus needed to advance the field. 
}

\section*{acknowledgment}
The authors acknowledge the fruitful discussion with Professor Jun Kitagawa about the boundary condition.

\appendix

\section{Examples for  Proposition \ref{Proposition 1}}\label{Appendix examples for propostion 1}
To show that it is possible  $\rho(W) = 1$, consider the following example.
Let $n=2m$, where $m \geq 2$ is an integer.  Suppose $\mathcal{X}=\{z_1, z_2, \cdots, z_n\}$ is a uniform grid of $S^{1} \subset \mathbb{R}^2$ so that $z_i=(\cos(\frac{2\pi (i-1)}{n}), \sin(\frac{2\pi (i-1)}{n}))$, where $i=1,\cdots, n$.  We choose $\epsilon$ so that $\mathcal{N}_k$ only contains two data points $(\cos(\frac{2\pi (i-2)}{n}), \sin(\frac{2\pi (i-2)}{n}))$ and $(\cos(\frac{2\pi i}{n}), \sin(\frac{2\pi i}{n}))$. Fix $z_k$, and $z_{k,1}$ and $z_{k,2}$ are the two data points in $\mathcal{N}_k$. Without loss of generality, we assume that $z_{k}=(0,0)$, $z_{k,1}=(a,b)$ and $z_{k,2}=(-a,b)$. Hence, $G_{n,k}$ at $z_{k}$ is 
\begin{align}
G_{n,k}=
\begin{bmatrix}
a& -a \\
b & b  \\
\end{bmatrix},
\end{align}
and the solution $y_k^\top=[{y}_{k,1}, {y}_{k,2}]$ to the regularized equation (\ref{Section2:wk}) with the regularizer $c>0$ satisfies
\begin{align}
\begin{bmatrix}
a^2+b^2+c& -a^2+b^2 \\
-a^2+b^2 & a^2+b^2+c  \\
\end{bmatrix}
\begin{bmatrix}
\bar{y}_1 \\
\bar{y}_2 \\
\end{bmatrix}=\begin{bmatrix}
1 \\
1 \\
\end{bmatrix}.
\end{align}
Therefore, we have ${y}_{k,1}={y}_{k,2}$, $w_k^\top=[1/2,1/2]$, and
\begin{equation}
W_{ki}= \left\{ \begin{array}{ll} 1/2 & \mbox{if $z_i=z_{k,j} \in \mathcal{N}_k$};\\ 0 & \mbox{otherwise}. \end{array} \right.
\end{equation}
Suppose $\lambda_0 \leq \lambda_1 \leq \cdots \leq \lambda_{n-1}$ are the eigenvalues of $W$. Then $\lambda_0=-1$, $\lambda_{n-1}=1$ and $\lambda_{2i-1}=\lambda_{2i}=\cos(\frac{\pi(m-i)}{m})$ for $i=1, \cdots, m-1$.

We provide another example to show that in general it is possible that $\rho(W) >1$. Consider a point cloud with ten points in $\mathbb{R}^3$, $(-0.56, -0.34, 1.03)$,\\ $(-0.51, 0.32,  -0.02)$, $(-0.53, -1.47,-0.57)$, $(1.34, 0.47,-0.15)$, $(1.01, -1.56, 1.22)$,\\ $(-0.55, -1,-0.07)$, $(0.09,-1.04, -0.2)$, $(-1.27,2.07,-0.9)$, $(1.26,-0.71,-1.2)$, and\\ $(1.46,0,0.61)$. The LLE matrix of this point cloud with 5 nearest neighbors and the regularizer $c=10^{-3}$ has an eigenvalue $-2.4233$.

\section{Technical lemmas for some geometric quantities}
In this section we collect several technical lemmas for some geometric quantities we will encounter in the proof. They might be also useful for other works when the manifold with boundary setup is considered. 

{For the manifold with boundary, denoted as $M$, with the isometric embedding $\iota$ into $\mathbb{R}^p$, we consider the extensions $\tilde{M}$ and $\tilde{\iota}$ introduced in Section \ref{Section Manifold Setup}. Let $\exp_x$ be the exponential map of $\tilde{M}$ at $x \in \tilde{M}$. Recall that $\exp_x$ is well defined over $\tilde{\iota}^{-1}(B_{\epsilon}^{\mathbb{R}^p}(\iota(x))\cap \iota(M))$ for any $x \in M$. For $x \in M$, we use $\Second_x$ to denote both the second fundamental form of $\iota(M)$ at $\iota(x)$ and  the second fundamental form of $\tilde{\iota}(M)$ at $\tilde{\iota}(x)$.  The first three lemmas are basic facts about $\exp_x$, the normal coordinate, and the volume form. The proof of these three lemmas can be found in \cite{Singer_Wu:2012}.}

\begin{lemma} \label{Lemma:1}
Fix $x\in M$. Consider the extension $\tilde{M}$. If we use the Cartesian coordinate to parametrize $T_{x}\tilde{M}$, the volume form has the following expansion
\begin{align}
dV=\bigg(&1-\frac{1}{6}\sum_{i,j=1}^d\texttt{Ric}_{x}(i,j) u_iu_j+O(u^3)\bigg) du,
\end{align}
where $u=\sum_{i=1}^d u_ie_i\in T_{x}\tilde{M}$, $\texttt{Ric}_{x}(i,j)=\texttt{Ric}_{x}(e_i,e_j)$.
\end{lemma}

\begin{lemma} \label{Lemma:2}
Fix $x\in M$.  Consider the  extensions $\tilde{M}$ and $\tilde{\iota}$. For $u\in T_x\tilde{M}$ with $\|u\|$ sufficiently small, we have the following Taylor expansion:
\begin{align}
\tilde{\iota} \circ \exp_x(u)-\tilde{\iota}(x)
=&\,\tilde{\iota}_{*}u+\frac{1}{2}\Second_x(u,u)+O(\|u\|^3).
\end{align}
\end{lemma}

In the next Lemma, we compare the geodesic distance and the Euclidean distance.

\begin{lemma} \label{Lemma:3}
Fix $x\in M$. Consider the  extensions $\tilde{M}$ and $\tilde{\iota}$.  If we use the polar coordinate $(t,\theta)\in [0,\infty)\times S^{d-1}$ to parametrize $T_{x}\tilde{M}$, when $t>0$ is sufficiently small and $\tilde{t}=\|\tilde{\iota} \circ \exp_x(\theta t)-\tilde{\iota}(x)\|_{\mathbb{R}^p}$, then
\begin{align}
\tilde{t} = &\,  t-\frac{1}{24} \|\Second_x(\theta,\theta)\|^2 t^3 +O(t^4)  \\
t = &\,  \tilde{t} + \frac{1}{24} \|\Second_x(\theta,\theta)\|^2 \tilde {t}^3 +O(\tilde{t}^4)\,,\nonumber
\end{align}
where $\theta \in S^{d-1} \subset T_x\tilde{M}$. 
\end{lemma}

The following lemma describes a parametrization of the boundary set. This parametrization is needed when we analyze the LLE matrix near the boundary.
\begin{lemma} \label{Lemma:3.5}
Fix $x \in M_{\epsilon}$. Consider the  extensions $\tilde{M}$ and $\tilde{\iota}$.
\begin{align}
&(\tilde{\iota} \circ \exp_x)^{-1}(B_{\epsilon}^{\mathbb{R}^p}(\iota(x))  \cap \iota(\partial M) ) \\
=&\,\Big\{\sum_{l=1}^du^l\partial_l \in T_x \tilde{M} \Big|\, (u^1, \cdots, u^{d-1}) \in K,\,  u^d=q(u^1, \cdots, u^{d-1})\Big\}\,,\nonumber
\end{align}
where 
\begin{align}
q(u^1, \cdots, u^{d-1})=\tilde{\epsilon}_x + \sum_{i,j=1}^{d-1} a_{ij}(x_\partial) u^iu^j+O(\|u\|^3),
\end{align}
and $a_{ij}(x_\partial)$ is the second fundamental form of the embedding of $\partial M$ in $M$ at $x_\partial$.
\end{lemma}
\begin{proof}
Note that $(\tilde{\iota} \circ \exp_x)^{-1}(B_{\epsilon}^{\mathbb{R}^p}(\iota(x))  \cap \iota(\partial M) )$ is a hypersurface with boundary in $T_x M$. 
Since $\partial M$ is smooth, by the implicit function theorem,  if $\epsilon$ is small enough,
\begin{align}
&(\tilde{\iota} \circ \exp_x)^{-1}(B_{\epsilon}^{\mathbb{R}^p}(\iota(x))  \cap \iota(\partial M) ) \\
=&\,\Big\{\sum_{l=1}^du^l\partial_l \in T_x \tilde{M}  \Big|\, (u^1, \cdots, u^{d-1}) \in K,\,  u^d=q(u^1, \cdots, u^{d-1})\Big\}\,\nonumber
\end{align}
for a smooth function $q$ of $u^1, \cdots, u^{d-1}$. By Taylor's expansion, we have
\begin{align}
q(u^1, \cdots, u^{d-1})=\tilde{\epsilon}_x + \sum_{i,j=1}^{d-1} a_{ij}(x) u^iu^j+O(\|u\|^3),
\end{align}
where the first order disappears since the tangent space of $(\tilde{\iota} \circ \exp_x)^{-1}(B_{\epsilon}^{\mathbb{R}^p}(\iota(x))  \cap \iota(\partial M) )$ in $T_x M$ at $\exp_x^{-1}(x_\partial)$ is perpendicular to $u_d$ direction by Gauss's lemma, and $a_{ij}(x)$ is the coefficient of the second order expansion. 
Due to the smoothness of the manifold, $a_{ij}(x)$ is smooth along the minimizing geodesic from $x_\partial$ to $x$. Also, when $x=x_\partial$, $a_{ij}(x_\partial)$ is the second fundamental form of the embedding of $\partial M$ in $M$ at $x_\partial$. 
Therefore, by another Taylor's expansion, $a_{ij}(x)=a_{ij}(x_\partial)+O(u_d)$, the conclusion follows.
\end{proof}

Next Lemma describes the discrepancy between $\int_{D_{\epsilon}(x)} f(u) du$ and  $\int_{\tilde{D}_{\epsilon}(x)} f(u) du$. Note that the order of the discrepancy does not dependent on the location of $x$.
\begin{corollary} \label{Lemma:4}
Fix $x\in M$. When $\epsilon>0$ is sufficiently small, we have 
\begin{equation}
\left|\int_{D_{\epsilon}(x)}  du-\int_{\tilde{D}_{\epsilon}(x)}  du\right|=O(\epsilon^{d+2}).
\end{equation}
\end{corollary}

\begin{proof}
Based on Lemma \ref{Lemma:3} and the definition of $\tilde{D}_{\epsilon}(x)$, the distance between the boundary of $D_{\epsilon}(x)$ and the boundary of $\tilde{D}_{\epsilon}(x)$ is of order $\epsilon^3$. The volume of the boundary $\tilde{D}_{\epsilon}(x)$ is of order $\epsilon^{d-1}$. Hence the volume difference between $\tilde{D}_{\epsilon}(x)$ and $D_{\epsilon}(x)$ is of order $\epsilon^{d-1} \cdot \epsilon^3=\epsilon^{d+2}$. The conclusion follows.
\end{proof}

\section{Technical lemmas for the kernel analysis}

To have a closer look at the kernel, we need the following quantities.
First, we introduce some notations. For $v\in\mathbb{R}^p$, denote
\begin{align} \label{vectornotation}
v=[\![v_1,\,v_2]\!]\in \mathbb{R}^p\,,
\end{align}
where $v_1\in \mathbb{R}^{d}$ forms the first $d$ coordinates of $v$ and $v_2\in\mathbb{R}^{p-d}$ forms the last $p-d$ coordinates of $v$. 
Thus, for $v=[\![v_1,\,v_2]\!]\in T_{\iota(x)}\mathbb{R}^p$, $v_1=J_{p,d}^\top v$ is the coordinate of the tangential component of $v$ on $\iota_*T_xM$ and $v_2=\bar{J}_{p,p-d}^\top v$ is the coordinate of the normal component of $v$ associated with a chosen basis of the normal bundle. Define 
\[
\mathfrak{N}_{ij}(x):=\bar{J}_{p,p-d}^\top \Second_{ij}(x). 
\]
Note that $\mathfrak{N}_{ij}(x)=\mathfrak{N}_{ji}(x)$.

\begin{definition}[Moments]\label{Definition Moments}
For $x\in M$, consider the following moments that capture the geometric asymmetry: 
\begin{align}
\mu_{v}(x,\epsilon):=\int_{\tilde{D}_{\epsilon}(x)}\prod_{i=1}^d u_i^{v_i} du\,,\nonumber
\end{align}
where $v=[v_1,\ldots,v_d]^\top$ describes the moment order.
\end{definition}

In next lemma, we quantitatively describe all the moments up to the third order. This Lemma tells us that when $x\notin M_\epsilon$ (when $x$ is far away from the boundary), all odd order moments disappear due to the symmetry of the integration domain. However, when $x\in M_\epsilon$, it no longer holds -- the integration domain becomes asymmetric, and the odd moments no longer disappear. Therefore, we can show that $\mu_0(x,\epsilon)$, $\mu_{e_d}(x,\epsilon)$, $\mu_{2e_i}(x,\epsilon)$  and $\mu_{2e_i+e_d}(x,\epsilon)$  are all the non-trivial moments needed in analyzing LLE. The proof follows from the symmetry argument and a straightforward integration, so we omit it here. 

\begin{lemma} \label{Lemma:5} [Symmetry] Suppose $\epsilon$ is sufficiently small. Then, the moments up to order three can be quantitatively described as follows. In fact, $\mu_{0}(x,\epsilon)$, $\mu_{e_d}(x,\epsilon)$,  $\mu_{2e_i}$ and $\mu_{2e_i+e_d}(x,\epsilon)$  for all $i=1, \cdots, d$ are the only non-trivial moments. And they are continuous functions of $x$ on $M$. Define $\frac{|S^{d-2}|}{d-1}=1$ when $d=1$, we have 
\begin{enumerate}
\descitem{Zero order moment, $\mu_0$}
If $x \in M_{\epsilon}$, $\mu_{0}$ is an increasing function of $\tilde{\epsilon}_x $  and 
\begin{equation}
\mu_{0}(x,\epsilon)= \frac{|S^{d-1}|}{2d}\epsilon^d+\int_{0}^{\tilde{\epsilon}_x } \frac{|S^{d-2}|}{d-1}(\epsilon^2-h^2)^{\frac{d-1}{2}}dh+O(\epsilon^{d+1}).\nonumber
\end{equation}
If  $x \not \in M_{\epsilon}$, then 
 \begin{equation}
\mu_0(x,\epsilon)=\frac{|S^{d-1}|}{d}\epsilon^d.\nonumber
\end{equation}
In general, the following bound holds for $\mu_0(x,\epsilon) $:
\begin{equation}
\frac{|S^{d-1}|}{2d}\epsilon^d+O(\epsilon^{d+1}) \leq \mu_0(x,\epsilon) \leq \frac{|S^{d-1}|}{d}\epsilon^d. \nonumber
\end{equation}

\descitem{First order moment, $\mu_{e_i}$}
If $x \in M_{\epsilon}$, $\mu_{e_d}$ is an increasing function of $\tilde{\epsilon}_x $  and 
\begin{equation}
\mu_{e_d}(x,\epsilon)=-\frac{|S^{d-2}|}{d^2-1}(\epsilon^2-\tilde{\epsilon}_x ^2)^{\frac{d+1}{2}}+O(\epsilon^{d+2}).\nonumber
\end{equation}
If $x \not \in M_{\epsilon}$, then 
\begin{equation}
\mu_{e_d}(x,\epsilon)=0.\nonumber
\end{equation}
In general, $\mu_{e_d}(x,\epsilon)$ is of order $\epsilon^{d+1}$. 
For the rest of the moments, $\mu_{e_i}=0$, for $i=1, \cdots, d-1$.

\descitem{Second order moment, $\mu_{e_i+e_j}$}
If $x \in M_{\epsilon}$, $\mu_{2e_i}$ is an increasing function of $\tilde{\epsilon}_x $ for $i=1,\cdots,d$. We have
\begin{equation}
\mu_{2e_i}(x,\epsilon)=\frac{|S^{d-1}|}{2d(d+2)}\epsilon^{d+2}+\int_{0}^{\tilde{\epsilon}_x } \frac{|S^{d-2}|}{d^2-1}(\epsilon^2-h^2)^{\frac{d+1}{2}} dh+O(\epsilon^{d+3}), \nonumber
\end{equation}
for $i=1,\cdots,d-1$, and
\begin{equation}
\mu_{2e_d}(x,\epsilon)=\frac{|S^{d-1}|}{2d(d+2)}\epsilon^{d+2}+\int_{0}^{\tilde{\epsilon}_x } \frac{|S^{d-2}|}{d-1}(\epsilon^2-h^2)^{\frac{d-1}{2}}h^2 dh+O(\epsilon^{d+3})\,. \nonumber
\end{equation}
If $x \not \in M_{\epsilon}$, then 
\begin{equation}
\mu_{2e_i}(x,\epsilon)=\frac{|S^{d-1}|}{d(d+2)}\epsilon^{d+2}.\nonumber
\end{equation}
In general, the following bounds hold for $\mu_{2e_i}(x,\epsilon)$, where $i=1,\cdots, d$:
\begin{equation}
\frac{|S^{d-1}|}{2d(d+2)}\epsilon^{d+2}+O(\epsilon^{d+3}) \leq \mu_{2e_i}(x,\epsilon) \leq \frac{|S^{d-1}|}{d(d+2)}\epsilon^{d+2}, \nonumber
\end{equation}
For the rest of the moments, $\mu_{e_i+e_j}=0$, whenever $i \not= j$.

\descitem{Third order moment, $\mu_{e_i+e_j+e_k}$}
 If $x \in M_{\epsilon}$, $\mu_{2e_i+e_d}$ is an increasing function of $\tilde{\epsilon}_x $  and 
\begin{equation}
\mu_{2e_i+e_d}(x,\epsilon)=-\frac{|S^{d-2}|}{(d^2-1)(d+3)}(\epsilon^2-\tilde{\epsilon}_x ^2)^{\frac{d+3}{2}}+O(\epsilon^{d+4}), \nonumber
\end{equation}
for $i=1, \cdots d-1$, and 
\begin{equation}
\mu_{3e_d}(x,\epsilon)=-\frac{|S^{d-2}|}{(d^2-1)(d+3)}(\epsilon^{2}-\tilde{\epsilon}_x ^2)^{\frac{d+1}{2}}(2\epsilon^2+(d+1)\tilde{\epsilon}_x ^2)+O(\epsilon^{d+4}). \nonumber
\end{equation}
If $x \not \in M_{\epsilon}$, then 
\begin{equation}
\mu_{3e_d}(x,\epsilon)=0.\nonumber
\end{equation}
In general, $\mu_{2e_i+e_d}(x,\epsilon)$ is of order $\epsilon^{d+3}$. And $\mu_{e_i+e_j+e_k}=0$, for the rest of the cases.

\end{enumerate}
\end{lemma}

Below, we relate those non-trivial moments described in the previous lemma to those $\sigma$ functions in Definition \ref{Lemma:5:summary}. The proof follows from a straightforward change of variable, so we omit the details.

\begin{corollary} \label{sigma and mu}
The relationship between the non-trivial moments in Lemma \ref{Lemma:5} and the functions $\sigma$ defined in Definition \ref{Lemma:5:summary} satisfies:
\begin{align}
\mu_{0}(x,\epsilon)&=\sigma_{0}(\tilde{\epsilon}_x )\epsilon^{d}+O(\epsilon^{d+1})\nonumber\\
\mu_{e_d}(x,\epsilon)&=\sigma_{1,d}(\tilde{\epsilon}_x )\epsilon^{d+1}+O(\epsilon^{d+2}).\nonumber\\
\mu_{2e_i}(x,\epsilon)&=\sigma_{2}(\tilde{\epsilon}_x )\epsilon^{d+2}+O(\epsilon^{d+3}),\nonumber
\end{align}
for $i=1, \cdots d-1$, and
\begin{align}
\mu_{2e_d}(x,\epsilon)=\sigma_{2,d}(\tilde{\epsilon}_x )\epsilon^{d+2}+O(\epsilon^{d+3})\,.\nonumber
\end{align}
Moreover,
\begin{align}
\mu_{2e_i+e_d}(x,\epsilon)=\sigma_{3}(\tilde{\epsilon}_x )\epsilon^{d+3}+O(\epsilon^{d+4}),\nonumber
\end{align}
for $i=1, \cdots d-1$, and
\begin{align}
\mu_{3e_d}(x,\epsilon)=\sigma_{3,d}(\tilde{\epsilon}_x )\epsilon^{d+3}+O(\epsilon^{d+4}).\nonumber
\end{align}
\end{corollary}

Next, we prove the following lemma about the ratio between the volumes of $d-1$ sphere and $d-2$ sphere. We need it to study the relation between different $\sigma$ functions later. 
\begin{lemma} \label{ratio of spheres}
For $d\in \mathbb{N}$, we have
\begin{align}
\frac{(d+1)^2(d+3)}{8d^2(d+2)^2}<\frac{|S^{d-2}|^2 }{(d-1)^2|S^{d-1}|^2}<\frac{(d+1)^2}{4d^2(d+2)}\,,
\end{align}
where $\frac{|S^{d-2}|}{d-1}$ is defined as $1$ when $d=1$.
\end{lemma}
\begin{proof} 
The inequality can be verified by a straightforward calculation for $d \leq 6$.

Next, we show that $\frac{|S^{d-2}|^2 }{|S^{d-1}|^2}<\frac{(d^2-1)^2}{4d^2(d+2)}$ for $d >6$. Note that $\frac{|S^{d-2}|^2}{ |S^{d-1}|^2}=\frac{\Gamma(\frac{d}{2})^2}{ \pi \Gamma(\frac{d-1}{2})^2}$. Hence, it suffices to prove $\frac{\Gamma(\frac{d}{2})^2}{ \Gamma(\frac{d-1}{2})^2} < \frac{\pi (d^2-1)^2}{4d^2(d+2)}$. In \cite{kershaw1983some}, it is proved that for all $x>0$ and $0<s<1$,
\begin{align}\label{Gamma bounds}
(x+\frac{s}{2})^{1-s} <\frac{ \Gamma(x+1)}{\Gamma(x+s)}<e^{(1-s)\psi(x+\frac{1+s}{2})},
\end{align}
where $\psi(y)=\frac{\Gamma'(y)}{\Gamma(y)}$. Choose $x=\frac{d}{2}-1$ and $s=\frac{1}{2}$, then $\frac{\Gamma(\frac{d}{2})^2}{\Gamma(\frac{d-1}{2})^2 }<e^{\psi(\frac{d}{2}-\frac{1}{4})}$. Hence, it suffice to show that 
\begin{equation}
e^{\psi(\frac{d}{2}+\frac{1}{4})}<\frac{\pi (d^2-1)^2}{4d^2(d+2)}.
\end{equation}
Actually, we have $e^{\psi(y)}<y$ for any postive $y$. The conclusion follows by verifying $\frac{d}{2}+\frac{1}{4}<\frac{\pi (d^2-1)^2}{4d^2(d+2)}$ for $d>6$.

At last, we show that $\frac{(d^2-1)^2(d+3)}{8d^2(d+2)^2}<\frac{|S^{d-2}|^2 }{|S^{d-1}|^2}$, which is equivalent to $\frac{\pi (d^2-1)^2(d+3)}{8d^2(d+2)^2}<\frac{\Gamma(\frac{d}{2})^2}{ \Gamma(\frac{d-1}{2})^2}$. By  \eqref{Gamma bounds}, with $x=\frac{d}{2}-1$ and $s=\frac{1}{2}$, we have $\frac{d}{2}-\frac{3}{4}<\frac{\Gamma(\frac{d}{2})^2}{ \Gamma(\frac{d-1}{2})^2}$. The conclusion follows by verifying $\frac{\pi (d^2-1)^2(d+3)}{8d^2(d+2)^2}<\frac{d}{2}-\frac{3}{4}$  for $d$ large.

\end{proof}

We calculate some major ingredients that we are going to use in the proof of the main theorem. Specifically, we calculate the first two order terms in $\mathbb{E}[\chi_{B_{\epsilon}^{\mathbb{R}^p}(\iota(x))}(\iota(X))]$, $\mathbb{E}[(f(X)-f(x))\chi_{B_{\epsilon}^{\mathbb{R}^p}(\iota(x))}(\iota(X))]$, and the first two order terms in the tangent component of $\mathbb{E}[(\iota(X)-\iota(x))\chi_{B_{\epsilon}^{\mathbb{R}^p}(\iota(x))}(\iota(X))]$ and $\mathbb{E}[(\iota(X)-\iota(x))(f(X)-f(x))\chi_{B_{\epsilon}^{\mathbb{R}^p}(\iota(x))}(\iota(X))]$.
This long Lemma is the generalization of  \cite[Lemma B.5]{Wu_Wu:2017} to the boundary. In particularly, when $x\notin M_\epsilon$, we recover \cite[Lemma B.5]{Wu_Wu:2017}.

\begin{lemma} \label{Lemma:6}
Fix $x \in  M$ and $f\in C^3(M)$. When $\epsilon>0$ is sufficiently small, the following expansions hold.
\begin{enumerate} 
\item $\mathbb{E}[\chi_{B_{\epsilon}^{\mathbb{R}^p}(\iota(x))}(\iota(X))]$ satisfies
\begin{equation} 
\mathbb{E}[\chi_{B_{\epsilon}^{\mathbb{R}^p}(\iota(x))}(\iota(X))]= P(x)\mu_0(x,\epsilon) +\partial_d P(x)\mu_{e_d}(x,\epsilon)+O(\epsilon^{d+2})\,.   \nonumber
\end{equation} 
\item $\mathbb{E}[(f(X)-f(x))\chi_{B_{\epsilon}^{\mathbb{R}^p}(\iota(x))}(\iota(X))] $ satisfies 
\begin{align}
 \mathbb{E}[(f(X)-f(x))&\chi_{B_{\epsilon}^{\mathbb{R}^p}(\iota(x))}(\iota(X))] =\,   P(x)  \partial_d f(x) \mu_{e_d}(x,\epsilon)\nonumber\\
&+\sum_{i=1}^d(\frac{P(x)}{2}\partial^2_{ii}f(x) +\partial_i f(x) \partial_i P(x) )\mu_{2e_i}(x,\epsilon)+O(\epsilon^{d+3})\,.\nonumber
\end{align} 
\item The vector $\mathbb{E}[(\iota(X)-\iota(x))\chi_{B_{\epsilon}^{\mathbb{R}^p}(\iota(x))}(\iota(X))]$ satisfies
\begin{equation}
\mathbb{E}[(\iota(X)-\iota(x))\chi_{B_{\epsilon}^{\mathbb{R}^p}(\iota(x))}(\iota(X))]=[\![v_1,v_2]\!]\,, \nonumber
\end{equation}
where 
\begin{align}
v_1=&\,  P(x) \mu_{e_d}(x,\epsilon) J_{p,d}^\top e_d + \sum_{i=1}^d \big(\partial_iP(x) \mu_{2e_i}(x,\epsilon)\big)J_{p,d}^\top e_i +O(\epsilon^{d+3})  \nonumber \\
v_2 = &\, \frac{P(x)}{2} \sum_{i=1}^d \mathfrak{N}_{ii}(x) \mu_{2e_i}(x,\epsilon)+ O(\epsilon^{d+3}).\nonumber 
\end{align}
\item The vector $\mathbb{E}[(\iota(X)-\iota(x))(f(X)-f(x))\chi_{B_{\epsilon}^{\mathbb{R}^p}(\iota(x))}(\iota(X))]$ satisfies
\begin{equation}
\mathbb{E}[(\iota(X)-\iota(x))(f(X)-f(x))\chi_{B_{\epsilon}^{\mathbb{R}^p}(\iota(x))}(\iota(X))]=[\![v_1,v_2]\!]\,,\nonumber
\end{equation}
where
\begin{align}
v_1=&\,P(x) \sum_{i=1}^d \big(\partial_i f(x) \mu_{2e_i}(x, \epsilon) \big) \,J_{p,d}^\top e_i \nonumber\\
&+ \sum_{i=1}^{d-1}\big[ \partial_i f(x)\partial_d P(x)+ \partial_d f(x)\ \partial_i P(x)+P(x)\partial^2_{id} f(x) \big]\mu_{2e_i+e_d}(x,\epsilon)  \,J_{p,d}^\top e_i \nonumber\\
&+\sum_{i=1}^{d}\Big(\big[ \partial_i f(x)\partial_i P(x) +\frac{P(x)}{2}\partial^2_{ii} f(x)\big]\mu_{2e_i+e_d}(x,\epsilon)   \Big)\,J_{p,d}^\top e_d+O(\epsilon^{d+4}), \nonumber \\
v_2 =&\, P(x) \sum_{i=1}^{d-1} \partial_i f(x) \mathfrak{N}_{id}(x) \mu_{2e_i+e_d} (x, \epsilon)\nonumber\\
&+ \frac{P(x)}{2}\partial_d f(x)\sum_{i=1}^d \mathfrak{N}_{ii}(x) \mu_{2e_i+e_d}(x, \epsilon)+O(\epsilon^{d+4}).\nonumber
\end{align}
\end{enumerate}
\end{lemma}

\begin{proof} 
{We use the extensions $\tilde{M}$ and $\tilde{\iota}$ introduced in Section \ref{Section Manifold Setup}. For any $x\in M$,  let $\exp_x$ be the exponential map of $\tilde{M}$. }
First, we calculate $\mathbb{E}[\chi_{B_{\epsilon}^{\mathbb{R}^p}(\iota(x))}(\iota(X))]$. 
\begin{align}
& \mathbb{E}[\chi_{B_{\epsilon}^{\mathbb{R}^p}(\iota(x))}(\iota(X))] \\
 = &\, \int_{D_{\epsilon}(x)} \big(P(x)+\sum_{i=1}^d\partial_i P(x)u_i+O(u^2)\big)\big(1-\sum_{i,j=1}^d\frac{1}{6}\texttt{Ric}_{x}(i,j)u_iu_j+O(u^3)\big) du \nonumber \\
= &\, P(x)\int_{\tilde{D}_{\epsilon}(x)} du+\int_{\tilde{D}_{\epsilon}(x)} \sum_{i=1}^d\partial_i P(x)u_idu+O(\epsilon^{d+2}) \nonumber  \\
= &\, P(x)\mu_0(x,\epsilon) +\partial_d P(x)\mu_{e_d}(x,\epsilon)+O(\epsilon^{d+2}) \,, \nonumber
\end{align}
where the second equality holds by applying Corollary \ref{Lemma:4} that the error of changing domain from $D_{\epsilon}(x)$ to $\tilde{D}_{\epsilon}(x)$ is of order $\epsilon^{d+2}$. We use Lemma \ref{Lemma:5} in the last step. Note that $P(x)$ is bounded away from $0$.

Second, we calculate $\mathbb{E}[(f(X)-f(x))\chi_{B_{\epsilon}^{\mathbb{R}^p}(\iota(x))}(\iota(X))]$. Note that when $\epsilon$ is sufficiently small, we have
\begin{align}
f \circ \exp_x(u) -f(x) =\sum_{i=1}^d\partial_i f(x)u_i+\frac{1}{2}\sum_{i,j=1}^d\partial^2_{ij}f(x)u_iu_j+O(u^3),
\end{align}
which is of order $\epsilon$ for $u\in D_{\epsilon}(x)$. By a direct expansion, we have
\begin{align}
&\mathbb{E}[(f(X)-f(x))\chi_{B_{\epsilon}^{\mathbb{R}^p}(\iota(x))}(\iota(X))]\\
= &\, \int_{D_{\epsilon}(x)} (\sum_{i=1}^d\partial_i f(x) u_i+\frac{1}{2}\sum_{i,j=1}^d \partial^2_{ij} f(x) u_iu_j+O(u^3)) (P(x)+\sum_{i=1}^d \partial_i P(x) u_i+O(u^2)) \nonumber \\
&\qquad\times (1-\sum_{i,j=1}^d\frac{1}{6}\texttt{Ric}_{x}(i,j)u_iu_j+O(u^3)) du \,,\nonumber 
\end{align}
which by Corollary \ref{Lemma:4} and Lemma \ref{Lemma:5} becomes
\begin{align}
&\,\int_{D_{\epsilon}(x)}  \big[ P(x) \sum_{i=1}^d \partial_i f(x) u_i  +\frac{P(x)}{2} \sum_{i,j=1}^d\partial^2_{ij}f(x) u_iu_j +\sum_{i=1}^d \partial_i f(x) u_i \sum_{j=1}^d \partial_j P(x) u_j +O(u^3) \big] du \nonumber \\
=& \, P(x)  \partial_d f(x) \int_{\tilde{D}_{\epsilon}(x)} u_d du +\sum_{i=1}^d(\frac{P(x)}{2}\partial^2_{ii}f(x) +\partial_i f(x) \partial_i P(x) )\int_{\tilde{D}_{\epsilon}(x)}u_i^2du+O(\epsilon^{d+3}) \nonumber \\
=& \, P(x)  \partial_d f(x) \mu_{e_d}(x,\epsilon)+\sum_{i=1}^d(\frac{P(x)}{2}\partial^2_{ii}f(x) +\partial_i f(x) \partial_i P(x) )\mu_{2e_i}(x,\epsilon)+O(\epsilon^{d+3})\,. \nonumber
\end{align}
Note that the leading term in the integral is of order $\epsilon$, so the error of changing the domain from $D_{\epsilon}(x)$ to $\tilde{D}_{\epsilon}(x)$ is of order $\epsilon^{d+3}$. 

Third, by a direct expansion, we have
\begin{align}
&\mathbb{E}[(\iota(X)-\iota(x))\chi_{B_{\epsilon}^{\mathbb{R}^p}(\iota(x))}(\iota(X))]  \\
= &\, \int_{D_{\epsilon}(x)} (\tilde{\iota}_*u+\frac{1}{2}\Second_{x}(u,u)+O(u^3)) (P(x)+\sum_{i=1}^d \partial_i P(x) u_i+O(u^2)) \nonumber \\
&\qquad\times (1-\sum_{i,j=1}^d\frac{1}{6}\texttt{Ric}_{x}(i,j)u_iu_j+O(u^3))du \nonumber\,,
\end{align}
which is a vector in $\mathbb{R}^p$.
We then find the tangential part and the normal part of $\mathbb{E}[(\iota(X)-\iota(x))\chi_{B_{\epsilon}^{\mathbb{R}^p}(\iota(x))}(\iota(X))] $ respectively. 
The tangential part is 
\begin{align}
&\int_{D_{\epsilon}(x)} (\tilde{\iota}_*u+O(u^3)) (P(x)+\sum_{i=1}^d \partial_i P(x) u_i+O(u^2))  \\
&\qquad \times (1-\sum_{i,j=1}^d\frac{1}{6}\texttt{Ric}_{x}(i,j) u_iu_j+O(u^3))du \nonumber\\
=&\,\int_{\tilde{D}_{\epsilon}(x)} (\tilde{\iota}_*u+O(u^3)) (P(x)+\sum_{i=1}^d \partial_i P(x) u_i+O(u^2))\nonumber  \\
&\qquad \times (1-\sum_{i,j=1}^d\frac{1}{6}\texttt{Ric}_{x}(i,j) u_iu_j+O(u^3))du+O(\epsilon^{d+3})\nonumber\,,
\end{align}
where the equality holds by Corollary \ref{Lemma:4}.  
Similarly, by Corollary \ref{Lemma:4}, the normal part is
\begin{align}
&\int_{D_{\epsilon}(x)} (\frac{1}{2}\Second_{x}(u,u)+O(u^3)) (P(x)+\sum_{i=1}^d \partial_i P(x) u_i+O(u^2))  \\
&\qquad\times (1-\sum_{i,j=1}^d\frac{1}{6}\texttt{Ric}_{x}(i,j) u_iu_j+O(u^3))du \nonumber\\
=&\int_{\tilde{D}_{\epsilon}(x)} (\frac{1}{2}\Second_{x}(u,u)+O(u^3)) (P(x)+\sum_{i=1}^d \partial_i P(x) u_i+O(u^2))  \nonumber\\
&\qquad\times (1-\sum_{i,j=1}^d\frac{1}{6}\texttt{Ric}_{x}(i,j) u_iu_j+O(u^3))du+O(\epsilon^{d+4})\nonumber
\end{align}
since the leading term $P(x)\Second_{x}(u,u)$ is of order $\epsilon^2$ on $D_{\epsilon}(x)$.
As a result, by putting the tangent part and normal part together, $\mathbb{E}[(\iota(X)-\iota(x))\chi_{B_{\epsilon}^{\mathbb{R}^p}(\iota(x))}(\iota(X))]=[\![v_1,v_2]\!]$, where 
\begin{align}
 v_1 =&\,J_{p,d}^\top \Big[ P(x)  \int_{\tilde{D}_{\epsilon}(x)} \tilde{\iota}_*u du +\int_{\tilde{D}_{\epsilon}(x)}\tilde{\iota}_*u\sum_{i=1}^d \partial_i P(x)u_i du + O(\epsilon^{d+3}) \Big]\\
 = &\, \bigg( P(x) \int_{\tilde{D}_{\epsilon}(x)} u_d  du  \bigg) J_{p,d}^\top e_d + \sum_{i=1}^d \bigg(\partial_iP(x) \int_{\tilde{D}_{\epsilon}(x)}u_i^2 du \bigg)J_{p,d}^\top e_i +O(\epsilon^{d+3}) \nonumber \\
 = &\,  P(x) \mu_{e_d}(x,\epsilon)  J_{p,d}^\top e_d + \sum_{i=1}^d \partial_iP(x) \mu_{2e_i}(x,\epsilon)J_{p,d}^\top e_i +O(\epsilon^{d+3}) \nonumber 
 \end{align}
 and
\begin{align}
v_2 =&\, \frac{P(x)}{2}  \bar J_{p,p-d}^\top \int_{\tilde{D}_{\epsilon}(x)} \Second_{x}(u,u) du + O(\epsilon^{d+3}) = \frac{P(x)}{2} \sum_{i=1}^d \mathfrak{N}_{ii}(x) \mu_{2e_i}(x,\epsilon)+ O(\epsilon^{d+3}). \nonumber
\end{align}
Finally, we evaluate $\mathbb{E}[(\iota(X)-\iota(x))(f(X)-f(x))\chi_{B_{\epsilon}^{\mathbb{R}^p}(\iota(x))}(\iota(X))]$ and then find the tangential part and the normal part. By a direct expansion, 
\begin{align}
&\mathbb{E}[(\iota(X)-\iota(x))(f(X)-f(x))\chi_{B_{\epsilon}^{\mathbb{R}^p}(\iota(x))}(\iota(X))]  \\
=&\, \int_{D_{\epsilon}(x)} (\tilde{\iota}_*u+\frac{1}{2}\Second_{x}(u,u)+O(u^3)) (\sum_{i=1}^d\partial_i f(x) u_i+\frac{1}{2}\sum_{i,j=1}^d \partial^2_{ij} f(x) u_iu_j+O(u^3)) \nonumber \\
& \qquad\times \big(P(x)+\sum_{i=1}^d \partial_i P(x) u_i+O(u^2)\big)\big(1-\sum_{i,j=1}^d\frac{1}{6}\texttt{Ric}_{x}(i,j)u_iu_j+O(u^3)\big)du \nonumber.
\end{align}
The tangential part is 
\begin{align}
& \int_{D_{\epsilon}(x)} (\tilde{\iota}_*u+O(u^3)) \big(\sum_{i=1}^d\partial_i f(x) u_i+\frac{1}{2}\sum_{i,j=1}^d \partial^2_{ij} f(x) u_iu_j+O(u^3)\big)  \\
&\qquad\times \big(P(x)+\sum_{i=1}^d \partial_i P(x) u_i+O(u^2)\big)\big(1-\sum_{i,j=1}^d\frac{1}{6}\texttt{Ric}_{x}(i,j) u_iu_j+O(u^3)\big)du.\nonumber
\end{align}
The leading term $P(x)\tilde{\iota}_*u\sum_{i=1}^d\partial_i f(x) u_i$ is of order $\epsilon^2$ on $D_{\epsilon}(x)$, therefore the error of changing domain from $D_{\epsilon}(x)$ to $\tilde{D}_{\epsilon}(x)$ is of order $\epsilon^{d+4}$.
The normal part is 
\begin{align}
& \int_{D_{\epsilon}(x)} (\frac{1}{2}\Second_{x}(u,u)+O(u^3)) \big(\sum_{i=1}^d\partial_i f(x) u_i+\frac{1}{2}\sum_{i,j=1}^d \partial^2_{ij} f(x) u_iu_j+O(u^3)\big)  \\
&\qquad\times  \big(P(x)+\sum_{i=1}^d \partial_i P(x) u_i+O(u^2)\big)\big(1-\sum_{i,j=1}^d\frac{1}{6}\texttt{Ric}_{x}(i,j) u_iu_j+O(u^3)\big)du.\nonumber
\end{align}
The leading term $P(x)\Second_{x}(u,u)\sum_{i=1}^d\partial_i f(x) u_i$ is of order $\epsilon^3$ on $D_{\epsilon}(x)$. Therefore, the error of changing domain from $D_{\epsilon}(x)$ to $\tilde{D}_{\epsilon}(x)$ is of order $\epsilon^{d+5}$.
Putting the above together, $\mathbb{E}[(\iota(X)-\iota(x))(f(X)-f(x))\chi_{B_{\epsilon}^{\mathbb{R}^p}(\iota(x))}(\iota(X))]=[\![v_1,v_2]\!]$, where by the symmetry of $\tilde{D}_{\epsilon}(x)$ we have
\begin{align}
v_1=& \,J_{p,d}^\top \bigg[P(x) \int_{\tilde{D}_{\epsilon}(x)} \tilde{\iota}_*u\sum_{i=1}^d\partial_i f(x) u_i du +\int_{\tilde{D}_{\epsilon}(x)} \tilde{\iota}_*u \sum_{i=1}^d\partial_i f(x) u_i \sum_{j=1}^d \partial_j P(x) u_j du  \\
&+ \frac{P(x)}{2} \int_{\tilde{D}_{\epsilon}(x)} \tilde{\iota}_*u \sum_{i,j=1}^d \partial^2_{ij} f(x) u_iu_j du +O(\epsilon^{d+4})\bigg] \nonumber \\
=&\, P(x) \sum_{i=1}^d \bigg(\partial_i f(x) \int_{\tilde{D}_{\epsilon}(x)}u_i^2 du \bigg) \,J_{p,d}^\top e_i \nonumber\\
&+ \sum_{i=1}^{d-1}\big[ \partial_i f(x)\partial_d P(x)+ \partial_d f(x)\ \partial_i P(x)+P(x)\partial^2_{id} f(x) \big]\int_{\tilde{D}_{\epsilon}(x)} u_i^2u_d du  \,J_{p,d}^\top e_i \nonumber\\
&+\sum_{i=1}^{d}\bigg(\big[ \partial_i f(x)\partial_i P(x) +\frac{P(x)}{2}\partial^2_{ii} f(x)\big]\int_{\tilde{D}_{\epsilon}(x)} u_i^2u_d du \bigg)\,J_{p,d}^\top e_d+O(\epsilon^{d+4}) \nonumber \\
=&\, P(x) \sum_{i=1}^d \partial_i f(x) \mu_{2e_i}(x, \epsilon)  \,J_{p,d}^\top e_i \nonumber\\
&+ \sum_{i=1}^{d-1}\big[ \partial_i f(x)\partial_d P(x)+ \partial_d f(x)\ \partial_i P(x)+P(x)\partial^2_{id} f(x) \big]\mu_{2e_i+e_d}(x,\epsilon)  \,J_{p,d}^\top e_i \nonumber \\
&+\sum_{i=1}^{d}\big[ \partial_i f(x)\partial_i P(x) +\frac{P(x)}{2}\partial^2_{ii} f(x)\big]\mu_{2e_i+e_d}(x,\epsilon)   \,J_{p,d}^\top e_d+O(\epsilon^{d+4})\,,\nonumber
\end{align}
and
\begin{align}
v_2 =&\, \frac{P(x)}{2}  \bar J_{p,p-d}^\top \sum_{i=1}^d \partial_i f(x) \int_{\tilde{D}_{\epsilon}(x)}\Second_{x}(u,u)u_i du+O(\epsilon^{d+4}) \nonumber\\
=&\,  P(x) \sum_{i=1}^{d-1} \partial_i f(x) \mathfrak{N}_{id}(x) \mu_{2e_i+e_d} (x, \epsilon)+ \frac{P(x)}{2}\partial_d f(x)\sum_{i=1}^d \mathfrak{N}_{ii}(x) \mu_{2e_i+e_d}(x, \epsilon)+O(\epsilon^{d+4}). \nonumber
\end{align}
\end{proof}

\section{Structure of the local covariance matrix under the manifold setup}\label{AppendixSectionCovariance}

In this section we provide detailed analysis for the local covariance matrix $C_x=\mathbb{E}[(\iota(X)-\iota(x))(\iota(X)-\iota(x))^\top \chi_{B_{\epsilon}^{\mathbb{R}^p}(\iota(x))}(\iota(X))]$. This Lemma could be viewed as the generalization of \cite[Proposition 3.2]{Wu_Wu:2017} in the sense that when $x\notin M_\epsilon$, the result is reduced to that of \cite[Proposition 3.2]{Wu_Wu:2017}. To handle the boundary effect, we only need to calculate the first two order terms in eigenvalues and orthonormal eigenvectors of $C_x$.

\begin{lemma} \label{Lemma:7}
Fix  $x \in M$. Suppose that $\texttt{rank}(C_x)=r$, there is a choice of $e_{d+1}, \cdots, e_{p}$ so that we have $e_{i}^\top \Second_{x}(e_j,e_j)=0$ for all $i=r+1, \cdots, p$ and $j=1, \cdots d$.  We have
\begin{align}\label{C_x structure}
C_x= &\, P(x){
\begin{bmatrix}
M^{(0)}(x,\epsilon) & 0 & 0 \\
0& 0 & 0\\
0& 0 & 0\\
\end{bmatrix}}
+
{\begin{bmatrix}
M^{(11)}(x,\epsilon) & M^{(12)}(x,\epsilon) & 0  \\
M^{(21)}(x,\epsilon) & 0 & 0\\
0& 0 & 0\\
\end{bmatrix}}\\
&\qquad+{\begin{bmatrix}
O(\epsilon^{d+4}) & O(\epsilon^{d+4})  \\
O(\epsilon^{d+4}) & M^{(3)}(x,\epsilon)+O(\epsilon^{d+5}) \\
\end{bmatrix}}\nonumber\,,
\end{align}
where $M^{(0)}$ is a $d\times d$ diagonal matrix with the $m$-th diagonal entry $\mu_{2e_m}(x,\epsilon)$. $M^{(11)}$ is a symmetric $d \times d$ matrix. $M^{(12)} \in \mathbb{R}^{d \times (r-d)}$. $M^{(21)} = {M^{(12)}}^\top $.
In particular, when $x \not\in M_{\epsilon}$, $\begin{bmatrix}
M^{(11)}(x,\epsilon) & M^{(12)}(x,\epsilon) & 0 \\
M^{(21)}(x,\epsilon) & 0 & 0\\
0 & 0 & 0\\
\end{bmatrix}=0.$ $M^{(3)}(x,\epsilon)$ is diagonal $(p-d) \times (p-d)$ matrix and is of order $\epsilon^{d+4}$. 
The first $d$ eigenvalues of $C_x$ are 
\begin{equation}\label{eigenvalue structure}
\lambda_i=P(x)\mu_{2e_i}(x,\epsilon)+\lambda_i^{(1)}(x,\epsilon)+O(\epsilon^{d+4}), 
\end{equation}
where $i=1,\ldots,d$. And $\lambda_i^{(1)}(x,\epsilon)=O(\epsilon^{d+3})$.  If $x \not\in M_{\epsilon}$, $\lambda_i^{(1)}(x,\epsilon)=0$.  The last $p-d$ eigenvalues of $C_x$ are $\lambda_i=O(\epsilon^{d+4})$, where $i=d+1,\ldots,p$.

The corresponding orthonormal eigenvector matrix is 
\begin{align}
X(x,\epsilon)=X(x,0)
+X(x,0)
S(x,\epsilon)+O(\epsilon^2),
\end{align}
where 
\begin{align} \label{eigenvector structure}
X(x,0)=
\begin{bmatrix}
X_1(x) & 0 & 0 \\
0& X_2(x) & 0\\
0& 0 & X_3(x)\\
\end{bmatrix}\,, \quad
\mathsf{S}(x,\epsilon)=
\begin{bmatrix}
\mathsf{S}_{11}(x,\epsilon) & \mathsf{S}_{12}(x,\epsilon) & \mathsf{S}_{13}(x,\epsilon) \\
\mathsf{S}_{21}(x,\epsilon) & \mathsf{S}_{22}(x,\epsilon) & \mathsf{S}_{23}(x,\epsilon) \\
\mathsf{S}_{31}(x,\epsilon) & \mathsf{S}_{32}(x,\epsilon) & \mathsf{S}_{33}(x,\epsilon) \\
\end{bmatrix},
\end{align}
$X_1 \in O(d)$, $X_2 \in O(r-d)$ and $X_3 \in O(p-r)$. The matrix $\mathsf{S}(x,\epsilon)$ is divided into blocks the same as $X(x,0)$. Moreover, $\mathsf{S}(x,\epsilon)$ is an antisymmetric matrix with $0$ on the diagonal entries. In particular, if $x \not\in M_{\epsilon}$, $ \mathsf{S}(x,\epsilon)=0$.
\end{lemma}

The proof is essentially the same as that of \cite[Proposition 3.2]{Wu_Wu:2017}, except that when $x$ is close to the boundary, the integral domain is no longer symmetric.
\begin{proof}
By definition, the $(m,n)$-th entry of $C_x$ is
\begin{align}
& e_m^\top C_xe_n= \int_{D_{\epsilon}(x)}  (\iota(y)-\iota(x))^\top e_m  (\iota(y)-\iota(x))^\top e_n P(y)dV(y).\label{Proof:Lemma5:Expression1}
\end{align}
By the expression
\begin{align}
\tilde{\iota} \circ \exp_{x}(u)-\tilde{\iota}(x)&\,=\tilde{\iota}_*u+\frac{1}{2}\Second_x(u,u)+O(u^3)\,,
\end{align}
we have
\begin{align}
&(\iota(y)-\iota(x))^\top e_m  (\iota(y)-\iota(x))^\top e_n\nonumber\\
&\qquad = (e_m ^\top \tilde{\iota}_*u )( e_n ^\top \tilde{\iota}_*u)+\frac{1}{2} (e_m ^\top \tilde{\iota}_*u )( e_n ^\top \Second_x(u,u))+\frac{1}{2}( e_m ^\top \Second_x(u,u)) (e_n ^\top \tilde{\iota}_*u )+O(u^4).\nonumber
\end{align} 
Thus,
(\ref{Proof:Lemma5:Expression1}) is reduced to
\begin{align}
  e_m^\top C_xe_n= &\, \int_{D_{\epsilon}(x)} \big( (e_m ^\top \tilde{\iota}_*u )( e_n ^\top \tilde{\iota}_*u)+\frac{1}{2} (e_m ^\top \tilde{\iota}_*u )( e_n ^\top \Second_x(u,u))+\frac{1}{2}( e_m ^\top \Second_x(u,u)) (e_n ^\top \tilde{\iota}_*u ) \nonumber\\
  &+\frac{1}{4} [e_m ^\top \Second_x(u,u) e_n ^\top \Second_x(u,u)]+O(u^5) \big) \\
  &\qquad\times
 \big(P(x)+\nabla_u P(x)+O(u^2)\big)
\big(1-\sum_{i,j=1}^d\frac{1}{6} \texttt{Ric}_{x}(i,j) u_iu_j+O(u^3)\big)du  .\nonumber
\end{align}
For $1 \leq m,n \leq d$, $ (e_m ^\top \tilde{\iota}_*u )( e_n ^\top \tilde{\iota}_*u)=u_m u_n$. Moreover, $ e_n ^\top \Second_x(u,u)$ and $ e_m ^\top \Second_x(u,u)$ are zero, so
\begin{align} 
&e_m^\top C_xe_n \label{Proof:Cx}\\
 =&\,  \int_{D_{\epsilon}(x)}( u_mu_n+O(u^4)) \big(P(x)+\nabla_u P(x)+O(u^2)\big)\nonumber\\
&\qquad\times\big(1-\sum_{i,j=1}^d\frac{1}{6} \texttt{Ric}_{x}(i,j) u_iu_j+O(u^3)\big)du \nonumber\\
=&\, P(x)  \int_{\tilde{D}_{\epsilon}(x)} u_mu_n du+ \int_{\tilde{D}_{\epsilon}(x)} u_mu_n \sum_{k=1}^d u_k \partial_k P(x) du + O(\epsilon^{d+4}) \nonumber.
\end{align}
where
we use Lemma  \ref{Lemma:4} to handle the error of changing domain from $D_{\epsilon}(x)$ to $\tilde{D}_{\epsilon}(x)$, which is $O(\epsilon^{d+4})$.
By the symmetry of domain $\tilde{D}_{\epsilon}(x)$, if $1\leq m=n \leq d$, 
\begin{equation}
M_{m,n}^{(0)}= \int_{\tilde{D}_{\epsilon}(x)} u_m^2 du=\mu_{2e_m}(x,\epsilon) \label{LemmaCx:Definition:M0}
\end{equation}
and $M_{m,n}^{(0)}$ is $0$ otherwise.

Next, 
\begin{equation}
M_{m,n}^{(11)}= \int_{\tilde{D}_{\epsilon}(x)} u_mu_n \sum_{k=1}^d u_k \partial_k P(x) du \label{LemmaCx:Definition:M11}
\end{equation}
So, by the symmetry of domain $\tilde{D}_{\epsilon}(x)$, we have 
\begin{align}
M_{m,n}^{(11)}=\left\{
\begin{array}{ll}
  \partial_d P(x) \mu_{2e_m+e_d}(x,\epsilon)  & 1 \leq m=n \leq d, \\
  \partial_n P(x) \mu_{2e_n+e_d}(x,\epsilon)  & m=d,\,  1 \leq n \leq d,\\
  \partial_m P(x) \mu_{2e_m+e_d}(x,\epsilon)  & n=d, \, 1 \leq m \leq d, \\
   0 & \mbox{otherwise.}
\end{array}\right.
\end{align}
For $d+1 \leq m \leq p$ and $d+1 \leq n \leq p$, we have 
\begin{align}
  e_m^\top C_xe_n= &\, \int_{D_{\epsilon}(x)} \big(\frac{1}{4} [e_m ^\top \Second_x(u,u) e_n ^\top \Second_x(u,u)]+ O(u^5)\big)(P(x)+O(u))
(1+O(u))du\\
= &\,  \frac{P(x)}{4} \int_{\tilde{D}_{\epsilon}(x)}  e_m ^\top \Second_x(u,u) e_n ^\top \Second_x(u,u) du+O(\epsilon^{d+5}). \nonumber 
\end{align}
Hence, we have 
\begin{align}
M^{(3)}_{m-d,n-d}(x,\epsilon)=\frac{P(x)}{4} \int_{\tilde{D}_{\epsilon}(x)}  e_m ^\top \Second_x(u,u) e_n ^\top \Second_x(u,u) du.
\end{align}
Since $M^{(3)}_{m-d,m-d}(x,\epsilon)$ is symmetric, we can choose $e_{d+1}, \cdots, e_{p}$ so that it is diagonal. Then $M^{(3)}_{m-d,m-d}(x,\epsilon)=0$ implies
\begin{align}
\int_{\tilde{D}_{\epsilon}(x)}  (e_m ^\top \Second_x(u,u) )^2 du=0. 
\end{align}
Note that since $e_m ^\top \Second_x(u,u) $ is a quadratic form of $u$, we have $e_m ^\top \Second_x(u,u)=0$. 
Since $C_x$ has rank $r$, $M^{(3)}_{m-d,m-d}(x,\epsilon)=0$ for $m=r+1, \cdots, p$, and $e_m ^\top \Second_x(e_i,e_j)=0$ for $m=r+1, \cdots, p$ and $i,j=1, \cdots, d$.

For $1 \leq m \leq d$ and $n \geq d$, 
\begin{align}
 e_m^\top C_xe_n= &\, \int_{D_{\epsilon}(x)} \big(\frac{1}{2} (e_m ^\top \tilde{\iota}_*u )( e_n ^\top \Second_x(u,u))+O(u^4) \big)  \big(P(x)+\nabla_u P(x)+O(u^2)\big) \\
  &\qquad\times
\big(1-\sum_{i,j=1}^d\frac{1}{6} \texttt{Ric}_{x}(i,j) u_iu_j+O(u^3)\big)du  \nonumber \\ 
&=  \frac{P(x)}{2}  \int_{D_{\epsilon}(x)} u_m ( e_n ^\top \Second_x(u,u))du+O(\epsilon^{d+4})\nonumber\,.
\end{align}
We use Lemma  \ref{Lemma:4} to handle the error of changing domain from $D_{\epsilon}(x)$ to $\tilde{D}_{\epsilon}(x)$, which is $O(\epsilon^{d+5})$. Hence, for $1 \leq m \leq d$ and $ d+1 \leq n \leq r$, 
\begin{align}
M^{(12)}(x)_{m,n-d}= \frac{P(x)}{2}  \int_{\tilde{D}_{\epsilon}(x)} u_m ( e_n ^\top \Second_x(u,u))du\,.
\end{align}
By symmetry of $C_x$, we have $M^{(21)} = {M^{(12)}}^\top$.

For $1 \leq m \leq d$ and $ r+1 \leq n \leq p$, 
\begin{align}
 e_m^\top C_xe_n= \frac{P(x)}{2}  \int_{\tilde{D}_{\epsilon}(x)} u_m ( e_n ^\top \Second_x(u,u))du+O(\epsilon^{d+4})=O(\epsilon^{d+4})\,.
\end{align}
For $1 \leq n \leq d$ and $ r+1 \leq m \leq p$,  $e_m^\top C_xe_n=O(\epsilon^{d+4})$ by symmetry.

Based on Lemma \ref{Lemma:5},  $\begin{bmatrix}
M^{(0)}(x,\epsilon) & 0 & 0 \\
0& 0 & 0\\
0& 0 & 0\\
\end{bmatrix}$ is of order $\epsilon^{d+2}$ and $\begin{bmatrix}
M^{(11)}(x,\epsilon) & M^{(12)}(x,\epsilon) & 0 \\
M^{(21)}(x,\epsilon) & 0 & 0\\
0 & 0 & 0\\
\end{bmatrix}$ is of order $\epsilon^{d+3}$. Note that the entries of$\begin{bmatrix}
M^{(11)}(x,\epsilon) & M^{(12)}(x,\epsilon) & 0 \\
M^{(21)}(x,\epsilon) & 0 & 0\\
0 & 0 & 0\\
\end{bmatrix}$ are integrals of odd-order polynomials over $\tilde{D}_{\epsilon}(x)$. Hence, the matrix is $0$ when $x \not\in  M_{\epsilon}$. 
By applying the perturbation theory (see, for example, \cite[Appendix A]{Wu_Wu:2017}), 
 the first $d$ eigenvalues of $C_x$ are 
\begin{align} \label{talyor expansion of eigenvalue}
\lambda_i=P(x)\mu_{2e_i}(x,\epsilon)+\lambda_i^{(1)}(x,\epsilon)+O(\epsilon^{d+4}), 
\end{align}
for $i=1,\ldots,d$ and any $x \in M$, where $\{ \lambda_i^{(1)}(x,\epsilon)\}$ are of order $\epsilon^{d+3}$.  The calculation of $\{ \lambda_i^{(1)}(x,\epsilon)\}$ depends on $M^{(11)}(x,\epsilon)$ and whether $\mu_{2e_i}(x,\epsilon)$ are the same.
Moreover, 
$\lambda_i=O(\epsilon^{d+4})$ for $i=d+1, \ldots, p$.

Suppose that $\texttt{rank}(C_x)=r$, based on the perturbation theory (see, for example, \cite[Appendix A]{Wu_Wu:2017}), the orthonormal eigenvector matrix of $C_x$ is in the form 
\begin{align} \label{talyor expansion of eigenvector}
X(x,\epsilon)=
\begin{bmatrix}
X_1(x) & 0 & 0 \\
0& X_2(x) & 0\\
0& 0 & X_3(x)\\
\end{bmatrix}+
\begin{bmatrix}
X_1(x) & 0 & 0 \\
0& X_2(x) & 0\\
0& 0 & X_3(x)\\
\end{bmatrix}
S(x,\epsilon)+O(\epsilon^2),
\end{align}
where  $X_1(x) \in O(d)$, $X_2(x) \in O(r-d)$ and $X_3(x) \in O(p-r)$. And
\begin{align}
\mathsf{S}(x,\epsilon)=
\begin{bmatrix}
\mathsf{S}_{11}(x,\epsilon) & \mathsf{S}_{12}(x,\epsilon) & \mathsf{S}_{13}(x,\epsilon) \\
\mathsf{S}_{21}(x,\epsilon) & \mathsf{S}_{22}(x,\epsilon) & \mathsf{S}_{23}(x,\epsilon) \\
\mathsf{S}_{31}(x,\epsilon) & \mathsf{S}_{32}(x,\epsilon) & \mathsf{S}_{33}(x,\epsilon) \\
\end{bmatrix}.\nonumber
\end{align}
$\mathsf{S}(x,\epsilon)$ is an antisymmetric matrix with $0$ on the diagonal entries. It is of order $\epsilon$ and depends on those terms of $C_x$ of order $\epsilon^{d+2}$, order $\epsilon^{d+3}$ and higher orders. In particular, 
\begin{equation}
X_1(x)\mathsf{S}_{12}(x,\epsilon)=-[P(x)M^{(0)}(x,\epsilon)]^{-1}M^{(12)}(x,\epsilon)X_2.\nonumber
\end{equation}
And a straightforward calculation shows that
\begin{equation}\label{X1S12 d-1}
e_i^\top  J_{p,d}X_1(x)\mathsf{S}_{12}(x,\epsilon)=-\frac{\mu_{2e_i+e_d}(x,\epsilon)}{\mu_{2e_i}(x,\epsilon)}\mathfrak{N}_{id}^\top (x) J_{p-d,r-d}X_2(x),
\end{equation}
for $i=1, \cdots, d-1$, and
\begin{equation}\label{X1S12 d}
e_d^\top  J_{p,d}X_1(x)\mathsf{S}_{12}(x,\epsilon)=-\frac{1}{2}\sum_{j=1}^d\frac{\mu_{2e_j+e_d}(x,\epsilon)}{\mu_{2e_d}(x,\epsilon)}\mathfrak{N}_{jj}^\top (x) J_{p-d,r-d}X_2(x).
\end{equation}
If $x \not\in M_{\epsilon}$, $ \mathsf{S}(x,\epsilon)=0$.
Moreover, if among first $\mu_{2e_1}, \cdots, \mu_{2e_d}$, there are $1\leq k\leq d$ distinct ones, then there is a choice of the basis  in the tangent space of $M$ so that
\begin{align} \label{description of X_1}
X_1(x)=\begin{bmatrix}
X_1^{(1)}(x) & 0 & \cdots &0 \\
0 & X_1^{(2)} (x)& \cdots & 0\\
0 & 0 &  \ddots &0\\
0 &0  & \cdots & X_1^{(k)}(x)\\
\end{bmatrix}\,,
\end{align}
where each $X_1^{(i)}(x)$ is an orthogonal matrix corresponding to the same  of $\mu_{2e_i}$.
The conclusion follows. 
\end{proof}

\section{Analysis on the augmented vector $\mathbf{T}(x)$} \label{proof of vector T}

We now calculate $\mathbf{T}(x)$. For our purpose, we need an asymptotic expansion up to the first two orders for the tangent component of $\mathbf{T}(x)$ and to the first order for the normal component  when $\epsilon$ is sufficiently small.
\begin{lemma}\label{Lemma:8} 
$\mathbf{T}(x)=[\![ v^{(-1)}_1+ v^{(0)}_{1,1}+v^{(0)}_{1,2}+v^{(0)}_{1,3}+v^{(0)}_{1,4}, v^{(-1)}_2]\!]+[\![O(\epsilon),O(1) ]\!]$, where
\begin{align}
v^{(-1)}_1=&\, \frac{\mu_{e_d}(x,\epsilon)}{\mu_{2e_d}(x,\epsilon)}J_{p,d}^\top e_d\nonumber\,,\\
v^{(0)}_{1,1}=&\, \frac{\nabla P(x)}{P(x)}\,,\nonumber \\
v^{(0)}_{1,2}=&\, -\frac{\epsilon^{d+3}\mu_{e_d}(x,\epsilon)}{P(x)(\mu_{2e_d}(x,\epsilon))^2} J_{p,d}^\top e_d\label{v 0 1 2}\,, \\
v^{(0)}_{1,3}=&\,  -\sum_{i=1}^d \frac{\partial_i P(x)\mu_{e_d}(x,\epsilon)  \mu_{2e_i+e_d}(x,\epsilon)}{P(x) \mu_{2e_i}(x,\epsilon) \mu_{2e_d}(x,\epsilon)}J_{p,d}^\top e_i \nonumber\,,
\end{align}
\begin{align}
v^{(0)}_{1,4}= & \frac{P(x)}{2 \epsilon^{d+3}} \sum_{i=1}^{d-1} \sum_{j=1}^d \big[\big(\frac{\mu_{e_d}(x,\epsilon)\mu_{2e_j+e_d}(x,\epsilon)}{\mu_{2e_d}(x,\epsilon)}-\mu_{2e_j}(x,\epsilon)\big)\frac{\mu_{2e_i+e_d}(x,\epsilon)}{\mu_{2e_i}(x,\epsilon)}\mathfrak{N}_{jj}^\top (x) \big]\mathfrak{N}_{id}(x) J_{p,d}^\top e_i \nonumber\\
+ & \frac{P(x)}{4 \epsilon^{d+3}}\sum_{i=1}^d \sum_{j=1}^d \big[\big(\frac{\mu_{e_d}(x,\epsilon)\mu_{2e_j+e_d}(x,\epsilon)}{\mu_{2e_d}(x,\epsilon)}-\mu_{2e_j}(x,\epsilon)\big)\frac{\mu_{2e_i+e_d}(x,\epsilon)}{\mu_{2e_d}(x,\epsilon)}\mathfrak{N}_{jj}^\top (x) \big]\mathfrak{N}_{ii}(x) J_{p,d}^\top e_d, \nonumber
\end{align}
and 
\begin{align} \label{v -1 2}
v^{(-1)}_2=&  \frac{P(x)}{2 \epsilon^{d+3}} \sum_{j=1}^d\big(\mu_{2e_j}(x,\epsilon)-\frac{\mu_{e_d}(x,\epsilon)\mu_{2e_j+e_d}(x,\epsilon)}{\mu_{2e_d}(x,\epsilon)}\big)\mathfrak{N}_{jj}.
\end{align}
\end{lemma}
Note that by Lemma \ref{Lemma:5}, $v^{(-1)}_1$ is of order $\epsilon^{-1}$ when $x \in  M_{\epsilon}$ and $0$ when $x \notin  M_{\epsilon}$; $v^{(0)}_{1,2}$ is of order $1$ since $\mu_{e_d}(x,\epsilon)$ is of order $\epsilon^{d+1}$ and $\mu_{2e_i}(x,\epsilon)$ is of order $\epsilon^{d+2}$ for $i=1,\ldots,d$. Moreover, when $x \notin  M_{\epsilon}$, we have $\mu_{e_d}(x,\epsilon)=0$ and $\mu_{2e_i+e_d}(x,\epsilon)=0$. Hence,  $v^{(0)}_{1,2}=0$,$v^{(0)}_{1,3}=0$ and $v^{(0)}_{1,4}=0$. Similarly, $v^{(-1)}_2$ is of order $\epsilon^{-1}$.

\begin{proof}
Recall that 
\begin{equation}
\mathbf{T}(x)^\top= \sum_{i=1}^r \frac{\mathbb{E}[(\iota(X)-\iota(x))\chi_{B_{\epsilon}^{\mathbb{R}^p}(\iota(x))}(\iota(X))] ^\top \beta_i \beta_i^\top}{\lambda_i+\epsilon^{d+3}} \,.
\end{equation}
To show the proof, we evaluate the terms in $\mathbf{T}(x)$ one by one. 

Based on Lemma \ref{Lemma:7}, the first $d$ eigenvalues are $\lambda_i=P(x)\mu_{2e_i}(x,\epsilon)+ \lambda_i^{(1)}(x,\epsilon)+O(\epsilon^{d+4})$,
where $i=1,\ldots,d$, and the corresponding eigenvectors are
\begin{equation}
\beta_i=\begin{bmatrix} X_1(x)J_{p,d}^\top e_i \\   0_{(p-d)\times 1} \end{bmatrix}+ \begin{bmatrix} X_1(x)\mathsf{S}_{11}(x,\epsilon) J_{p,d}^\top e_i +O(\epsilon^2) \\ O(\epsilon) 
\end{bmatrix},
\end{equation}
where  $X_1(x) \in O(d)$. 
For $i=d+1,\ldots,r$,  $\lambda_i=O(\epsilon^{d+4})$, and the corresponding eigenvectors are
\begin{equation}
\beta_i=\begin{bmatrix}
0_{d \times 1}\\ J_{p-d,r-d} X_2(x)\mathfrak{J}_{p,r-d}^\top e_i \end{bmatrix}+\begin{bmatrix}X_1(x)\mathsf{S}_{12}(x,\epsilon) \mathfrak{J}_{p,r-d}^\top e_i +O(\epsilon^2) \\ O(\epsilon)  
\end{bmatrix},
\end{equation}
where  $X_2(x) \in O(r-d)$. 

By Lemma \ref{Lemma:6}, we have
\begin{equation}
\mathbb{E}[(\iota(X)-\iota(x))\chi_{B_{\epsilon}^{\mathbb{R}^p}(\iota(x))}(\iota(X))]=[\![v_1,v_2]\!]\,, 
\end{equation}
where 
\begin{align}
v_1=&\,  P(x) \mu_{e_d}(x,\epsilon)  J_{p,d}^\top e_d + \sum_{i=1}^d \partial_iP(x) \mu_{2e_i}(x,\epsilon) J_{p,d}^\top e_i +O(\epsilon^{d+3}),   \\
v_2 = &\, \frac{P(x)}{2} \sum_{i=1}^d \mathfrak{N}_{ii}(x) \mu_{2e_i}(x,\epsilon)+ O(\epsilon^{d+3}). \nonumber
\end{align}
Next, we calculate $\mathbb{E}[(\iota(X)-\iota(x))\chi_{B_{\epsilon}^{\mathbb{R}^p}(\iota(x))}(\iota(X))] ^\top \beta_i $
, for $i=1, \ldots, d$. Note that the normal component of $\beta_i$ is of order $\epsilon$ and the normal component of $\mathbb{E}[(\iota(X)-\iota(x))\chi_{B_{\epsilon}^{\mathbb{R}^p}(\iota(x))}(\iota(X))]$ is of order $\epsilon^{d+2}$, so they will only contribute in the $O(\epsilon^{d+3})$ term. 
Therefore, for $i=1, \ldots, d$, the first two order terms of $\mathbb{E}[(\iota(X)-\iota(x))\chi_{B_{\epsilon}^{\mathbb{R}^p}(\iota(x))}(\iota(X))] ^\top \beta_i $ are
\begin{align}
& \mathbb{E}[(\iota(X)-\iota(x))\chi_{B_{\epsilon}^{\mathbb{R}^p}(\iota(x))}(\iota(X))] ^\top \beta_i\nonumber \\
= &\, \big(P(x) \mu_{e_d}(x,\epsilon) \big)\big(e_d^\top J_{p,d}X_1(x)J_{p,d}^\top e_i \big)+\big(P(x) \mu_{e_d}(x,\epsilon) \big) \big( e_d^\top J_{p,d}X_1(x)\mathsf{S}_{11} (x,\epsilon)J_{p,d}^\top e_i \big)  \nonumber \\
& +\sum_{j=1}^d \big(\partial P_j(x) \mu_{2e_j}(x, \epsilon) \big)\big(e_j^\top J_{p,d}X_1(x)J_{p,d}^\top e_i\big)+O(\epsilon^{d+3}) \nonumber.
\end{align}

By putting the above expressions together, a direct calculation shows that the normal component of $ \sum_{i=1}^d \frac{\mathbb{E}[(\iota(X)-\iota(x))\chi_{B_{\epsilon}^{\mathbb{R}^p}(\iota(x))}(\iota(X))] ^\top \beta_i \beta_i^\top}{\lambda_i+\epsilon^{d+3}}$ is of order $1$ and the tangent component of \\
$ \sum_{i=1}^d \frac{\mathbb{E}[(\iota(X)-\iota(x))\chi_{B_{\epsilon}^{\mathbb{R}^p}(\iota(x))}(\iota(X))] ^\top \beta_i \beta_i^\top}{\lambda_i+\epsilon^{d+3}}$ is of order $\epsilon^{-1}$:
\begin{align} \label{tangent 4 terms}
&P(x)\mu_{e_d}(x,\epsilon)\sum_{i=1}^d\frac{(e_d^\top J_{p,d}X_1(x)J_{p,d}^\top e_i)X_1(x)J_{p,d}^\top e_i }{\lambda_i+\epsilon^{d+3}} \\
+\,&\sum_{i=1}^d \frac{\sum_{j=1}^d (\partial P_j(x) \mu_{2e_j}(x, \epsilon) )(e_j^\top J_{p,d}X_1(x)J_{p,d}^\top e_i) X_1 (x)J_{p,d}^\top e_i}{\lambda_i+\epsilon^{d+3}}\nonumber\\
+\,& P(x)\mu_{e_d}(x,\epsilon)\sum_{i=1}^d \frac{(e_d^\top J_{p,d}X_1(x)\mathsf{S}_{11} (x,\epsilon)J_{p,d}^\top e_i) X_1 (x)J_{p,d}^\top e_i  }{\lambda_i+\epsilon^{d+3}} \nonumber\\
+\,& P(x)\mu_{e_d}(x,\epsilon)\sum_{i=1}^d \frac{(e_d^\top J_{p,d}X_1(x)J_{p,d}^\top e_i)X_1(x)\mathsf{S}_{11}(x,\epsilon) J_{p,d}^\top e_i }{\lambda_i+\epsilon^{d+3}}  +O(\epsilon)\,,\nonumber
\end{align}
where the first term is of order $\epsilon^{-1}$, the second to the fourth terms are of order $1$ since $\mu_{e_d}(x,\epsilon)$ is of order $\epsilon^{d+1}$, $\mu_{2e_i}(x,\epsilon)$ is of order $\epsilon^{d+2}$ for $i=1,\ldots,d$, $\mathsf{S}_{11}(x,\epsilon) $ is of order $\epsilon$ and $\lambda_i$ is of order $\epsilon^{d+2}$ for $i=1,\ldots,d$.
Note that above formula involves $\lambda_i$ and $\mathsf{S}_{11}(x,\epsilon)$. We are going to express those terms by $\mu_{2e_i}$ and $\mu_{2e_i+e_d}$. In the following paragraph, we prepare some necessary ingredients to simplify the formula of tangent component of $ \sum_{i=1}^d \frac{\mathbb{E}[(\iota(X)-\iota(x))\chi_{B_{\epsilon}^{\mathbb{R}^p}(\iota(x))}(\iota(X))] ^\top \beta_i \beta_i^\top}{\lambda_i+\epsilon^{d+3}}$.
Recall that from (\ref{description of X_1}),  
\begin{align} 
X_1(x)=\begin{bmatrix}
X_1^{(1)}(x) & 0 & \cdots &0 \\
0 & X_1^{(2)} (x)& \cdots & 0\\
0 & 0 &  \ddots &0\\
0 &0  & \cdots & X_1^{(k)}(x)\\
\end{bmatrix}\,,\nonumber
\end{align}
$1 \leq k \leq d$. Here different $X_1^{(i)}$ corresponds to different $\mu_{2e_i}$. Each $X_1^{(i)}$ is an orthogonal matrix. 
By reordering the basis $\{e_1, \cdots, e_d \}$ of the tangent space $T_xM$, we suppose that among $\mu_{2e_1}(x, \epsilon), \cdots,  \mu_{2e_{d-1}}(x, \epsilon)$, the first $t$ terms of them are different from $\mu_{2e_d}$. 
Define
\begin{align} \label{block X 11}
X_{1,1}(x):=\begin{bmatrix}
X_1^{(1)}(x) & 0 & \cdots &0 \\
0 & X_1^{(2)} (x)& \cdots & 0\\
0 & 0 &  \ddots &0\\
0 &0  & \cdots & X_1^{(k-1)}(x)\\
\end{bmatrix}\,.
\end{align}
Hence, we have 
\begin{equation}\label{decomposition X_1, 2 blocks}
X_1(x)=\begin{bmatrix}
X_{1,1}(x)& 0  \\
0 & X_{1,2} (x)\\
\end{bmatrix}\,,\nonumber
\end{equation}
where $X_{1,1}(x)\in O(t)$, $X_{1,2} \in O(d-t)$, and $0 \leq t \leq d-1$. Divide $M^{(11)}(x,\epsilon)$ in \eqref{C_x structure} and $\mathsf{S}_{11}(x,\epsilon)$ in \eqref{eigenvector structure} corresponding to $\begin{bmatrix}
X_{1,1}(x)& 0  \\
0 & X_{1,2} (x)\\
\end{bmatrix}$:
\begin{align}
M^{(11)}(x,\epsilon)=
\begin{bmatrix}
M^{(11)}_1(x,\epsilon) & M^{(11)}_2(x,\epsilon) \\
M^{(11)}_3(x,\epsilon) & M^{(11)}_4(x,\epsilon)\\
\end{bmatrix} \,,\quad
\mathsf{S}_{11}(x,\epsilon)=
\begin{bmatrix}
\mathsf{S}_{11,1}(x,\epsilon) & \mathsf{S}_{11,2}(x,\epsilon)  \\
\mathsf{S}_{11,3}(x,\epsilon) & \mathsf{S}_{11,4}(x,\epsilon)  \\
\end{bmatrix}.\nonumber
\end{align}
Recall that 
\begin{align}
M_{m,n}^{(11)}=\left\{
\begin{array}{ll}
  \partial_d P(x) \mu_{2e_m+e_d}(x,\epsilon)  & 1 \leq m=n \leq d, \\
  \partial_n P(x) \mu_{2e_n+e_d}(x,\epsilon)  & m=d,  1 \leq n \leq d,\\
  \partial_m P(x) \mu_{2e_m+e_d}(x,\epsilon)  & n=d,  1 \leq m \leq d, \\
   0 & \mbox{otherwise.}
\end{array}\right.
\end{align}
By the perturbation theory (see, e.g., \cite[Appendix A]{Wu_Wu:2017}), $ \lambda_{t+1}^{(1)}(x,\epsilon), \cdots, \lambda_{d}^{(1)}(x,\epsilon)$ are the eigenvalues of $M^{(11)}_4(x,\epsilon)$ and $X_{1,2} (x)$ is the orthonormal eigenvector matrix of $M^{(11)}_4(x,\epsilon)$.
We have
\begin{equation} \label{S_11,2}
\mathsf{S}_{11,2}(x,\epsilon)=X_{1,1}^\top(x) [P(x)\mu_{2e_d}(x, \epsilon)I_{t \times t}-\Lambda]^{-1} M^{(11)}_2(x,\epsilon)X_{1,2}(x)\,,
\end{equation}
where
\begin{align}
\Lambda=
\begin{bmatrix}
P(x) \mu_{2e_1}(x, \epsilon)  & \cdots &0 \\
0  &  \ddots &0\\
0  & \cdots  & P(x) \mu_{2e_t}(x, \epsilon)\\
\end{bmatrix},\nonumber
\end{align}
Next, we simplify the terms in equation \eqref{tangent 4 terms} one by one. We start from the first one. Recall that based on the structure of $X_1$, we have $e_d^\top J_{p,d}X_1(x)J_{p,d}^\top e_i=0$ for $i=1, \cdots, t$. Hence,
\begin{align}
& P(x)\mu_{e_d}(x,\epsilon)\sum_{i=1}^d \frac{(e_d^\top J_{p,d}X_1(x)J_{p,d}^\top e_i)X_1(x)J_{p,d}^\top e_i }{\lambda_i+\epsilon^{d+3}} \\
=&\, P(x)\mu_{e_d}(x,\epsilon)\sum_{i=t+1}^d \frac{(e_d^\top J_{p,d}X_1(x)J_{p,d}^\top e_i)X_1(x)J_{p,d}^\top e_i }{P(x)\mu_{2e_d}(x,\epsilon)+\lambda^{(1)}_i(x,\epsilon)+\epsilon^{d+3}+O(\epsilon^{d+4})} \nonumber \\
=&\, P(x)\mu_{e_d}(x,\epsilon)\sum_{i=t+1}^d \bigg[ \frac{1}{P(x)\mu_{2e_d}(x,\epsilon)}-\frac{\lambda^{(1)}_i(x,\epsilon)+\epsilon^{d+3}}{(P(x)\mu_{2e_d}(x,\epsilon))^2}+O(\epsilon^{-d}) \bigg]\nonumber\\
&\qquad\times (e_d^\top J_{p,d}X_1(x)J_{p,d}^\top e_i)X_1(x)J_{p,d}^\top e_i   \nonumber \\
=&\, \frac{\mu_{e_d}(x,\epsilon)}{\mu_{2e_d}(x,\epsilon)}\sum_{i=t+1}^d  (e_d^\top J_{p,d}X_1(x)J_{p,d}^\top e_i)X_1(x)J_{p,d}^\top e_i  
\nonumber\\
&\,- \frac{\mu_{e_d}(x,\epsilon)}{P(x)(\mu_{2e_d}(x,\epsilon))^2}  \sum_{i=t+1}^d  \lambda^{(1)}_i(x,\epsilon) (e_d^\top J_{p,d}X_1(x)J_{p,d}^\top e_i)X_1(x)J_{p,d}^\top e_i \nonumber \\
 &\,- \frac{\epsilon^{d+3}\mu_{e_d}(x,\epsilon)}{P(x)(\mu_{2e_d}(x,\epsilon))^2} \sum_{i=t+1}^d   (e_d^\top J_{p,d}X_1(x)J_{p,d}^\top e_i)X_1(x)J_{p,d}^\top e_i +O(\epsilon)\nonumber.
\end{align}
Note that we use (\ref{eigenvalue structure}) in the first step. Moreover, we have
\begin{align}
\frac{\mu_{e_d}(x,\epsilon)}{\mu_{2e_d}(x,\epsilon)}\sum_{i=t+1}^d  (e_d^\top J_{p,d}X_1(x)J_{p,d}^\top e_i)X_1(x)J_{p,d}^\top e_i=
\frac{\mu_{e_d}(x,\epsilon)}{\mu_{2e_d}(x,\epsilon)}J_{p,d}^\top e_d
\end{align}
and
\begin{align}
\frac{\epsilon^{d+3}\mu_{e_d}(x,\epsilon)}{P(x)(\mu_{2e_d}(x,\epsilon))^2} \sum_{i=t+1}^d   (e_d^\top J_{p,d}X_1(x)J_{p,d}^\top e_i)X_1(x)J_{p,d}^\top e_i= \frac{\epsilon^{d+3}\mu_{e_d}(x,\epsilon)}{P(x)(\mu_{2e_d}(x,\epsilon))^2} J_{p,d}^\top e_d. 
\end{align}
By using the eigen-decomposition of $M^{(11)}_4(x,\epsilon)$, we have
\begin{align}
& e_k^\top J_{p,d}\sum_{i=t+1}^d  \lambda^{(1)}_i(x,\epsilon)  (e_d^\top J_{p,d}X_1(x)J_{p,d}^\top e_i)X_1(x)J_{p,d}^\top e_i \\
=\,& \sum_{i=t+1}^d   \lambda^{(1)}_i(x,\epsilon) (e_d^\top J_{p,d}X_1(x)J_{p,d}^\top e_i) (e_k^\top J_{p,d}X_1(x)J_{p,d}^\top e_i ) \nonumber \\
=\,& \sum_{i=t+1}^d   \lambda^{(1)}_i(x,\epsilon) (e_d^\top J_{p,d}X_1(x)J_{p,d}^\top e_i) (e_i^\top J_{p,d}X^\top_1(x)J_{p,d}^\top e_k) \nonumber \\
=\,& \partial_k P(x) \mu_{2e_k+e_d}(x,\epsilon)  \nonumber
\end{align}
for $t+1 \leq k \leq d$ and this quantity is $0$ if $1 \leq k \leq d$.
Thus, if we sum up the above terms, we have
\begin{align} \label{part 1 in tangent T}
& P(x)\mu_{e_d}(x,\epsilon)\sum_{i=1}^d \frac{(e_d^\top J_{p,d}X_1(x)J_{p,d}^\top e_i)X_1(x)J_{p,d}^\top e_i }{\lambda_i+\epsilon^{d+3}}\\ 
=&  \frac{\mu_{e_d}(x,\epsilon)}{\mu_{2e_d}(x,\epsilon)}J_{p,d}^\top e_d-\frac{\epsilon^{d+3}\mu_{e_d}(x,\epsilon)}{P(x)(\mu_{2e_d}(x,\epsilon))^2} J_{p,d}^\top e_d\nonumber\\
&\qquad- \sum_{j=t+1}^d \frac{\partial_j P(x) \mu_{e_d}(x,\epsilon) \mu_{2e_j+e_d}(x,\epsilon)}{P(x)(\mu_{2e_d}(x,\epsilon))^2} J_{p,d}^\top e_j+O(\epsilon) \nonumber\,.
\end{align}

Next, we simplify the second term in \eqref{tangent 4 terms}. Recall the description of $X_1$ (e.g. (\ref{description of X_1})). We have 
\begin{align}
\sum_{i=1}^d \frac{\big(e_j^\top J_{p,d}X_1(x)J_{p,d}^\top e_i\big)X_1(x)J_{p,d}^\top e_i}{\mu_{2e_i}(x,\epsilon)}=\frac{1}{\mu_{2e_j}(x,\epsilon)}J_{p,d}^\top e_j. 
\end{align}
Hence,
\begin{align}  \label{part 2 in tangent T}
& \sum_{i=1}^d \frac{\sum_{j=1}^d \big(\partial P_j(x) \mu_{2e_j}(x, \epsilon) \big)\big(e_j^\top J_{p,d}X_1(x)J_{p,d}^\top e_i\big)X_1(x)J_{p,d}^\top e_i }{\lambda_i+\epsilon^{d+3}} \\
=&\, \sum_{i=1}^d \frac{\sum_{j=1}^d \big(\partial P_j(x) \mu_{2e_j}(x, \epsilon) \big)\big(e_j^\top J_{p,d}X_1(x)J_{p,d}^\top e_i\big)X_1(x)J_{p,d}^\top e_i }{P(x)\mu_{2e_i}(x,\epsilon)+\lambda_i^{(1)}(x,\epsilon)+\epsilon^{d+3}+O(\epsilon^{d+4})}  \nonumber \\
=&\, \sum_{j=1}^d \frac{\partial P_j(x)\mu_{2e_j}(x,\epsilon)}{P(x)} \sum_{i=1}^d \frac{\big(e_j^\top J_{p,d}X_1(x)J_{p,d}^\top e_i\big)X_1(x)J_{p,d}^\top e_i}{\mu_{2e_i}(x,\epsilon)}  +O(\epsilon)\nonumber \\
=&\, \frac{\nabla P(x)}{P(x)}+O(\epsilon)\,. \nonumber 
\end{align}
At last, we simplify the third and the last terms in \eqref{tangent 4 terms} together, because we need to use the antisymmetric property of $\mathsf{S}_{11} (x,\epsilon)$.
\begin{align}
&P(x)\mu_{e_d}(x,\epsilon)\sum_{i=1}^d \frac{(e_d^\top J_{p,d}X_1(x)\mathsf{S}_{11} (x,\epsilon)J_{p,d}^\top e_i) X_1 (x)J_{p,d}^\top e_i  }{\lambda_i} \\
&\qquad\qquad\qquad+ P(x)\mu_{e_d}(x,\epsilon)\sum_{i=1}^d \frac{(e_d^\top J_{p,d}X_1(x)J_{p,d}^\top e_i)X_1(x)\mathsf{S}_{11}(x,\epsilon) J_{p,d}^\top e_i }{\lambda_i} \nonumber\\
=&\, \mu_{e_d}(x,\epsilon)\sum_{i=1}^d \bigg[ \frac{(e_d^\top J_{p,d}X_1(x)\mathsf{S}_{11} (x,\epsilon)J_{p,d}^\top e_i) X_1 (x)J_{p,d}^\top e_i  }{\mu_{2e_i}(x,\epsilon)+\lambda_i^{(1)}(x,\epsilon)/P(x)+\epsilon^{d+3}+O(\epsilon^{d+4})} \nonumber\\
&\qquad\qquad\qquad+ \frac{(e_d^\top J_{p,d}X_1(x)J_{p,d}^\top e_i)X_1(x)\mathsf{S}_{11}(x,\epsilon) J_{p,d}^\top e_i }{\mu_{2e_i}(x,\epsilon)+\lambda_i^{(1)}(x,\epsilon)/P(x)+\epsilon^{d+3}+O(\epsilon^{d+4})} \bigg] \nonumber \\
=&\, \bar{v}+O(\epsilon) \,,\nonumber 
\end{align}
where we denote 
\begin{align}
\bar{v}:=& \mu_{e_d}(x,\epsilon)\sum_{i=1}^d \bigg[ \frac{(e_d^\top J_{p,d}X_1(x)\mathsf{S}_{11} (x, \epsilon )J_{p,d}^\top e_i) X_1 (x)J_{p,d}^\top e_i  }{\mu_{2e_i}(x,\epsilon)} \\
&\qquad\qquad\qquad+ \frac{(e_d^\top J_{p,d}X_1(x)J_{p,d}^\top e_i)X_1(x)\mathsf{S}_{11}(x, \epsilon ) J_{p,d}^\top e_i }{\mu_{2e_i}(x,\epsilon)} \bigg].\nonumber
\end{align}
We now simplify $\bar{v}$.
Note that, for $1 \leq k \leq d$,  
\begin{align}
&e_k^\top J_{p,d}\bar{v} \nonumber\\
=&\, \mu_{e_d}(x,\epsilon)\sum_{i=1}^d \big(\sum_{j=1}^d e_d^\top J_{p,d}X_1(x)J_{p,d}^\top e_j \frac{ e_j^\top J_{p,d}\mathsf{S}_{11} (x, \epsilon)J_{p,d}^\top e_i}{\mu_{2e_i}(x, \epsilon)}\big) e_k^\top J_{p,d}X_1(x)J_{p,d}^\top e_i \nonumber\\
&+\, \mu_{e_d}(x,\epsilon) \sum_{i=1}^d e_d^\top J_{p,d}X_1(x)J_{p,d}^\top e_i \big(\sum_{j=1}^d e_k^\top J_{p,d}X_1(x)J_{p,d}^\top e_j \frac{ e_j^\top J_{p,d}\mathsf{S}_{11} (x, \epsilon)J_{p,d}^\top e_i}{\mu_{2e_i}(x, \epsilon)}\big)  \nonumber \\
=&\,  \mu_{e_d}(x,\epsilon)\sum_{j=1}^d e_d^\top J_{p,d}X_1(x)J_{p,d}^\top e_j \big(\sum_{i=1}^d e_k^\top J_{p,d}X_1(x)J_{p,d}^\top e_i \frac{ e_j^\top J_{p,d}\mathsf{S}_{11} (x, \epsilon)J_{p,d}^\top e_i}{\mu_{2e_i}(x, \epsilon)}\big)  \nonumber \\
&+\, \mu_{e_d}(x,\epsilon)\sum_{j=1}^d e_d^\top J_{p,d}X_1(x)J_{p,d}^\top e_j \big(\sum_{i=1}^d e_k^\top J_{p,d}X_1(x)J_{p,d}^\top e_i \frac{ e_i^\top J_{p,d}\mathsf{S}_{11} (x, \epsilon)J_{p,d}^\top e_j}{\mu_{2e_j}(x, \epsilon)}\big)  \nonumber \\
=&\, \mu_{e_d}(x,\epsilon) \sum_{j=1}^d e_d^\top J_{p,d}X_1(x)J_{p,d}^\top e_j \nonumber\\
&\times\Big[\sum_{i=1}^d e_k^\top J_{p,d}X_1(x)J_{p,d}^\top e_i \Big(\frac{1}{\mu_{2e_j}(x, \epsilon)}- \frac{1}{\mu_{2e_i}(x, \epsilon)}\Big) e_i^\top J_{p,d}\mathsf{S}_{11} (x, \epsilon)J_{p,d}^\top e_j\Big]. \nonumber
\end{align}
In the last step, we use the fact that $\mathsf{S}_{11} (x, \epsilon)$ is antisymmetric.
Based on the structure of $X_1(x)$, $e_d^\top J_{p,d}X_1(x)J_{p,d}^\top e_j=0$ for $j=1, \cdots, t$, $e_k^\top J_{p,d}X_1(x)J_{p,d}^\top e_i=0$, for  $k=1, \cdots, t$ and $i=t+1, \cdots, d$ and $e_k^\top J_{p,d}X_1(x)J_{p,d}^\top e_i=0$ for $k=t+1, \cdots, d$ and $i=1, \cdots, t$.
We can further simplify $\bar{v}$ as
\begin{align}
e_k^\top& J_{p,d} \bar{v} =\, \mu_{e_d}(x,\epsilon)\sum_{j=t+1}^d e_d^\top J_{p,d}X_1(x)J_{p,d}^\top e_j \label{ek bar v}\\
&\times\Big[\sum_{i=1}^t e_k^\top J_{p,d}X_1(x)J_{p,d}^\top e_i \Big(\frac{1}{\mu_{2e_d}(x, \epsilon)}- \frac{1}{\mu_{2e_i}(x, \epsilon)}\Big) e_i^\top J_{p,d}\mathsf{S}_{11} (x, \epsilon)J_{p,d}^\top e_j\Big],\nonumber
\end{align}
for $k=1, \cdots, t$, and
\begin{align}
&e_k^\top J_{p,d} \bar{v} \nonumber\\
=&\, \mu_{e_d}(x,\epsilon)\sum_{j=t+1}^d e_d^\top J_{p,d}X_1(x)J_{p,d}^\top e_j \\
&\times\Big[\sum_{i=t+1}^d e_k^\top J_{p,d}X_1(x)J_{p,d}^\top e_i \Big(\frac{1}{\mu_{2e_d}(x, \epsilon)}- \frac{1}{\mu_{2e_i}(x, \epsilon)}\Big) e_i^\top J_{p,d}\mathsf{S}_{11} (x, \epsilon)J_{p,d}^\top e_j\Big]\nonumber \\
=&\,0, \nonumber
\end{align}
for $k=t+1, \cdots, d$, where we use the fact that $\mu_{2e_{t+1}}(x, \epsilon)=\cdots =\mu_{2e_{d}}(x, \epsilon)$.

Next, we focus on the case when $1 \leq k \leq t$. By (\ref{S_11,2}), for $1 \leq i \leq t$ and $t+1 \leq j \leq d$, we have 
\begin{align}
e_i^\top J_{p,d} \mathsf{S}_{11}(x,\epsilon) J_{p,d}^\top e_j =&\sum_{l=1}^t e_i^\top J_{p,d} X_{1}^\top (x)  J_{p,d}^\top e_l \nonumber\\
&\times\sum_{m=t+1}^d \frac{(e_l^\top J_{p,d}M^{(11)}(x,\epsilon) J_{p,d}^\top e_m) (e_m^\top J_{p,d}X_{1}(x) J_{p,d}^\top e_j)}{P(x) \mu_{2e_d}(x,\epsilon) -P(x) \mu_{2e_l}(x,\epsilon)}.
\end{align}
Note by Lemma \ref{Lemma:7}, if  $1\leq l \leq t$, and $t+1 \leq m <d$, then $e_l^\top J_{p,d}M^{(11)}(x,\epsilon) J_{p,d}^\top e_m =0$. And 
\begin{equation}
e_l^\top J_{p,d}M^{(11)}(x,\epsilon) J_{p,d}^\top e_d=  \partial_l P(x) \mu_{2e_l+e_d}(x,\epsilon).
\end{equation} 
Hence, 
\begin{align}
e_i^\top J_{p,d} \mathsf{S}_{11}(x,\epsilon) &J_{p,d}^\top e_j = e_j^\top J_{p,d}X^\top_{1}(x) J_{p,d}^\top e_d\nonumber\\
&\times\sum_{l=1}^t e_i^\top J_{p,d} X_{1}^\top (x)  J_{p,d}^\top e_l  \frac{\partial_l P(x) \mu_{2e_l+e_d}(x,\epsilon)}{P(x) \mu_{2e_d}(x,\epsilon) -P(x) \mu_{2e_l}(x,\epsilon)}.
\end{align}

We substitute above equation into (\ref{ek bar v}),
\begin{align}
&e_k^\top J_{p,d} \bar{v} \nonumber\\
=&\, \mu_{e_d}(x,\epsilon)\sum_{j=t+1}^d (e_d^\top J_{p,d}X_1(x)J_{p,d}^\top e_j)(e_j^\top J_{p,d}X^\top_{1}(x) J_{p,d}^\top e_d) \bigg[ \sum_{i=1}^t e_k^\top J_{p,d}X_1(x)J_{p,d}^\top e_i \nonumber \\  
& \times\Big(\frac{1}{\mu_{2e_d}(x, \epsilon)}- \frac{1}{\mu_{2e_i}(x, \epsilon)}\Big) \sum_{l=1}^t e_i^\top J_{p,d} X_{1}^\top (x)  J_{p,d}^\top e_l \frac{ \partial_l P(x) \mu_{2e_l+e_d}(x,\epsilon)}{P(x) \mu_{2e_d}(x,\epsilon) -P(x) \mu_{2e_l}(x,\epsilon)}\bigg] \nonumber \\ 
=& \mu_{e_d}(x,\epsilon)  \bigg[ \sum_{i=1}^t e_k^\top J_{p,d}X_1(x)J_{p,d}^\top e_i \Big(\frac{1}{\mu_{2e_d}(x, \epsilon)}- \frac{1}{\mu_{2e_i}(x, \epsilon)}\Big) \nonumber \\
& \times\sum_{l=1}^t e_i^\top J_{p,d} X_{1}^\top (x)  J_{p,d}^\top e_l \frac{ \partial_l P(x) \mu_{2e_l+e_d}(x,\epsilon)}{P(x) \mu_{2e_d}(x,\epsilon) -P(x) \mu_{2e_l}(x,\epsilon)}\bigg] \nonumber \\
=& \mu_{e_d}(x,\epsilon)  \bigg[  \sum_{l=1}^t  \sum_{i=1}^t (e_k^\top J_{p,d}X_1(x)J_{p,d}^\top e_i) (e_i^\top J_{p,d} X_{1}^\top (x)  J_{p,d}^\top e_l) \Big(\frac{1}{\mu_{2e_d}(x, \epsilon)}- \frac{1}{\mu_{2e_i}(x, \epsilon)}\Big) \nonumber \\
& \times\frac{ \partial_l P(x) \mu_{2e_l+e_d}(x,\epsilon)}{P(x) \mu_{2e_d}(x,\epsilon) -P(x) \mu_{2e_l}(x,\epsilon)}\bigg] \nonumber \\
=& \mu_{e_d}(x,\epsilon)  \bigg[ (\frac{1}{\mu_{2e_d}(x, \epsilon)}- \frac{1}{\mu_{2e_k}(x, \epsilon)})  \sum_{l=1}^t  \sum_{i=1}^t (e_k^\top J_{p,d}X_1(x)J_{p,d}^\top e_i) (e_i^\top J_{p,d} X_{1}^\top (x)  J_{p,d}^\top e_l)  \nonumber \\
& \times\frac{ \partial_l P(x) \mu_{2e_l+e_d}(x,\epsilon)}{P(x) \mu_{2e_d}(x,\epsilon) -P(x) \mu_{2e_l}(x,\epsilon)}\bigg] \nonumber \,,
\end{align}
where we use the fact that 
\begin{align}
\sum_{j=t+1}^d (e_d^\top J_{p,d}X_1(x)J_{p,d}^\top e_j)(e_j^\top J_{p,d}X^\top_{1}(x) J_{p,d}^\top e_d)=1
\end{align}
in the second step and the fact that
\begin{align}
(e_k^\top J_{p,d}X_1(x)J_{p,d}^\top e_i) (e_i^\top J_{p,d} X_{1}^\top (x)  J_{p,d}^\top e_l) \not=0
\end{align} 
only if $e_k^\top J_{p,d}X_1(x)J_{p,d}^\top e_i$ and $e_i^\top J_{p,d} X_{1}^\top (x)  J_{p,d}^\top e_l$ are entries in the block $X^{(m)}_1$ in (\ref{block X 11}) corresponding to $\mu_{2e_k}$ in the fourth step. Note that 
\begin{align}
\sum_{i=1}^t (e_k^\top J_{p,d}X_1(x)J_{p,d}^\top e_i) (e_i^\top J_{p,d} X_{1}^\top (x)  J_{p,d}^\top e_l) =1, 
\end{align}
if $1\leq k=l \leq t$ and is $0$ otherwise. 

Hence, we have
\begin{align} \label{part 3 in tangent T}
e_k^\top  J_{p,d}\bar{v}= -\frac{\partial_k P(x)\mu_{e_d}(x,\epsilon)  \mu_{2e_k+e_d}(x,\epsilon)}{P(x) \mu_{2e_k}(x, \epsilon) \mu_{2e_d}(x, \epsilon)},
\end{align}
for $1\leq k \leq t$, and $e_k^\top J_{p,d} \bar{v}=0$, for $t+1 \leq k \leq d$.
If we sum up equations (\ref {part 1 in tangent T}) (\ref {part 2 in tangent T}) and (\ref {part 3 in tangent T}), the tangent component of 
$ \sum_{i=1}^d \frac{\mathbb{E}[(\iota(X)-\iota(x))\chi_{B_{\epsilon}^{\mathbb{R}^p}(\iota(x))}(\iota(X))] ^\top \beta_i \beta_i^\top}{\lambda_i+\epsilon^{d+3}}$ becomes 
\begin{align}
& \frac{\mu_{e_d}(x,\epsilon)}{\mu_{2e_d}(x,\epsilon)}J_{p,d}^\top e_d+\frac{\nabla P(x)}{P(x)}-\frac{\epsilon^{d+3}\mu_{e_d}(x,\epsilon)}{P(x)(\mu_{2e_d}(x,\epsilon))^2} J_{p,d}^\top e_d\\
&\qquad- \sum_{i=t+1}^d \frac{\partial_i P(x) \mu_{e_d}(x,\epsilon) \mu_{2e_i+e_d}(x,\epsilon)}{P(x)(\mu_{2e_d}(x,\epsilon))^2} J_{p,d}^\top e_i \nonumber\\
&\qquad -\sum_{i=1}^t \frac{\partial_i P(x)\mu_{e_d}(x,\epsilon)  \mu_{2e_i+e_d}(x,\epsilon)}{P(x) \mu_{2e_i} (x,\epsilon)\mu_{2e_d}(x,\epsilon)}J_{p,d}^\top e_i+O(\epsilon) \nonumber \\
=&\, \frac{\mu_{e_d}(x,\epsilon)}{\mu_{2e_d}(x,\epsilon)}J_{p,d}^\top e_d+\frac{\nabla P(x)}{P(x)}-\frac{\epsilon^{d+3}\mu_{e_d}(x,\epsilon)}{P(x)(\mu_{2e_d}(x,\epsilon))^2} J_{p,d}^\top e_d\nonumber\\
&\qquad- \sum_{i=1}^d \frac{\partial_i P(x)\mu_{e_d}(x,\epsilon)  \mu_{2e_i+e_d}(x,\epsilon)}{P(x) \mu_{2e_i} (x, \epsilon)\mu_{2e_d}(x, \epsilon)}J_{p,d}^\top e_i+O(\epsilon) \nonumber,
\end{align}
where we use $\mu_{2e_{t+1}}(x, \epsilon)=\cdots =\mu_{2e_{d}}(x, \epsilon)$ in the last step. 

We now finish calculating the tangent component of $\sum_{i=1}^d \frac{\mathbb{E}[(\iota(X)-\iota(x))\chi_{B_{\epsilon}^{\mathbb{R}^p}(\iota(x))}(\iota(X))] ^\top \beta_i \beta_i^\top}{\lambda_i+\epsilon^{d+3}}$. Next, we need to calculate both the tangent and the normal component of $ \sum_{i=d+1}^r \frac{\mathbb{E}[(\iota(X)-\iota(x))\chi_{B_{\epsilon}^{\mathbb{R}^p}(\iota(x))}(\iota(X))] ^\top \beta_i \beta_i^\top}{\lambda_i+\epsilon^{d+3}}$
Note that for $i=d+1, \ldots, r$,
\begin{align}
& \mathbb{E}[(\iota(X)-\iota(x))\chi_{B_{\epsilon}^{\mathbb{R}^p}(\iota(x))}(\iota(X))] ^\top \beta_i \\
= &\,   P(x)  \mu_{e_d}(x,\epsilon) \big(  e_d^\top  J_{p,d}  X_1(x)\mathsf{S}_{12}(x,\epsilon) \mathfrak{J}_{p,r-d}^\top e_i \big ) \nonumber\\
&\qquad+ \frac{P(x)}{2} \sum_{j=1}^d  \mu_{2e_j}(x,\epsilon) \mathfrak{N}_{jj}^\top (x) J_{p-d,r-d} X_2(x)\mathfrak{J}_{p,r-d}^\top e_i+O(\epsilon^{d+3}), \nonumber 
\end{align}
where both terms are of order $\epsilon^{d+2}$.
Since $\lambda_i=O(\epsilon^{d+4})$, $\epsilon^{d+3}$ dominates the eigenvalues. For $i=d+1,\ldots,r$, we have
\begin{align}
&  \frac{\mathbb{E}[(\iota(X)-\iota(x))\chi_{B_{\epsilon}^{\mathbb{R}^p}(\iota(x))}(\iota(X))]^\top \beta_i }{\lambda_i +\epsilon^{d+3}} 
= \,  \frac{\mathbb{E}[(\iota(X)-\iota(x))\chi_{B_{\epsilon}^{\mathbb{R}^p}(\iota(x))}(\iota(X))]^\top \beta_i }{\epsilon^{d+3}+O(\epsilon^{d+4})}  \\
= &\,  P(x) \frac{\mu_{e_d}(x,\epsilon)}{\epsilon^{d+3}} \big(e_d^\top  J_{p,d}  X_1(x)\mathsf{S}_{12}(x,\epsilon) \mathfrak{J}_{p,r-d}^\top e_i \big )\nonumber\\
&\qquad + \frac{P(x)}{2} \sum_{j=1}^d  \frac{\mu_{2e_j}(x,\epsilon)}{\epsilon^{d+3}} \mathfrak{N}_{jj}^\top (x) J_{p-d,r-d} X_2(x)\mathfrak{J}_{p,r-d}^\top e_i+O(1). \nonumber 
\end{align}
Similarly, we need to express the above formula in terms of $\mu_{2e_i}(x,\epsilon)$ and $\mu_{2e_i+e_d}(x,\epsilon)$. The simplification here mainly relies on the perturbation formula equations (\ref{X1S12 d-1}) and (\ref{X1S12 d}) which relates $\mathsf{S}_{12}(x,\epsilon)$ with the second fundamental form of the manifold at $x$. 
First of all, a direct calculation shows that
\begin{align}
& \sum_{i=d+1}^r \frac{\mathbb{E}[(\iota(X)-\iota(x))\chi_{B_{\epsilon}^{\mathbb{R}^p}(\iota(x))}(\iota(X))]^\top \beta_i \beta_i^\top}{\lambda_i +\epsilon^{d+3}} \nonumber\\
= &\, \sum_{i=d+1}^r \frac{\mathbb{E}[(\iota(X)-\iota(x))\chi_{B_{\epsilon}^{\mathbb{R}^p}(\iota(x))}(\iota(X))]^\top \beta_i \beta_i^\top}{\epsilon^{d+3}+O(\epsilon^{d+4})} \nonumber \\
= &\,  \Big[\!\!\!\Big[ \sum_{i=d+1}^r \bigg[ P(x) \frac{\mu_{e_d}(x,\epsilon)}{\epsilon^{d+3}} \big(e_d^\top  J_{p,d}  X_1(x)\mathsf{S}_{12}(x,\epsilon) \mathfrak{J}_{p,r-d}^\top e_i \big ) \bigg]  X_1(x)\mathsf{S}_{12}(x,\epsilon) \mathfrak{J}_{p,r-d}^\top e_i \nonumber \\
& \quad+ \sum_{i=d+1}^r \bigg[\frac{P(x)}{2} \sum_{j=1}^d  \frac{\mu_{2e_j}(x,\epsilon)}{\epsilon^{d+3}} \mathfrak{N}_{jj}^\top (x) J_{p-d,r-d} X_2(x)\mathfrak{J}_{p,r-d}^\top e_i \bigg]X_1(x)\mathsf{S}_{12}(x,\epsilon) \mathfrak{J}_{p,r-d}^\top e_i+O(\epsilon),\nonumber  \\
&\quad \sum_{i=d+1}^r \bigg[ P(x) \frac{\mu_{e_d}(x,\epsilon)}{\epsilon^{d+3}} \big(e_d^\top  J_{p,d}  X_1(x)\mathsf{S}_{12}(x,\epsilon) \mathfrak{J}_{p,r-d}^\top e_i \big ) \bigg]  J_{p-d,r-d} X_2(x)\mathfrak{J}_{p,r-d}^\top e_i \nonumber  \\
& + \sum_{i=d+1}^r \bigg[\frac{P(x)}{2} \sum_{j=1}^d  \frac{\mu_{2e_j}(x,\epsilon)}{\epsilon^{d+3}} \mathfrak{N}_{jj}^\top (x) J_{p-d,r-d} X_2(x)\mathfrak{J}_{p,r-d}^\top e_i \bigg] J_{p-d,r-d} X_2(x)\mathfrak{J}_{p,r-d}^\top e_i+O(1)\Big]\!\!\!\Big] \nonumber \\
= &\,  \Big[\!\!\!\Big[ \sum_{i=d+1}^r \bigg\{\frac{P(x)}{2 \epsilon^{d+3}} \sum_{j=1}^d \Big[\Big(\mu_{2e_j}(x,\epsilon)-\frac{\mu_{e_d}(x,\epsilon)\mu_{2e_j+e_d}(x,\epsilon)}{\mu_{2e_d}(x,\epsilon)}\Big)\mathfrak{N}_{jj}^\top (x) \Big]J_{p-d,r-d} X_2(x)\mathfrak{J}_{p,r-d}^\top e_i \bigg\}\nonumber\\
&\qquad\times X_1(x)\mathsf{S}_{12}(x,\epsilon) \mathfrak{J}_{p,r-d}^\top e_i+O(\epsilon),\nonumber  \\
&\quad \sum_{i=d+1}^r \bigg\{\frac{P(x)}{2 \epsilon^{d+3}} \sum_{j=1}^d \Big[\Big(\mu_{2e_j}(x,\epsilon)-\frac{\mu_{e_d}(x,\epsilon)\mu_{2e_j+e_d}(x,\epsilon)}{\mu_{2e_d}(x,\epsilon)}\Big)\mathfrak{N}_{jj}^\top (x) \Big] \nonumber\\
&\qquad\times J_{p-d,r-d} X_2(x)\mathfrak{J}_{p,r-d}^\top e_i \bigg\} J_{p-d,r-d} X_2(x)\mathfrak{J}_{p,r-d}^\top e_i+O(1)\Big]\!\!\!\Big]  \nonumber\,,
\end{align}
where we use (\ref{X1S12 d}) in the last step.
To simplify the tangent and normal components of \\ $\sum_{i=d+1}^r \frac{\mathbb{E}[(\iota(X)-\iota(x))\chi_{B_{\epsilon}^{\mathbb{R}^p}(\iota(x))}(\iota(X))]^\top \beta_i \beta_i^\top}{\lambda_i +\epsilon^{d+3}}$, we need the following formula. Suppose $v \in \mathbb{R}^{r-d}$, $G \in \mathbb{R}^{d \times (r-d)}$ with $e_i^\top  J_{p,d}G= w_i^\top$ for $i=1,\cdots, d$.  By Lemma \ref{Lemma:7}, $X_2(x) \in O(r-d)$. We can represent the inner product between $v$ and $w_i$ in orthonormal basis formed by the column vectors of $X_2(x)$.
\begin{align}\label{inner product representation}
\sum_{i=d+1}^r \big[v^\top X_2(x)\mathfrak{J}_{p,r-d}^\top e_i \big] G X_2(x)\mathfrak{J}_{p,r-d}^\top e_i= \sum_{i=1}^d v^\top w_i  J_{p,d}^\top e_i.
\end{align}
By (\ref{X1S12 d-1}), the tangent component of $\sum_{i=d+1}^r \frac{\mathbb{E}[(\iota(X)-\iota(x))\chi_{B_{\epsilon}^{\mathbb{R}^p}(\iota(x))}(\iota(X))]^\top \beta_i \beta_i^\top}{\lambda_i +\epsilon^{d+3}}$ is
\begin{align}
& \sum_{i=d+1}^r \bigg\{\frac{P(x)}{2 \epsilon^{d+3}} \sum_{j=1}^d \Big[\Big(\mu_{2e_j}(x,\epsilon)-\frac{\mu_{e_d}(x,\epsilon)\mu_{2e_j+e_d}(x,\epsilon)}{\mu_{2e_d}(x,\epsilon)}\Big)\mathfrak{N}_{jj}^\top (x) \Big]J_{p-d,r-d} X_2(x)\mathfrak{J}_{p,r-d}^\top e_i \bigg\}\nonumber\\
&\qquad\times X_1(x)\mathsf{S}_{12}(x,\epsilon) \mathfrak{J}_{p,r-d}^\top e_i \nonumber\\
= & \frac{P(x)}{2 \epsilon^{d+3}} \sum_{i=1}^{d-1} \sum_{j=1}^d \Big[\Big(\frac{\mu_{e_d}(x,\epsilon)\mu_{2e_j+e_d}(x,\epsilon)}{\mu_{2e_d}(x,\epsilon)}-\mu_{2e_j}(x,\epsilon)\Big)\frac{\mu_{2e_i+e_d}(x,\epsilon)}{\mu_{2e_i}(x,\epsilon)}\mathfrak{N}_{jj}^\top (x) \Big]\nonumber\\
&\qquad\times J_{p-d,r-d} J_{p-d,r-d} ^\top \mathfrak{N}_{id}(x) J_{p,d}^\top e_i \nonumber \\
 &\,+ \frac{P(x)}{4 \epsilon^{d+3}}\sum_{i=1}^d \sum_{j=1}^d \big[\big(\frac{\mu_{e_d}(x,\epsilon)\mu_{2e_j+e_d}(x,\epsilon)}{\mu_{2e_d}(x,\epsilon)}-\mu_{2e_j}(x,\epsilon)\big)\frac{\mu_{2e_i+e_d}(x,\epsilon)}{\mu_{2e_d}(x,\epsilon)}\mathfrak{N}_{jj}^\top (x) \big]\nonumber\\
 &\qquad\times J_{p-d,r-d} J_{p-d,r-d} ^\top \mathfrak{N}_{ii}(x) J_{p,d}^\top e_d \nonumber \\
= & \frac{P(x)}{2 \epsilon^{d+3}} \sum_{i=1}^{d-1} \sum_{j=1}^d \big[\big(\frac{\mu_{e_d}(x,\epsilon)\mu_{2e_j+e_d}(x,\epsilon)}{\mu_{2e_d}(x,\epsilon)}-\mu_{2e_j}(x,\epsilon)\big)\frac{\mu_{2e_i+e_d}(x,\epsilon)}{\mu_{2e_i}(x,\epsilon)}\mathfrak{N}_{jj}^\top (x) \big]\mathfrak{N}_{id}(x) J_{p,d}^\top e_i \nonumber \\
&\,+ \frac{P(x)}{4 \epsilon^{d+3}}\sum_{i=1}^d \sum_{j=1}^d \big[\big(\frac{\mu_{e_d}(x,\epsilon)\mu_{2e_j+e_d}(x,\epsilon)}{\mu_{2e_d}(x,\epsilon)}-\mu_{2e_j}(x,\epsilon)\big)\frac{\mu_{2e_i+e_d}(x,\epsilon)}{\mu_{2e_d}(x,\epsilon)}\mathfrak{N}_{jj}^\top (x) \big]\mathfrak{N}_{ii}(x) J_{p,d}^\top e_d \nonumber,
\end{align}
where in the first step we apply equations (\ref{X1S12 d-1}), (\ref{X1S12 d}) and (\ref{inner product representation}). In the last step, $e_m^\top \Second(e_i,e_j)=0$ for $m=r+1, \cdots, p$, and $i,j=1, \cdots, d$. Hence, 
\begin{align}
\mathfrak{N}_{jj}^\top (x) J_{p-d,r-d} J_{p-d,r-d} ^\top \mathfrak{N}_{ii}(x)=\mathfrak{N}_{jj}^\top (x)\mathfrak{N}_{ii}(x).
\end{align}

By Lemma \ref{Lemma:7}, we have $X_2(x) \in O(r-d)$, and $e_m^\top \Second(e_j,e_j)=0$ for $m=r+1, \cdots, p$, and $j=1, \cdots, d$. Hence, (\ref{inner product representation}) implies that
\begin{align}
\sum_{i=d+1}^r [ \mathfrak{N}_{jj}^\top (x) J_{p-d,r-d} X_2(x)\mathfrak{J}_{p,r-d}^\top e_i ] J_{p-d,r-d} X_2(x)\mathfrak{J}_{p,r-d}^\top e_i= \mathfrak{N}_{jj}.
\end{align}
We use it to simplify the normal component $\sum_{i=d+1}^r \frac{\mathbb{E}[(\iota(X)-\iota(x))\chi_{B_{\epsilon}^{\mathbb{R}^p}(\iota(x))}(\iota(X))]^\top \beta_i \beta_i^\top}{\lambda_i +\epsilon^{d+3}}$.
We have 
\begin{align}
& \sum_{i=d+1}^r \frac{\mathbb{E}[(\iota(X)-\iota(x))\chi_{B_{\epsilon}^{\mathbb{R}^p}(\iota(x))}(\iota(X))]^\top \beta_i \beta_i^\top}{\lambda_i +\epsilon^{d+3}}\nonumber \\
= &\,  \Big[\!\!\!\Big[ \frac{P(x)}{2 \epsilon^{d+3}} \sum_{i=1}^{d-1} \sum_{j=1}^d \big[\big(\frac{\mu_{e_d}(x,\epsilon)\mu_{2e_j+e_d}(x,\epsilon)}{\mu_{2e_d}(x,\epsilon)}-\mu_{2e_j}(x,\epsilon)\big)\frac{\mu_{2e_i+e_d}(x,\epsilon)}{\mu_{2e_i}(x,\epsilon)}\mathfrak{N}_{jj}^\top (x) \big]\mathfrak{N}_{id}(x) J_{p,d}^\top e_i \nonumber \\
&\qquad+  \frac{P(x)}{4 \epsilon^{d+3}}\sum_{i=1}^d \sum_{j=1}^d \Big[\Big(\frac{\mu_{e_d}(x,\epsilon)\mu_{2e_j+e_d}(x,\epsilon)}{\mu_{2e_d}(x,\epsilon)}-\mu_{2e_j}(x,\epsilon)\Big)\frac{\mu_{2e_i+e_d}(x,\epsilon)}{\mu_{2e_d}(x,\epsilon)}\mathfrak{N}_{jj}^\top (x) \Big]\nonumber\\
&\qquad\qquad\times\mathfrak{N}_{ii}(x) J_{p,d}^\top e_d+O(\epsilon),\nonumber  \\
&\quad \frac{P(x)}{2 \epsilon^{d+3}} \sum_{j=1}^d \big(\mu_{2e_j}(x,\epsilon)-\frac{\mu_{e_d}\mu_{2e_j+e_d}}{\mu_{2e_d}}\big)\mathfrak{N}_{jj}+O(1)\Big]\!\!\!\Big]  \nonumber\,.
\end{align}
By summing up $\sum_{i=1}^d \frac{\mathbb{E}[(\iota(X)-\iota(x))\chi_{B_{\epsilon}^{\mathbb{R}^p}(\iota(x))}(\iota(X))] ^\top \beta_i \beta_i^\top}{\lambda_i+\epsilon^{d+3}}$ and $\sum_{i=d+1}^r \frac{\mathbb{E}[(\iota(X)-\iota(x))\chi_{B_{\epsilon}^{\mathbb{R}^p}(\iota(x))}(\iota(X))] ^\top \beta_i \beta_i^\top}{\lambda_i+\epsilon^{d+3}}$, we have the conclusion.

\end{proof}

\section{Bias analysis on the kernel of LLE and the associated integral operator} \label{proof of t1 and t2}

\subsection{Proof of Proposition \ref{proposition 1}}

\begin{enumerate}
\item
When $x \in M \setminus M_{\epsilon}$, $\mu_{e_d}=0$. By Lemma \ref{Lemma:8}, $\mathbf{T}(x)=[\![ O(1), O(\epsilon^{-1})]\!]$. If $\iota(y) \in B_{\epsilon}^{\mathbb{R}^p}(\iota(x))$, $\iota(y)-\iota(x)=[\![ O(\epsilon), O(\epsilon^2)]\!]$. So, $(\iota(y)-\iota(x))^\top \mathbf{T}(x)=O(\epsilon)$ and $K_{\epsilon}(x,y)=1-O(\epsilon)>0$ when $\epsilon$ is small enough. 
\item
When $x \in M_{\epsilon}$, $\mathbf{T}(x)=[\![ \frac{\mu_{e_d}(x,\epsilon)}{\mu_{2e_d}(x,\epsilon)}J_{p,d}^\top e_d+O(1), O(\epsilon^{-1})]\!]$ and 
\begin{align}
\iota(y)-\iota(x)=& [\![ \sum_{i=1}^d u_i e_i+O(\|u\|^3), O(\|u\|^2)]\!] =  [\![ \sum_{i=1}^d u_i e_i+O(\epsilon^3), O(\epsilon^2) ]\!]. 
\end{align}
Therefore, by Corollary \ref{sigma and mu},
\begin{equation}
K_{\epsilon}(x,y)=1- \frac{\sigma_{1,d}(\tilde{\epsilon}_x )u_d}{\sigma_{2,d}(\tilde{\epsilon}_x )\epsilon}+O(\epsilon).
\end{equation}
By definition, $- \frac{\sigma_{1,d}(\tilde{\epsilon}_x )}{\sigma_{2,d}(\tilde{\epsilon}_x )}>0$ and it is a decreasing function of $\tilde{\epsilon}_x $. Therefore, to discuss the infimum of $K_{\epsilon}(x,y)$,  it is sufficient to consider the case when $x \in \partial M$, i.e. when $\tilde{\epsilon}_x =0$.  
If $\tilde{\epsilon}_x =0$, then 
\begin{equation}
K_{\epsilon}(x,y)=1 +\Big[\frac{2d(d+2)|S^{d-2}|}{(d^2-1)|S^{d-1}| \epsilon}+O(1)\Big]u_d+O(\epsilon). 
\end{equation}
Hence, let $u^*_d=\inf u_d$ where the infimum is taken over $x \in \partial M$ and $\iota(y) \in B_{\epsilon}^{\mathbb{R}^p}(\iota(x))$, then if $\epsilon$ is small enough,
\begin{equation}
\inf_{x, y} K_{\epsilon}(x,y)=1 +\Big[\frac{2d(d+2)|S^{d-2}|}{(d^2-1)|S^{d-1}| \epsilon}+O(1)\Big]u^*_d+O(\epsilon). 
\end{equation}
Obviously, $u^*_d=-\epsilon+O(\epsilon^2)$.  Therefore, $\inf_{x, y} K_{\epsilon}(x,y)=1-\frac{2d(d+2)|S^{d-2}|}{(d^2-1)|S^{d-1}| }+O(\epsilon)$. It is worth to note that $\frac{2d(d+2)|S^{d-2}|}{(d^2-1)|S^{d-1}|}>1$ by Lemma \ref{ratio of spheres}. 

\item 
By Lemma \ref{Lemma:4} and part (1)
\begin{align}
&\mathbb{E}K_{\epsilon}(x,X)\nonumber \\
=\,&\int_{D(x)} (1- \frac{\mu_{e_d}(x,\epsilon)u_d}{\mu_{2e_d}(x,\epsilon)}+O(\epsilon)) (P(x)+O(u))(1+O(u^2)) du \nonumber\\
=\,&\int_{\tilde{D}(x)} (1- \frac{\mu_{e_d}(x,\epsilon)u_d}{\mu_{2e_d}(x,\epsilon)}+O(\epsilon)) (P(x)+O(u))(1+O(u^2)) du+O(\epsilon^{d+2}) \nonumber \\
=\,&P(x) \int_{\tilde{D}(x)} 1- \frac{\sigma_{1,d}(\tilde{\epsilon}_x )u_d}{\sigma_{2,d}(\tilde{\epsilon}_x )\epsilon} du+O(\epsilon^{d+1})  \nonumber 
\end{align}
Since $- \frac{\sigma_{1,d}(\tilde{\epsilon}_x )}{\sigma_{2,d}(\tilde{\epsilon}_x )}>0$ and it is a decreasing function of $\tilde{\epsilon}_x $, it suffice to show that  if $x \in \partial M$, then $\int_{\tilde{D}(x)} 1- \frac{\mu_{e_d}(x,\epsilon)u_d}{\mu_{2e_d}(x,\epsilon)} du \geq C(d) \epsilon^{d}$.

If $x\in \partial M$, then $1- \frac{\sigma_{1,d}(\tilde{\epsilon}_x )u_d}{\sigma_{2,d}(\tilde{\epsilon}_x )\epsilon}=1+\frac{2d(d+2)|S^{d-2}|u_d}{(d^2-1)|S^{d-1}\epsilon}$, and
\begin{align}
& \int_{\tilde{D}(x)}1- \frac{\sigma_{1,d}(\tilde{\epsilon}_x )u_d}{\sigma_{2,d}(\tilde{\epsilon}_x )\epsilon} du \nonumber\\
\geq &\, \frac{|S^{d-2}|}{d-1}\int_{-\epsilon}^0 [1 +\frac{2d(d+2)|S^{d-2}|u_d}{(d^2-1)|S^{d-1}| \epsilon}](\epsilon^2-u_d^2)^{\frac{d-1}{2}}du_d \nonumber \\
=&\,  \epsilon^{d} \frac{|S^{d-2}|}{d-1}\int_{0}^1 [1-\frac{2d(d+2)|S^{d-2}|a}{(d^2-1)|S^{d-1}|}](1-a^2)^{\frac{d-1}{2}}da \nonumber \\
=&\,   \epsilon^{d} \big[ \frac{|S^{d-2}|}{d-1} \int_{0}^1 (1-a^2)^{\frac{d-1}{2}}da-\frac{|S^{d-2}|}{d-1}\frac{2d(d+2)|S^{d-2}|}{(d^2-1)|S^{d-1}|}\int_{0}^1 a(1-a^2)^{\frac{d-1}{2}}da\big] \nonumber \\
=&\, \epsilon^{d} \big[ \frac{|S^{d-1}|}{2d}-\frac{2d(d+2)|S^{d-2}|^2}{(d^2-1)^2|S^{d-1}|}\big] \nonumber.
\end{align}
We have thus finished the proof since $\frac{|S^{d-1}|}{2d}-\frac{2d(d+2)|S^{d-2}|^2}{(d^2-1)^2|S^{d-1}|}>0$ for any $d$ following from Lemma \ref{ratio of spheres}.
\end{enumerate}

\subsection{Proof of Proposition \ref{properties of the coeffecients}}
When $d=1$, the  differentiability follows from the direct calculation. For $d>1$, the differentiability follows from the fundamental theorem of calculus. The rest of the statements follow directly from the definition of $\sigma$, except $\phi_{1} (\tilde{\epsilon}_x )>0$ and $\phi_{2} (\tilde{\epsilon}_x )<0$ when $\tilde{\epsilon}_x =0$. 

We now prove $\phi_{1} (\tilde{\epsilon}_x )>0$. When $\tilde{\epsilon}_x =0$,
\begin{align}
\sigma_{2,d}(0) \sigma_{2}(0)-\sigma_{3}(0) \sigma_{1,d}(0)=\frac{|S^{d-1}|^2}{4d^2(d+2)^2}-\frac{|S^{d-2}|^2}{(d^2-1)^2(d+3)},
\end{align}
which is positive since we have proved $\frac{|S^{d-2}|^2 }{|S^{d-1}|^2}<\frac{(d^2-1)^2}{4d^2(d+2)}$ in Lemma \ref{ratio of spheres}.  Note that $\sigma_{2,d}(\tilde{\epsilon}_x ) \sigma_{2}(\tilde{\epsilon}_x )-\sigma_{3}(\tilde{\epsilon}_x ) \sigma_{1,d}(\tilde{\epsilon}_x )$ is increasing when $0 \leq \tilde{\epsilon}_x  \leq \epsilon$. Hence, $\sigma_{2,d}(\tilde{\epsilon}_x ) \sigma_{2}(\tilde{\epsilon}_x )-\sigma_{3}(\tilde{\epsilon}_x ) \sigma_{1,d}(\tilde{\epsilon}_x )>0$.
Similarly, we can show that $\sigma_{2,d}(\tilde{\epsilon}_x ) \sigma_{0}(\tilde{\epsilon}_x ) -\sigma_{1,d}^2(\tilde{\epsilon}_x )>0$. Therefore, we conclude that $\phi_{1} (\tilde{\epsilon}_x )>0$.

Next, we study $\phi_{2}$. To prove $\phi_{2} (\tilde{\epsilon}_x )<0$ when $\tilde{\epsilon}_x =0$, it suffices to show $\sigma^2_{2,d}(0)-\sigma_{3,d}(0) \sigma_{1,d}(0)<0$ since we have shown $\sigma_{2,d}(\tilde{\epsilon}_x ) \sigma_{0}(\tilde{\epsilon}_x ) -\sigma_{1,d}^2(\tilde{\epsilon}_x )>0$ above. 
When $\tilde{\epsilon}_x =0$, we have
\begin{align}
\sigma^2_{2,d}(0)-\sigma_{3,d}(0) \sigma_{1,d}(0)=\frac{|S^{d-1}|^2}{4d^2(d+2)^2}-\frac{2|S^{d-2}|^2}{(d^2-1)^2(d+3)},
\end{align}
which is negative due to $\frac{|S^{d-2}|^2 }{|S^{d-1}|^2}>\frac{(d^2-1)^2(d+3)}{8d^2(d+2)^2}$ proved in Lemma \ref{ratio of spheres}.
We now check that $\sigma^2_{2,d}(\tilde{\epsilon}_x )-\sigma_{3,d}(\tilde{\epsilon}_x ) \sigma_{1,d}(\tilde{\epsilon}_x )>0$ when $\tilde{\epsilon}_x=\epsilon$. Since $\sigma^2_{2,d}(\tilde{\epsilon}_x )-\sigma_{3,d}(\tilde{\epsilon}_x ) \sigma_{1,d}(\tilde{\epsilon}_x )$ is an increasing continuous function of $\tilde{\epsilon}_x$, there is a unique $t^*=t^*(x)\in (0,\tilde{\epsilon}_x)$ such that $\sigma^2_{2,d}(\tilde{\epsilon}_x )-\sigma_{3,d}(\tilde{\epsilon}_x ) \sigma_{1,d}(\tilde{\epsilon}_x )=0$, and hence $\phi_{2} (t^*)=0$. We thus have
\begin{align}
&\Big[\frac{|S^{d-1}|}{2d(d+2)}+\frac{|S^{d-2}|}{d-1}\int_{0}^{\frac{t^*}{\epsilon}} (1-z^2)^{\frac{d-1}{2}}z^2 dz \Big]^2\label{solution of t^*}\\
=&\,\frac{|S^{d-2}|^2}{(d^2-1)^2(d+3)}\Big(2+(d+1)\big(\frac{t^*}{\epsilon}\big)^2\Big)\Big(1-\big(\frac{t^*}{\epsilon}\big)^2\Big)^{d+1}.\nonumber
\end{align}
Since $t^*$ does not depend on $x$, the set $\mathcal{S}$ is diffeomorphic to $\partial M$ when $\epsilon$ is sufficiently small. 
Since $0<\frac{t^*}{\epsilon}<1$, \eqref{solution of t^*} becomes
\begin{align}
\Big[\frac{|S^{d-1}|}{2d(d+2)}+\frac{|S^{d-2}|}{d-1}\int_{0}^{\frac{t^*}{\epsilon}} (1-z^2)^{\frac{d-1}{2}}z dz \Big]^2>\frac{2|S^{d-2}|^2}{(d^2-1)^2(d+3)}\Big(1-\big(\frac{t^*}{\epsilon}\big)^2\Big)^{d+1},
\end{align}
which is equivalent to 
\begin{align}
\frac{|S^{d-1}|}{2d(d+2)}+\frac{|S^{d-2}|}{d^2-1}\Big(1-\big(1-\big(\frac{t^*}{\epsilon}\big)^2\big)^{\frac{d+1}{2}}\Big) >\sqrt{\frac{2}{d+3}}\frac{|S^{d-2}|}{d^2-1}\Big(1-\big(\frac{t^*}{\epsilon}\big)^2\Big)^{\frac{d+1}{2}}.
\end{align}
If we isolate $t^*$ in the above equation, we have the lower bound for $t^*$:
\begin{align}
\Bigg(1-\bigg[\frac{1+\frac{(d^2-1)|S^{d-1}|}{2d(d+2)|S^{d-2}|}}{1+\sqrt{\frac{2}{d+3}}}\bigg]^{\frac{2}{d+1}}\Bigg)^{\frac{1}{2}} \epsilon<t^*.
\end{align}  
Note that by Lemma \ref{ratio of spheres}, $\frac{(d^2-1)|S^{d-1}|}{2d(d+2)|S^{d-2}|}<\sqrt{\frac{2}{d+3}}$, so $1-\bigg[\frac{1+\frac{(d^2-1)|S^{d-1}|}{2d(d+2)|S^{d-2}|}}{1+\sqrt{\frac{2}{d+3}}}\bigg]^{\frac{2}{d+1}}>0$.

Next, we find the upper bound of $t^*$. Since $\frac{t^*}{\epsilon}<1$, by \eqref{solution of t^*}, we have,
\begin{align}
&\Big[\frac{|S^{d-1}|}{2d(d+2)}+\frac{|S^{d-2}|}{d-1}\int_{0}^{\frac{t^*}{\epsilon}} (1-z^2)^{\frac{d-1}{2}}z^3 dx \Big]^2\\
<&\,\frac{|S^{d-2}|^2}{(d^2-1)^2(d+3)}\Big(2+(d+1)\big(\frac{t^*}{\epsilon}\big)^2\Big)\Big(1-\big(\frac{t^*}{\epsilon}\big)^2\Big)^{d+1},\nonumber
\end{align}
which is equivalent to
\begin{align}
\frac{|S^{d-1}|}{2d(d+2)} & +\frac{|S^{d-2}|}{(d^2-1)(d+3)}\Big(2-\big(1-\big(\frac{t^*}{\epsilon}\big)^2\big)^{\frac{d+1}{2}}\Big)\Big(2+(d+1)\big(\frac{t^*}{\epsilon}\big)^2\Big) \\
& <\frac{|S^{d-2}|}{(d^2-1)\sqrt{d+3}}\Big(2+(d+1)\big(\frac{t^*}{\epsilon}\big)^2\Big)^{\frac{1}{2}}\Big(1-\big(\frac{t^*}{\epsilon}\big)^2\Big)^{\frac{d+1}{2}}. \nonumber
\end{align}
If we isolate $t^*$ in the above equation, we have the lower bound, 
\begin{equation}
t^*<\Bigg(1-\bigg[\frac{(d^2-1)|S^{d-1}|}{4d(d+2)|S^{d-2}|}+\frac{1}{d+3}\bigg]^{\frac{2}{d+1}}\Bigg)^{\frac{1}{2}} \epsilon.
\end{equation} 
By  Lemma \ref{ratio of spheres} and the upper bound, $t^* \rightarrow 0$ as $d \rightarrow \infty$.

\subsection{Proof of Theorem \ref{Theorem:t1}}
In this proof, we caluclate the first two order terms in $R_\epsilon f(x)$.
First, we are going to calculate $\mathbb{E}[\chi_{B_{\epsilon}^{\mathbb{R}^p}(\iota(x))}(\iota(X))]-\mathbb{E}[(\iota(X)-\iota(x))\chi_{B_{\epsilon}^{\mathbb{R}^p}(\iota(x))}(\iota(X))]^\top \mathbf{T}(x)$ and show that it is dominated by the order $\epsilon^{d}$ terms. Then we are going to calculate $\mathbb{E}[(f(X)-f(x))\chi_{B_{\epsilon}^{\mathbb{R}^p}(\iota(x))}(\iota(X))]-\mathbb{E}[(\iota(X)-\iota(x))(f(X)-f(x))\chi_{B_{\epsilon}^{\mathbb{R}^p}(\iota(x))}(\iota(X))] ^\top \mathbf{T}(x)$ and show that it is dominated by the order $\epsilon^{d+2}$ terms. Hence their ratio is dominated by the order $\epsilon^2$ terms. 

By Lemma \ref{Lemma:6} and Lemma \ref{Lemma:8}, we have
\begin{align} 
\mathbb{E}[\chi_{B_{\epsilon}^{\mathbb{R}^p}(\iota(x))}(\iota(X))]&=P(x) \mu_0(x,\epsilon)+O(\epsilon^{d+1}), \\
\mathbb{E}[(\iota(X)-\iota(x))\chi_{B_{\epsilon}^{\mathbb{R}^p}(\iota(x))}(\iota(X))]&=[\![P(x) \mu_{e_d}(x,\epsilon) J_{p,d}^\top  e_d+O(\epsilon^{d+2}),O(\epsilon^{d+2}) ]\!], \nonumber
\end{align}
and
$\mathbf{T}(x)=[\![ v^{(-1)}_1+ v^{(0)}_{1,1}+v^{(0)}_{1,2}+ v^{(0)}_{1,3}+v^{(0)}_{1,4}, v^{(-1)}_2]\!]+[\![O(\epsilon),O(1) ]\!]$, where
\begin{align}
& v^{(-1)}_1= \frac{\mu_{e_d}(x,\epsilon)}{\mu_{2e_d}(x,\epsilon)}J_{p,d}^\top e_d, \quad v^{(0)}_{1,1}= \frac{\nabla P(x)}{P(x)}, 
\end{align}
and $v^{(0)}_{1,2}$, $v^{(0)}_{1,3}$, $v^{(0)}_{1,4}$ and $v^{(-1)}_2$ are defined in Lemma \ref{Lemma:8}. Moreover, $v^{(0)}_{1,2}$, $v^{(0)}_{1,3}$ and $v^{(0)}_{1,4}$ are of order $1$  and $v^{(-1)}_2$ is of order $\epsilon^{-1}$.
Hence,
\begin{align} 
&\mathbb{E}[\chi_{B_{\epsilon}^{\mathbb{R}^p}(\iota(x))}(\iota(X))]-\mathbb{E}[(\iota(X)-\iota(x))\chi_{B_{\epsilon}^{\mathbb{R}^p}(\iota(x))}(\iota(X))]^{\top}\mathbf{T}(x) \\
=&\,P(x)\Big[ \mu_0(x,\epsilon) -\frac{\mu_{e_d}(x,\epsilon)^2}{\mu_{2e_d}(x,\epsilon)}\Big]+O(\epsilon^{d+1})\,, \nonumber \\
=&\,P(x)\Big[\frac{ \mu_0(x,\epsilon)\mu_{2e_d}(x,\epsilon) -\mu_{e_d}(x,\epsilon)^2}{\mu_{2e_d}(x,\epsilon)}\Big]+O(\epsilon^{d+1})\,, \nonumber
\end{align} 
where the leading term in above expression is of order $\epsilon^d$ by Lemma \ref{Lemma:5}.
Based on Lemma \ref{Lemma:6}, we have
\begin{align}
& \mathbb{E}[(f(X)-f(x))\chi_{B_{\epsilon}^{\mathbb{R}^p}(\iota(x))}(\iota(X))] \\
=&\,   P(x)  \partial_d f(x) \mu_{e_d}(x,\epsilon)+\sum_{i=1}^d\big[\frac{P(x)}{2}\partial^2_{ii}f(x) +\partial_i f(x) \partial_i P(x) \big]\mu_{2e_i}(x,\epsilon)+O(\epsilon^{d+3}),\nonumber
\end{align} 
and
\begin{equation}
\mathbb{E}[(\iota(X)-\iota(x))(f(X)-f(x))\chi_{B_{\epsilon}^{\mathbb{R}^p}(\iota(x))}(\iota(X))]=[\![v_1,v_2]\!]\,,
\end{equation}
where
\begin{align}
v_1=&\,P(x) \sum_{i=1}^d \big(\partial_i f(x) \mu_{2e_i}(x, \epsilon) \big) \,J_{p,d}^\top e_i \nonumber\\
&+ \sum_{i=1}^{d-1}\big[ \partial_i f(x)\partial_d P(x)+ \partial_d f(x)\ \partial_i P(x)+P(x)\partial^2_{id} f(x) \big]\mu_{2e_i+e_d}(x,\epsilon)  \,J_{p,d}^\top e_i \nonumber\\
&+\sum_{i=1}^{d}\Big(\big[ \partial_i f(x)\partial_i P(x) +\frac{P(x)}{2}\partial^2_{ii} f(x)\big]\mu_{2e_i+e_d}(x,\epsilon)   \Big)\,J_{p,d}^\top e_d+O(\epsilon^{d+4}), \nonumber \\
v_2 =&\, P(x) \sum_{i=1}^{d-1} \partial_i f(x) \mathfrak{N}_{id}(x) \mu_{2e_i+e_d} (x, \epsilon)+ \frac{P(x)}{2}\partial_d f(x)\sum_{i=1}^d \mathfrak{N}_{ii}(x) \mu_{2e_i+e_d}(x, \epsilon)+O(\epsilon^{d+4}).\nonumber
\end{align}
Therefore, we have
\begin{align}
& \mathbb{E}[(\iota(X)-\iota(x))(f(X)-f(x))\chi_{B_{\epsilon}^{\mathbb{R}^p}(\iota(x))}(\iota(X))]^{\top}\mathbf{T}(x) \\
=\,&P(x) \sum_{i=1}^d \big(\partial_i f(x) \mu_{2e_i}(x, \epsilon) \big) \, v^{(-1) \top}_1 J_{p,d}^\top e_i\nonumber\\
&+P(x) \sum_{i=1}^d \big(\partial_i f(x) \mu_{2e_i}(x, \epsilon) \big) \, v^{(0) \top}_{1,1} J_{p,d}^\top e_i \nonumber \\
&+  P(x) \sum_{i=1}^d \big(\partial_i f(x) \mu_{2e_i}(x, \epsilon) \big) \, [v^{(0)}_{1,2}+v^{(0)}_{1,3}+v^{(0)}_{1,4}]^\top J_{p,d}^\top e_i  \nonumber \\
&+ \sum_{i=1}^{d-1}\big[ \partial_i f(x)\partial_d P(x)+ \partial_d f(x)\ \partial_i P(x)+P(x)\partial^2_{id} f(x) \big]\mu_{2e_i+e_d}(x,\epsilon)  \, v^{(-1) \top}_1 J_{p,d}^\top e_i  \nonumber \\
&+ \sum_{i=1}^{d} \big[ \partial_i f(x)\partial_i P(x) +\frac{P(x)}{2}\partial^2_{ii} f(x)\big]\mu_{2e_i+e_d}(x,\epsilon) \, v^{(-1) \top}_1 J_{p,d}^\top e_d \nonumber \\
&+ P(x) \sum_{i=1}^{d-1} \partial_i f(x) \mu_{2e_i+e_d}(x, \epsilon) v^{(-1) \top}_2 \mathfrak{N}_{id}(x) \nonumber\\
&+ \frac{P(x)}{2}\partial_d f(x)  \sum_{i=1}^d \mu_{2e_i+e_d}(x, \epsilon) v^{(-1) \top}_2  \mathfrak{N}_{ii}(x) + O(\epsilon^{d+3})\nonumber\,.
\end{align}
Note that by Lemma \ref{Lemma:5}, the first term is of order $\epsilon^{d+1}$ and the second to seventh terms are of order $\epsilon^{d+2}$. Furthermore, we can simplify the first and the second term as:
\begin{align}
& P(x) \sum_{i=1}^d \big(\partial_i f(x) \mu_{2e_i}(x, \epsilon) \big) \, v^{(-1) \top}_1 J_{p,d}^\top e_i= P(x)  \partial_d f(x) \mu_{e_d}(x,\epsilon) \\
& P(x) \sum_{i=1}^d \big(\partial_i f(x) \mu_{2e_i}(x, \epsilon) \big) \, v^{(0) \top}_{1,1} J_{p,d}^\top e_i=
\sum_{i=1}^d \partial_i f(x) \partial_i P(x) \mu_{2e_i}(x,\epsilon)\nonumber\,.
\end{align}
Next we calculate $\mathbb{E}[(f(X)-f(x))\chi_{B_{\epsilon}^{\mathbb{R}^p}(\iota(x))}(\iota(X))] - \mathbb{E}[(\iota(X)-\iota(x))(f(X)-f(x))\chi_{B_{\epsilon}^{\mathbb{R}^p}(\iota(x))}(\iota(X))]^{\top}\mathbf{T}(x)$. Clearly, the common terms, $P(x)\partial_d f(x) \mu_{e_d}(x, \epsilon)$ and $\sum_{i=1}^d \partial_i f(x) \partial_i P(x)\mu_{2e_i}(x,\epsilon)$, are canceled, and hence only terms of order $\epsilon^{d+2}$ are left in the difference; that is, we have
\begin{align}
 & \mathbb{E}[(f(X)-f(x))\chi_{B_{\epsilon}^{\mathbb{R}^p}(\iota(x))}(\iota(X))] - \mathbb{E}[(\iota(X)-\iota(x))(f(X)-f(x))\chi_{B_{\epsilon}^{\mathbb{R}^p}(\iota(x))}(\iota(X))]^{\top}\mathbf{T}(x)\nonumber \\
=&\, \frac{P(x)}{2} \sum_{i=1}^d \partial^2_{ii}f(x)\mu_{2e_i}(x,\epsilon) - P(x) \sum_{i=1}^d \big(\partial_i f(x) \mu_{2e_i}(x, \epsilon) \big) \, [v^{(0)}_{1,2}+v^{(0)}_{1,3}+v^{(0)}_{1,4}]^\top J_{p,d}^\top e_i  \nonumber \\
&\quad- \sum_{i=1}^{d-1}\big[ \partial_i f(x)\partial_d P(x)+ \partial_d f(x)\ \partial_i P(x)+P(x)\partial^2_{id} f(x) \big]\mu_{2e_i+e_d}(x,\epsilon)  \, v^{(-1) \top}_1 J_{p,d}^\top e_i  \nonumber \\
&\quad- \sum_{i=1}^{d} \big[ \partial_i f(x)\partial_i P(x) +\frac{P(x)}{2}\partial^2_{ii} f(x)\big]\mu_{2e_i+e_d}(x,\epsilon) \, v^{(-1) \top}_1 J_{p,d}^\top e_d \nonumber \\
&\quad- P(x) \sum_{i=1}^{d-1} \partial_i f(x) \mu_{2e_i+e_d}(x, \epsilon) v^{(-1) \top}_2 \mathfrak{N}_{id}(x) \nonumber\\
&\quad- \frac{P(x)}{2}\partial_d f(x)   \sum_{i=1}^d \mu_{2e_i+e_d}(x, \epsilon) v^{(-1) \top}_2  \mathfrak{N}_{ii}(x) + O(\epsilon^{d+3})\nonumber\,.
\end{align}

Next, we simplify the above expression. Note that $v^{(-1) \top}_1 J_{p,d}^\top e_i = \frac{\mu_{e_d}(x,\epsilon)}{\mu_{2e_d}(x,\epsilon)}$ if $i=d$, and it is $0$ otherwise. Hence,
\begin{align}
-& \sum_{i=1}^{d-1}\big[ \partial_i f(x)\partial_d P(x)+ \partial_d f(x)\ \partial_i P(x)+P(x)\partial^2_{id} f(x) \big]\mu_{2e_i+e_d}(x,\epsilon)  \, v^{(-1) \top}_1 J_{p,d}^\top e_i=0\nonumber
\end{align}
and by definition of $v^{(0)}_{1,3}$ and $v^{(-1)}_1$, we have
\begin{align}
P(x)\mu_{2e_i}(x, \epsilon)v^{(0)^\top}_{1,3} J_{p,d}^\top e_i+\partial_i P(x) \mu_{2e_i+e_d}(x,\epsilon) \, v^{(-1) \top}_1 J_{p,d}^\top e_d=0\,.
\end{align}
For $i=1, \cdots, d-1$, by definition of $v^{(0)}_{1,4}$ and $v^{(-1)}_2$, we have 
\begin{align}
P(x) \mu_{2e_i}(x, \epsilon) \,v^{(0)\top}_{1,4} J_{p,d}^\top e_i+P(x)\mu_{2e_i+e_d}(x, \epsilon) v^{(-1) \top}_2 \mathfrak{N}_{id}(x) =0
\end{align}
and
\begin{align}
&P(x) \mu_{2e_d}(x, \epsilon) \,v^{(0)\top}_{1,4} J_{p,d}^\top e_d+\frac{P(x)}{2} \sum_{i=1}^d \mu_{2e_i+e_d}(x, \epsilon) v^{(-1) \top}_2  \mathfrak{N}_{ii}(x) =0.
\end{align}
Moreover, we have  $v^{(0) \top}_{1,2} J_{p,d}^\top e_i = - \frac{\mu_{e_d}(x,\epsilon) \epsilon^{d+3}}{P(x) (\mu_{2e_d}(x,\epsilon))^2}$ if $i=d$, and it is $0$ otherwise. 
Therefore,
\begin{align}
 & \mathbb{E}[(f(X)-f(x))\chi_{B_{\epsilon}^{\mathbb{R}^p}(\iota(x))}(\iota(X))] - \mathbb{E}[(\iota(X)-\iota(x))(f(X)-f(x))\chi_{B_{\epsilon}^{\mathbb{R}^p}(\iota(x))}(\iota(X))]^{\top}\mathbf{T}(x)\nonumber \\
=&\, \frac{P(x)}{2} \sum_{i=1}^d \partial^2_{ii}f(x)\big[\mu_{2e_i}(x,\epsilon)-\mu_{2e_i+e_d}(x,\epsilon) \, \frac{\mu_{e_d}(x,\epsilon)}{\mu_{2e_d}(x,\epsilon)}\big]+\partial_d f(x) \big(\frac{\mu_{e_d}(x,\epsilon) \epsilon^{d+3}}{\mu_{2e_d}(x,\epsilon)}\big) \,.\nonumber 
\end{align}
Therefore, the ratio 
\begin{align}
& \frac{\mathbb{E}[(f(X)-f(x))\chi_{B_{\epsilon}^{\mathbb{R}^p}(\iota(x))}(\iota(X))]-\mathbb{E}[(\iota(X)-\iota(x))(f(X)-f(x))\chi_{B_{\epsilon}^{\mathbb{R}^p}(\iota(x))}(\iota(X))] ^\top \mathbf{T}(x)}{\mathbb{E}[\chi_{B_{\epsilon}^{\mathbb{R}^p}(\iota(x))}(\iota(X))]-\mathbb{E}[(\iota(X)-\iota(x))\chi_{B_{\epsilon}^{\mathbb{R}^p}(\iota(x))}(\iota(X))] ^\top \mathbf{T}(x)} \\
=&\, \sum_{i=1}^d \partial^2_{ii}f(x) \Big[\frac{\mu_{2e_i}(x,\epsilon) \mu_{2e_d}(x,\epsilon)-\mu_{2e_i+e_d}(x,\epsilon) \mu_{e_d}(x,\epsilon)}{2\mu_0(x,\epsilon)\mu_{2e_d}(x,\epsilon)-2\mu_{e_d}(x,\epsilon)^2}\Big] \nonumber \\
&+\partial_d f(x)  \frac{\mu_{e_d}(x,\epsilon) \epsilon^{d+3}}{P(x)\big(\mu_0(x,\epsilon)\mu_{2e_d}(x,\epsilon) -\mu_{e_d}(x,\epsilon)^2\big)} + O(\epsilon^3). \nonumber
\end{align}
And the conclusion follows by substituting terms and in Corollary \ref{sigma and mu}.

\section{Variance analysis on LLE} \label{proof of theorem t0}
For simplicity of notations, for each $x_k$, denote 
\[
\boldsymbol{f}:=(f(x_{k,1}),f(x_{k,2}),\ldots,f(x_{k,N}))^{\top}\in\mathbb{R}^N. 
\]
By a direct expansion of equations (\ref{Definition:Irho:Soft}), (\ref{Definition:Tn}), (\ref{Expansion:LLEweightedKernel}) and $c=n\epsilon^{d+3}$, we have
\begin{align}
\sum_{j=1}^n [W-I_{n \times n}]_{kj}f(x_{j})&=\frac{\boldsymbol{1}_N^{\top}\boldsymbol{f} - \boldsymbol{1}^{\top}_NG_{n,k}^\top  U_nI _{p,r_n}  (\Lambda_n+n \epsilon^{d+3} I_{p\times p})^{-1} U_n^\top  G_{n,k}\boldsymbol{f}}
{N -  \boldsymbol{1}^{\top}_NG_{n,k}^\top  U_n I _{p,r_n}  (\Lambda_n+n \epsilon^{d+3} I_{p\times p})^{-1} U_n^\top G_{n,k}\boldsymbol{1}_N }-f(x_k),
\end{align}
which can be rewritten as $\frac{g_{n,1}}{g_{n,2}}$, where
\begin{align}
g_{n,1}:=&\frac{1}{n\epsilon^d}\sum_{j=1}^N(f(x_{k,j})-f(x_k))- [\frac{1}{n\epsilon^d}\sum_{j=1}^N(x_{k,j}-x_k)]^{\top}U_n I_{p,r_n}(\frac{\Lambda_n}{n\epsilon^d}+\epsilon^{3} I_{p\times p})^{-1}\nonumber\\
&\qquad\times U_n^\top [\frac{1}{n\epsilon^d}\sum_{j=1}^N(x_{k,j}-x_k)(f(x_{k,j})-f(x_k))]\nonumber\\
g_{n,2}:=&\frac{N}{n\epsilon^d} -  [ \frac{1}{n\epsilon^d}\sum_{j=1}^N(x_{k,j}-x_k)]^\top U_n I_{p,r_n}(\frac{\Lambda_n}{n\epsilon^d}+\epsilon^{3} I_{p\times p})^{-1}U_n^\top [ \frac{1}{n\epsilon^d}\sum_{j=1}^N(x_{k,j}-x_k)]\,. \nonumber
\end{align}
The goal is to relate the finite sum quantity $\frac{g_{n,1}}{g_{n,2}}$ to $Q_{\epsilon}f(x_k):=\frac{g_1}{g_2}$, 
where
\begin{align}
g_1=\,& \mathbb{E}\left[\frac{1}{\epsilon^d}\chi_{B_{\epsilon}^{\mathbb{R}^p}(\iota(x_k))}(X)(f(X)-f(x_k))\right]-\mathbb{E}\left[\frac{1}{\epsilon^d}(\iota(X)-\iota(x_k))\chi_{B_{\epsilon}^{\mathbb{R}^p}(\iota(x_k))}(X)\right]^{\top}  \\
& \times\left(U I_{p,r} \left(\frac{\Lambda}{\epsilon^d}+\epsilon^{3} I_{p\times p}\right)^{-1} U^\top\right) \mathbb{E}\left[\frac{1}{\epsilon^d}(\iota(X)-\iota(x_k))\chi_{B_{\epsilon}^{\mathbb{R}^p}(\iota(x_k))}(X)(f(X)-f(x_k))\right] \nonumber
\end{align}
and
\begin{align}
g_2=\,& \mathbb{E}\left[\frac{1}{\epsilon^d} \chi_{B_{\epsilon}^{\mathbb{R}^p}(\iota(x_k))}(X)\right]- \mathbb{E}\left[\frac{1}{\epsilon^d} (\iota(X)-\iota(x_k))\chi_{B_{\epsilon}^{\mathbb{R}^p}(\iota(x_k))}(X)\right]^{\top} \\
&  \times \left(U I_{p,r} \left(\frac{\Lambda}{\epsilon^d}+\epsilon^{3} I_{p\times p}\right)^{-1} U^\top \right) \mathbb{E}\left[\frac{1}{\epsilon^d}(\iota(X)-\iota(x_k))\chi_{B_{\epsilon}^{\mathbb{R}^p}(\iota(x_k))}(X)\right]\,. \nonumber
\end{align}
We now control the size of the fluctuation of the following four terms
\begin{align}
&\frac{1}{n\epsilon^d}\sum_{j=1}^N1\label{Proof:Theorem1:FiniteSum0}\\
&\frac{1}{n\epsilon^{d}}\sum_{j=1}^N(f(x_{k,j})-f(x_k)) 
\label{Proof:Theorem1:FiniteSum1}\\
&\frac{1}{n\epsilon^{d}}\sum_{j=1}^N(x_{k,j}-x_k) \label{Proof:Theorem1:FiniteSum2}\\
&\frac{1}{n\epsilon^{d}}\sum_{j=1}^N(x_{k,j}-x_k)(f(x_{k,j})-f(x_k)) 
\label{Proof:Theorem1:FiniteSum3}
\end{align}
as a function of $n$ and $\epsilon$ by the Bernstein type inequality. Here, we put $\epsilon^{-d}$ in front of each term to normalize the kernel so that the computation is consistent with the existing literature, like \cite{Cheng_Wu:2013,Singer_Wu:2016}. 

The size of the fluctuation of these terms are controlled in the following Lemmas. The term (\ref{Proof:Theorem1:FiniteSum0}) is the usual kernel density estimation, so we have the following lemma.

\begin{lemma}\label{Proof:Theorem1:LemmaF1}
Suppose $\epsilon=\epsilon(n)$ so that $\frac{\sqrt{\log(n)}}{n^{1/2}\epsilon^{d/2+1}}\to 0$ and $\epsilon\to 0$ as $n\to \infty$. We have with probability greater than $1-n^{-2}$ that for all $k=1,\ldots,n$, 
\begin{equation}
\left|\frac{1}{n\epsilon^d}\sum_{j=1}^N1 - \mathbb{E}\frac{1}{\epsilon^d}\chi_{{B}^{\mathbb{R}^p}_\epsilon(x_k)}(\iota(X))\right| =  O\Big(\frac{\sqrt{\log (n)}}{n^{1/2}\epsilon^{d/2}}\Big)\,.\nonumber
\end{equation}
\end{lemma}

Denote $\Omega_0$ to be the event space that above Lemma is satisfied. The behavior of (\ref{Proof:Theorem1:FiniteSum1}) is summarized in the following Lemma. 

\begin{lemma}\label{Proof:Theorem1:LemmaF2}
Suppose $\epsilon=\epsilon(n)$ so that $\frac{\sqrt{\log(n)}}{n^{1/2}\epsilon^{d/2+1}}\to 0$ and $\epsilon\to 0$ as $n\to \infty$. We have with probability greater than $1-n^{-2}$ that for all $k=1,\ldots,n$,
\begin{equation}
\left|\frac{1}{n\epsilon^d}\sum_{j=1}^N(f(x_{k,j})-f(x_k)) - \mathbb{E}\frac{1}{\epsilon^d}(f(X)-f(x_k))\chi_{{B}^{\mathbb{R}^p}_\epsilon(x_k)}(\iota(X))\right| =  O\Big(\frac{\sqrt{\log (n)}}{n^{1/2}\epsilon^{d/2-1}}\Big)\,.\nonumber
\end{equation}
\end{lemma}

\begin{proof}
By denoting 
\begin{equation}
F_{1,j}=\frac{1}{\epsilon^d}(f(x_{j})-f(x_k))\chi_{{B}^{\mathbb{R}^p}_\epsilon(x_k)}(x_{j}), 
\end{equation}
we have
\begin{equation}
\frac{1}{n\epsilon^d}\sum_{j=1}^N(f(x_{k,j})-f(x_k))=\frac{1}{n}\sum_{j\neq k, j=1}^nF_{1,j}.
\end{equation}
Define a random variable
\begin{equation}
F_{1}:=\frac{1}{\epsilon^d}(f(X)-f(x_k))\chi_{{B}^{\mathbb{R}^p}_\epsilon(x_k)}(\iota(X)).
\end{equation}
Clearly, when $j\neq k$, $F_{1,j}$ can be viewed as randomly sampled i.i.d. from $F_{1}$.
Note that we have
\begin{equation}
\frac{1}{n}\sum_{j\neq k, j=1}^nF_{1,j}=\frac{n-1}{n}\left[\frac{1}{n-1}\sum_{j\neq k, j=1}^nF_{1,j}\right] \,.
\end{equation}
Since $\frac{n-1}{n}\to 1$ as $n\to\infty$, the error incurred by replacing $\frac{1}{n}$ by $\frac{1}{n-1}$ is of order $\frac{1}{n}$, which is negligible asymptotically, we can simply focus on analyzing $\frac{1}{n-1}\sum_{j=1,j\neq i}^n F_{1,j}$. 
We have by Lemma \ref{Lemma:5} and Lemma \ref{Lemma:6},
\begin{align}
\mathbb{E}[F_{1}]   =\,& O( \epsilon) \quad \mbox{if  $x \in M_{\epsilon}$}\\
\mathbb{E}[F_{1}]   =\,& O( \epsilon^2) \quad \mbox{if  $x \not \in M_{\epsilon}$}\nonumber
\end{align}
and
\begin{align}
\mathbb{E}[F_{1}^2]  =\,&  \sum_{i=1}^dP(x_k) (\partial_{i}f(x_k))^2 \mu_{2e_i}(x_k,\epsilon)\epsilon^{-2d}+O(\epsilon^{-d+3}),
\end{align}
By Lemma \ref{Lemma:5},  $\frac{|S^{d-1}|}{2d(d+2)}\epsilon^{-d+2}+O(\epsilon^{-d+3}) \leq \mu_{2e_i}(x_k,\epsilon)\epsilon^{-2d} \leq \frac{|S^{d-1}|}{d(d+2)}\epsilon^{-d+2}$, therefore, in any case,
\begin{align}
\sigma_1^2:=\text{Var}(F_{1})  \leq  \frac{|S^{d-1}| \|P\|_{L^\infty}}{d(d+2)}\epsilon^{-d+2}+O(\epsilon^{-d+3})\label{Appendix:Proof:Lemma:Variance:F}.
\end{align} 
With the above bounds, we could apply the large deviation theory. First, note that the random variable $F_{1}$ is uniformly bounded by 
\begin{equation}
c_1=2\|f\|_{L^\infty}\epsilon^{-d}\,,
\end{equation}
so we apply Bernstein's inequality to provide a large deviation bound. Recall Bernstein's inequality
\begin{equation}
\Pr \left\{\frac{1}{n-1}\sum_{j\neq k,j=1}^n (F_{1,j} - \mathbb{E}[F_{1}]) > \eta_1 \right\} \leq e^{-\frac{n\eta_1^2}{2\sigma_1^2 + \frac{2}{3}c_1\eta_1}},
\end{equation}
where $\eta_1>0$.
Note that $\mathbb{E}[F_{1}]=O(\epsilon)$, if $x_k \in M_\epsilon$ and $\mathbb{E}[F_{1}]=O(\epsilon^2)$, if $x_k \not \in M_\epsilon$. Hence, we assume $\eta_1=O(\epsilon^{2+s})$, where $s>0$. Then $c_1 \eta_1=O(\epsilon^{-d+2+s})$.  If $\epsilon$ is small enough, $2\sigma_1^2 + \frac{2}{3}c_1\eta_1 \leq C \epsilon^{-d+2}$ for some constant $C$ which depends on $f$ and $P$. We have,
\begin{equation}
\frac{n\eta_1^2}{2\sigma_1^2 + \frac{2}{3}c_1\eta_1} \geq  \frac{n\eta_1^2\epsilon^{d-2}}{C}\,.
\end{equation}
Suppose $n$ is chosen large enough so that
\begin{equation}
 \frac{n\eta_1^2\epsilon^{d-2}}{C} \geq  3\log (n)\,;
\end{equation}
that is, the deviation from the mean is set to
\begin{align}\label{proofalphachoice}
\eta_1 \geq O\Big(\frac{\sqrt{\log (n)}}{n^{1/2}\epsilon^{d/2-1}}\Big)\,.
\end{align}
Note that by the assumption that $\eta_1=O(\epsilon^{2+s})$, we know that $\eta_1/\epsilon^2= \frac{\sqrt{\log (n)}}{n^{1/2}\epsilon^{d/2+1}}\to 0$. It implies that the deviation greater than $\eta_1$ happens with probability less than 
\begin{align}
\exp\left(-\frac{n\eta_1^2}{2\sigma_1^2 + \frac{2}{3}c_1\eta_1}\right)&\leq \exp\left(-\frac{n\eta_1^2\epsilon^{d-2}}{C}\right)= \exp(-3\log (n))= 1/n^3.
\end{align} 
As a result, by a simple union bound, we have
\begin{equation}
\Pr \left\{\frac{1}{n-1}\sum_{j\neq k,\,j=1}^n (F_{1,j} - \mathbb{E}[F_1]) > \eta_1\Big|\,k=1,\ldots,n \right\} \leq ne^{-\frac{n\eta_1^2}{2\sigma_1^2 + \frac{2}{3}c_1\eta_1}}\leq 1/n^2.
\end{equation}
\end{proof}

Denote $\Omega_1$ to be the event space that the deviation $\frac{1}{n-1}\sum_{j\neq k,\,j=1}^n (F_{1,j} - \mathbb{E}[F_1])\leq \eta_1$ for all $i=1,\ldots,n$, where $\eta_1$ is chosen in (\ref{proofalphachoice}) is satisfied.

\begin{lemma}\label{Proof:Theorem1:LemmaF3}
Suppose $\epsilon=\epsilon(n)$ so that $\frac{\sqrt{\log(n)}}{n^{1/2}\epsilon^{d/2+1}}\to 0$ and $\epsilon\to 0$ as $n\to \infty$. We have with probability greater than $1-n^{-2}$ that for all $k=1,\ldots,n$, 
\begin{equation}
e_i^{\top}\left[\frac{1}{n\epsilon^d}\sum_{j=1}^N(x_{k,j}-x_k)    - \mathbb{E}\frac{1}{\epsilon^d}(\iota(X)-\iota(x_k))\chi_{{B}^{\mathbb{R}^p}_\epsilon(x_k)}(\iota(X))\right] =  O\Big(\frac{\sqrt{\log (n)}}{n^{1/2}\epsilon^{d/2-1}}\Big)\,,
\end{equation}
where $i=1,\ldots,d$. And 
\begin{equation}
e_i^{\top}\left[\frac{1}{n\epsilon^d}\sum_{j=1}^N(x_{k,j}-x_k)    - \mathbb{E}\frac{1}{\epsilon^d}(\iota(X)-\iota(x_k))\chi_{{B}^{\mathbb{R}^p}_\epsilon(x_k)}(\iota(X))\right] =  O\Big(\frac{\sqrt{\log (n)}}{n^{1/2}\epsilon^{d/2-2}}\Big)\,,
\end{equation}
where $i=d+1,\ldots,p$.
\end{lemma}

\begin{proof}
Fix $x_k$. By denoting 
\begin{equation}
\frac{1}{n\epsilon^d}\sum_{j=1}^N(x_{k,j}-x_k)=\frac{1}{n}\sum_{j\neq k, j=1}^n\sum_{\ell=1}^pF_{2,\ell,j}e_{\ell}.
\end{equation}
where
\begin{equation}
F_{2,\ell,j}:=\frac{1}{\epsilon^d}e_{\ell}^\top (x_{j}-x_k)\chi_{{B}^{\mathbb{R}^p}_\epsilon(x_k)}(x_{j}), 
\end{equation}
and we know that when $j\neq k$, $F_{2,\ell,j}$ is randomly sampled i.i.d. from 
\begin{equation}
F_{2,\ell}:=\frac{1}{\epsilon^d}e_\ell^\top (\iota(X)-\iota(x_k))\chi_{{B}^{\mathbb{R}^p}_\epsilon(x_k)}(\iota(X)).
\end{equation}
Similarly, we can focus on analyzing $\frac{1}{n-1}\sum_{j=1,j\neq i}^n F_{2,\ell,j}$ since $\frac{n-1}{n}\to 1$ as $n\to \infty$. 
By Lemma \ref{Lemma:6} we have
\begin{align}
\mathbb{E}[F_{2,\ell}]=\,& 
\left\{
\begin{array}{ll}
\displaystyle\big( P(x) \mu_{e_d}(x,\epsilon) \epsilon^{-d} \big) e_{\ell}^\top e_d + \sum_{i=1}^d \big(\partial_iP(x) \mu_{2e_i}(x,\epsilon) \epsilon^{-d}\big) e_{\ell}^\top e_i +O(\epsilon^{d+3}) & \mbox{ when }\ell=1,\ldots,d\\
\displaystyle\frac{P(x) \epsilon^{-d}}{2} e_{\ell}^\top \sum_{i=1}^d \mathfrak{N}_{ii}(x) \mu_{2e_i} + O(\epsilon^{d+3}) &\mbox{ when }\ell=d+1,\ldots,p.
\end{array}
\right.\nonumber
\end{align}
In other words, by Lemma \ref{Lemma:5},  for $\ell=1,\ldots,d$ we have $\mathbb{E}[F_{2,\ell}]=O(\epsilon)$ if $x_k \in M_{\epsilon}$, and $\mathbb{E}[F_{2,\ell}]=O(\epsilon^2)$ if $x_k \not \in M_{\epsilon}$. Moreover, $\mathbb{E}[F_{2,\ell}]=O(\epsilon^2)$  for $\ell=d+1,\ldots,p$.
By (\ref{Proof:Cx}) we have, for $\ell=1,\ldots,d$
\begin{align}
\mathbb{E}[F_{2,\ell}^2] \leq C_\ell \epsilon^{-d+2} +O(\epsilon^{-d+3}),
\end{align}
and $C_\ell$ depends on $\|P\|_{L^{\infty}}$. 
For $\ell=d+1,\ldots,p$,
\begin{align}
\mathbb{E}[F_{2,\ell}^2] \leq C_\ell \epsilon^{-d+4} +O(\epsilon^{-d+5}),
\end{align}
and $C_\ell$ depends on $\|P\|_{L^{\infty}}$ and second fundamental form of $M$. 

Thus, we conclude that
\begin{align}
&\sigma_{2,\ell}^2 \leq C_\ell \epsilon^{-d+2} +O(\epsilon^{-d+3}) \mbox{ when }\ell=1,\ldots,d\\
&\sigma_{2,\ell}^2 \leq C_\ell \epsilon^{-d+4} +O(\epsilon^{-d+5}) \mbox{ when }\ell=d+1,\ldots,p\,.\nonumber
\end{align} 
Note that for $\ell=d+1,\ldots,p$, the variance is of higher order than that of $\ell=1,\ldots,d$. 

With the above bounds, we could apply the large deviation theory. For $\ell=1,\ldots,d$, the random variable $F_{2,\ell}$ is uniformly bounded by $c_{2,\ell}=2\epsilon^{-d+1}$. Since  $\mathbb{E}[F_{2,\ell}]=O(\epsilon)$ if $x_k \in M_{\epsilon}$, and $\mathbb{E}[F_{2,\ell}]=O(\epsilon^2)$ if $x_k \not \in M_{\epsilon}$, we assume $\eta_{2,\ell}=O(\epsilon^{2+s})$, where $s>0$. Then $c_{2,\ell} \eta_{2,\ell}=O(\epsilon^{-d+3+s})$. If $\epsilon$ is small enough, $2\sigma_{2,\ell}^2 + \frac{2}{3}c_{2,\ell}\eta_{2,\ell}\leq C \epsilon^{-d+2}$ for some constant $C$ which depends on $P$ and manifold $M$. We have
\begin{equation}
\frac{n\eta_{2,\ell}^2}{2\sigma_{2,\ell}^2 + \frac{2}{3}c_{2,\ell}\eta_{2,\ell}} \geq  \frac{n\eta_{2,\ell}^2\epsilon^{d-2}}{C}\,.
\end{equation}
Suppose $n$ is chosen large enough so that
\begin{equation}
 \frac{n\eta_{2,\ell}^2\epsilon^{d-2}}{C} \geq 3\log (n)\,;
\end{equation}
that is, the deviation from the mean is set to
\begin{align} \label{proof:alphaChoiceF2finald}
\eta_{2,\ell} \geq O\Big(\frac{\sqrt{\log (n)}}{n^{1/2}\epsilon^{d/2-1}}\Big)\,.
\end{align}
Note that by the assumption that $\eta_{2,\ell}=O(\epsilon^{2+s})$, we know that $\eta_{2,\ell}/\epsilon^2= \frac{\sqrt{\log (n)}}{n^{1/2}\epsilon^{d/2+1}}\to 0$.
Thus, when $\epsilon$ is sufficiently smaller and $n$ is sufficiently large, the exponent in Bernstein's inequality 
\begin{equation}
\Pr \left\{\frac{1}{n-1}\sum_{j\neq k,j=1}^n (F_{2,\ell,j} - \mathbb{E}[F_{2,\ell}]) > \eta_{2,\ell} \right\} \leq \exp\Big(-\frac{n\eta_{2,\ell}^2}{2\sigma_{2,\ell}^2 + \frac{2}{3}c_{2,\ell}\eta_{2,\ell}}\Big) \leq \frac{1}{n^3}.
\end{equation}
By a simple union bound, for $\ell=1,\ldots,d$, we have
\begin{align}
\Pr \left\{\left|\frac{1}{n}\sum_{j\neq k,\,j=1}^n F_{2,\ell,j} - \mathbb{E}[F_{2,\ell}]\right| > \eta_{2,\ell}\Big|\,k=1,\ldots,n \right\}  \leq 1/n^2.
\end{align}

For $\ell=d+1,\ldots,p$, the random variable $F_{2,\ell}$ is uniformly bounded by $c_{2,\ell}=2\epsilon^{-d+1}$. Since $\mathbb{E}[F_{2,\ell}]=O(\epsilon^2)$  for $\ell=d+1,\ldots,p$, we assume $\eta_{2,\ell}=O(\epsilon^{3+s})$, where $s>0$. Then $c_{2,\ell} \eta_{2,\ell}=O(\epsilon^{-d+4+s})$. If $\epsilon$ is small enough, $2\sigma_{2,\ell}^2 + \frac{2}{3}c_{2,\ell}\eta_{2,\ell}\leq C \epsilon^{-d+4}$ for some constant $C$ which depends on $M$ and $P$. We have,
\begin{equation}
\frac{n\eta_{2,\ell}^2}{2\sigma_{2,\ell}^2 + \frac{2}{3}c_{2,\ell}\eta_{2,\ell}} \geq  \frac{n\eta_{2,\ell}^2\epsilon^{d-4}}{C}\,.
\end{equation}
Suppose $n$ is chosen large enough so that
\begin{equation}
 \frac{n\eta_{2,\ell}^2\epsilon^{d-4}}{C}= 3\log (n)\,;
\end{equation}
that is, the deviation from the mean is set to
\begin{align} \label{proof:alphaChoiceF2finalp}
\eta_{2,\ell} =O\Big(\frac{\sqrt{\log (n)}}{n^{1/2}\epsilon^{d/2-2}}\Big)\,.
\end{align}
Note that by the assumption that $\beta_1=O(\epsilon^{3+s})$, we know that $\eta_{2,\ell}/\epsilon^3= \frac{\sqrt{\log (n)}}{n^{1/2}\epsilon^{d/2+1}}\to 0$.

By a similar argument, for $\ell=d+1,\ldots,p$, we have
\begin{align}
\Pr \left\{\left|\frac{1}{n}\sum_{j\neq k,\,j=1}^n F_{2,\ell,j} - \mathbb{E}[F_{2,\ell}]\right| > \eta_{2,\ell}\Big|\,k=1,\ldots,n \right\}  \leq 1/n^2.
\end{align}

\end{proof}

Denote $\Omega_2$ to be the event space that the deviation $\left|\frac{1}{n}\sum_{j\neq k,\,j=1}^n F_{2,\ell,j} - \mathbb{E}[F_{2,\ell}]\right| \leq \eta_{2,\ell}$ for all $\ell=1,\ldots,p$ and $k=1,\ldots,n$, where $\eta_{2,\ell}$ are chosen in (\ref{proof:alphaChoiceF2finald}) and (\ref{proof:alphaChoiceF2finalp}). 
Next Lemma summarizes behavior of (\ref{Proof:Theorem1:FiniteSum3}) and can be proved similarly as Lemma \ref{Proof:Theorem1:LemmaF3}. 

\begin{lemma}\label{Proof:Theorem1:LemmaF3.1}
Suppose $\epsilon=\epsilon(n)$ so that $\frac{\sqrt{\log(n)}}{n^{1/2}\epsilon^{d/2+1}}\to 0$ and $\epsilon\to 0$ as $n\to \infty$. We have with probability greater than $1-n^{-2}$ that for all $k=1,\ldots,n$, 
\begin{equation}
e_i^{\top}\left[\frac{1}{n\epsilon^d}\sum_{j=1}^N(x_{k,j}-x_k) (f(x_{k,j})-f(x_k))   - \mathbb{E}\frac{1}{\epsilon^d}(\iota(X)-\iota(x_k))(f(X)-f(x_k))\chi_{{B}^{\mathbb{R}^p}_\epsilon(x_k)}(\iota(X))\right] =  O\Big(\frac{\sqrt{\log (n)}}{n^{1/2}\epsilon^{d/2-2}}\Big)\,,\nonumber
\end{equation}
where $i=1,\ldots,d$, and 
\begin{equation}
e_i^{\top}\left[\frac{1}{n\epsilon^d}\sum_{j=1}^N(x_{k,j}-x_k) (f(x_{k,j})-f(x_k))   - \mathbb{E}\frac{1}{\epsilon^d}(\iota(X)-\iota(x_k))(f(X)-f(x_k))\chi_{{B}^{\mathbb{R}^p}_\epsilon(x_k)}(\iota(X))\right] =  O\Big(\frac{\sqrt{\log (n)}}{n^{1/2}\epsilon^{d/2-3}}\Big)\,,\nonumber
\end{equation}
where $i=d+1,\ldots,p$.
\end{lemma}

Denote $\Omega_3$ to be the event space that  Lemma \ref{Proof:Theorem1:LemmaF3.1} is satisfied.  
In the next two lemmas, we describe the behavior of $\frac{1}{n\epsilon^{d}}G_{n,k}G_{n,k}^{\top}$. The proofs are the same as Lemma E.4 in \cite{Wu_Wu:2017} with $\rho=3$.
\begin{lemma}\label{Proof:Theorem1:LemmaF4}
Suppose $\epsilon=\epsilon(n)$ so that $\frac{\sqrt{\log(n)}}{n^{1/2}\epsilon^{d/2+1}}\to 0$ and $\epsilon\to 0$ as $n\to \infty$. We have with probability greater than $1-n^{-2}$ that for all $k=1,\ldots,n$,
\begin{equation}
\Big|e_i^{\top} \Big(\frac{1}{n\epsilon^{d}}G_{n,k}G_{n,k}^{\top}-\frac{1}{\epsilon^d}C_{x_k} \Big)e_j \Big|=O\Big(\frac{\sqrt{\log(n)}}{n^{1/2}\epsilon^{d/2-2}}\Big),
\end{equation}
where $i,j=1,\ldots,d$.
\begin{equation}
\Big|e_i^{\top} \Big(\frac{1}{n\epsilon^{d}}G_{n,k}G_{n,k}^{\top}-\frac{1}{\epsilon^d}C_{x_k} \Big)e_j \Big|=O\Big(\frac{\sqrt{\log(n)}}{n^{1/2}\epsilon^{d/2-4}}\Big),
\end{equation}
where $i,j=1+1,\ldots,p$.
\begin{equation}
\Big|e_i^{\top} \Big(\frac{1}{n\epsilon^{d}}G_{n,k}G_{n,k}^{\top}-\frac{1}{\epsilon^d}C_{x_k}\Big )e_j \Big|=O\Big(\frac{\sqrt{\log(n)}}{n^{1/2}\epsilon^{d/2-3}}\Big),
\end{equation}
otherwise.
\end{lemma}


\begin{lemma}\label{Proof:Theorem1:LemmaF5}
$r_n \leq r$ and $r_n$ is a non decreasing function of $n$. If $n$ is large enough, $r_n=r$.
Suppose $\epsilon=\epsilon(n)$ so that $\frac{\sqrt{\log(n)}}{n^{1/2}\epsilon^{d/2+1}}\to 0$ and $\epsilon\to 0$ as $n\to \infty$. We have with probability greater than $1-n^{-2}$ that for all $k=1,\ldots,n$,
\begin{align}
&\Big|e_i^\top \Big[I_{p,r_n}\Big(\frac{\Lambda_n}{n\epsilon^d}+\epsilon^{3} I_{p\times p}\Big)^{-1}-I_{p,r}\Big(\frac{\Lambda}{\epsilon^d}+\epsilon^{3} I_{p\times p}\Big)^{-1}\Big]e_i\Big| 
= O\Big(\frac{\sqrt{\log(n)}}{n^{1/2}\epsilon^{d/2+2}}\Big)
\end{align} 
for $i=1,\ldots,r$ and
\begin{equation}
U_n=U\Theta+\frac{\sqrt{\log(n)}}{n^{1/2}\epsilon^{d/2-2}}U\Theta\mathsf{S}+O\Big(\frac{\log(n)}{n\epsilon^{d-4}}\Big), 
\end{equation}
where $\mathsf{S}\in\mathfrak{o}(p)$, and $\Theta \in O(p)$. $\Theta$ commutes with $ I_{p,r}(\frac{\Lambda}{\epsilon^d}+\epsilon^{3} I_{p\times p})^{-1}$.
\end{lemma}

Denote $\Omega_4$ to be the event space that  Lemma \ref{Proof:Theorem1:LemmaF5} is satisfied. 
In the proofs of Lemma \ref{Lemma:8} and Theorem \ref{Theorem:t1}, we need the order $\epsilon^{d+3}$ terms of the eigenvalues $\{\lambda_i\}$ of $C_x$ for $i=1, \cdots, d$ and we need the order $\epsilon$ term of the eigenvectors $\{\beta_i\}$ of $C_x$ for $i=1, \cdots, p$. We also use the fact that $\{\lambda_i\}$ of $C_x$ for $i=d+1, \cdots, p$ are of order $\epsilon^{d+4}$, so that we can calculate the leading terms (order $\epsilon^2$) of $Q_{\epsilon}f(x)$ for all $x \in M$. Since $\frac{\sqrt{\log(n)}}{n^{1/2}\epsilon^{d/2+1}}\to 0$, the above two lemmas imply that the differences between the first $d$ eigenvalues of $\frac{1}{n\epsilon^{d}}G_{n,k}G_{n,k}^{\top}$ and $\frac{1}{\epsilon^d}C_{x_k}$ are less than $O(\epsilon^3)$. The differences between the rest of the eigenvalues of $\frac{1}{n\epsilon^{d}}G_{n,k}G_{n,k}^{\top}$ and $\frac{1}{\epsilon^d}C_{x_k}$ are less than $O(\epsilon^4)$. In other words, we can make sure that the rest of the eigenvalues of   $\frac{1}{n\epsilon^{d}}G_{n,k}G_{n,k}^{\top}$ are of order $\epsilon^4$. Moreover $U_n$ and $U\Theta$ differ by a matrix of order $\epsilon^3$. Consequently, in the following proof, we can show that the deviation between $\sum_{j=1}^n [W-I_{n \times n}]_{kj}f(x_{k,j})$ and $Q_{\epsilon}f(x_k)$ is less than $\epsilon^2$ for all $x_k$.

\begin{proof}[Proof of Theorem \ref{Theorem:t0}.]
Denote $\Omega:=\cap_{i=0,\ldots,4} \Omega_i$. By a direct union bound, the probability of the event space $\Omega$ is great than $1-n^{-2}$. Below, all arguments are conditional on $\Omega$.  Based on previous lemmas, we have, for $k=1, \ldots, n$,
\begin{align} \label{proof of thm est 1}
\frac{1}{n\epsilon^d}\sum_{j=1}^N1 = \mathbb{E}\frac{1}{\epsilon^d}\chi_{{B}^{\mathbb{R}^p}_\epsilon(x_k)}(\iota(X)) + O\Big(\frac{\sqrt{\log (n)}}{n^{1/2}\epsilon^{d/2}}\Big)\,, 
\end{align}
\begin{align} \label{proof of thm est 2}
\frac{1}{n\epsilon^d}\sum_{j=1}^N(f(x_{k,j})-f(x_k))=\mathbb{E}\frac{1}{\epsilon^d}(f(X)-f(x_k))\chi_{{B}^{\mathbb{R}^p}_\epsilon(x_k)}(\iota(X))+ O\Big(\frac{\sqrt{\log (n)}}{n^{1/2}\epsilon^{d/2-1}}\Big)\,,
\end{align}
and
\begin{align} \label{proof of thm est 3}
\frac{1}{n\epsilon^d}\sum_{j=1}^N(x_{k,j}-x_k) = \mathbb{E}\frac{1}{\epsilon^d}(\iota(X)-\iota(x_k))\chi_{{B}^{\mathbb{R}^p}_\epsilon(x_k)}(\iota(X))+ \mathcal{E}_1 \,,
\end{align}
where $\mathcal{E}_1 \in \mathbb{R}^p$, $e_i^\top \mathcal{E}_1= O\Big(\frac{\sqrt{\log (n)}}{n^{1/2}\epsilon^{d/2-1}}\Big)$ for $i=1,\ldots,d$, and $e_i^\top \mathcal{E}_1= O\Big(\frac{\sqrt{\log (n)}}{n^{1/2}\epsilon^{d/2-2}}\Big)$ for $i=d+1,\ldots,p$.
Moreover, we have
\begin{align} 
&\frac{1}{n\epsilon^d}\sum_{j=1}^N(x_{k,j}-x_k) (f(x_{k,j})-f(x_k)) \label{proof of thm est 4}\\
=&\, \mathbb{E}\frac{1}{\epsilon^d}(\iota(X)-\iota(x_k))(f(X)-f(x_k))\chi_{{B}^{\mathbb{R}^p}_\epsilon(x_k)}(\iota(X))+\mathcal{E}_2, \nonumber
\end{align}
where $\mathcal{E}_2 \in \mathbb{R}^p$. $e_i^\top \mathcal{E}_2= O\Big(\frac{\sqrt{\log (n)}}{n^{1/2}\epsilon^{d/2-2}}\Big)$ for $i=1,\ldots,d$, and $e_i^\top \mathcal{E}_2= O\Big(\frac{\sqrt{\log (n)}}{n^{1/2}\epsilon^{d/2-3}}\Big)$ for $i=d+1,\ldots,p$.
Therefore, we have
\begin{align*}
&\, U_n I_{p,r_n}\Big(\frac{\Lambda_n}{n\epsilon^d}+\epsilon^{3} I_{p\times p}\Big)^{-1}U_n^\top-UI_{p,r}\Big(\frac{\Lambda}{\epsilon^d}+\epsilon^{3} I_{p\times p}\Big)^{-1}U^\top \\
=&\, \Big(U\Theta+\frac{\sqrt{\log(n)}}{n^{1/2}\epsilon^{d/2-2}}U\Theta\mathsf{S}+O\Big(\frac{\log(n)}{n\epsilon^{d-4}}\Big)\Big)\Big(I_{p,r}\Big(\frac{\Lambda}{\epsilon^d}+\epsilon^{3} I_{p\times p}\Big)^{-1}+O\Big(\frac{\sqrt{\log(n)}}{n^{1/2}\epsilon^{d/2+2}}\Big)\Big) \nonumber\\
&\qquad\times\Big(U\Theta+\frac{\sqrt{\log(n)}}{n^{1/2}\epsilon^{d/2-2}}U\Theta\mathsf{S}+O(\frac{\log(n)}{n\epsilon^{d-4}})\Big)^\top - UI_{p,r}\Big(\frac{\Lambda}{\epsilon^d}+\epsilon^{3} I_{p\times p}\Big)^{-1}U^\top. \nonumber \\
=&\, \frac{\sqrt{\log(n)}}{n^{1/2}\epsilon^{d/2-2}} U\Theta\Big(S I_{p,r}\Big(\frac{\Lambda}{\epsilon^d}+\epsilon^{3} I_{p\times p}\Big)^{-1}+I_{p,r}\Big(\frac{\Lambda}{\epsilon^d}+\epsilon^{3} I_{p\times p})^{-1} S^\top\Big)\Theta^\top U^\top\nonumber\\
& \qquad+O\Big(\frac{\sqrt{\log(n)}}{n^{1/2}\epsilon^{d/2+2}}\Big) I_{p \times p} +\big[\mbox{higher order terms}\big]. \nonumber 
\end{align*}
Define a $p \times p$ matrix 
\begin{align}
\mathcal{E}_3=&\frac{\sqrt{\log(n)}}{n^{1/2}\epsilon^{d/2-2}} U\Theta\Big[S I_{p,r}\Big(\frac{\Lambda}{\epsilon^d}+\epsilon^{3} I_{p\times p}\Big)^{-1}+I_{p,r}\Big(\frac{\Lambda}{\epsilon^d}+\epsilon^{3} I_{p\times p}\Big)^{-1} S^\top\Big]\Theta^\top U^\top\\
&\qquad +O\Big(\frac{\sqrt{\log(n)}}{n^{1/2}\epsilon^{d/2+2}}\Big) I_{p \times p}\nonumber\,.
\end{align}
We have
\begin{align}
& [\frac{1}{n\epsilon^d}\sum_{j=1}^N(x_{k,j}-x_k)]^{\top}U_n I_{p,r_n}(\frac{\Lambda_n}{n\epsilon^d}+\epsilon^{3} I_{p\times p})^{-1}U_n^\top [\frac{1}{n\epsilon^d}\sum_{j=1}^N(x_{k,j}-x_k)(f(x_{k,j})-f(x_k))] \nonumber\\
=\,& [ \mathbb{E}\frac{1}{\epsilon^d}(\iota(X)-\iota(x_k))\chi_{{B}^{\mathbb{R}^p}_\epsilon(x_k)}(\iota(X))+ \mathcal{E}_1]^\top[UI_{p,r}(\frac{\Lambda}{\epsilon^d}+\epsilon^{3} I_{p\times p})^{-1}U^\top+\mathcal{E}_3+\mbox{higher order terms}] \nonumber \\ 
&\qquad \times[ \mathbb{E}\frac{1}{\epsilon^d}(\iota(X)-\iota(x_k))(f(X)-f(x_k))\chi_{{B}^{\mathbb{R}^p}_\epsilon(x_k)}(\iota(X))+\mathcal{E}_2] \nonumber \\
=\,& \mathbb{E}\frac{1}{\epsilon^d}(\iota(X)-\iota(x_k))\chi_{{B}^{\mathbb{R}^p}_\epsilon(x_k)}(\iota(X))^\top[ UI_{p,r}(\frac{\Lambda}{\epsilon^d}+\epsilon^{3} I_{p\times p})^{-1}U^\top] \nonumber\\
&\qquad\times\mathbb{E}\frac{1}{\epsilon^d}(\iota(X)-\iota(x_k))(f(X)-f(x_k))\chi_{{B}^{\mathbb{R}^p}_\epsilon(x_k)}(\iota(X)) \nonumber \\
&+ \mathcal{E}_1^\top UI_{p,r}(\frac{\Lambda}{\epsilon^d}+\epsilon^{3} I_{p\times p})^{-1}U^\top  \mathbb{E}\frac{1}{\epsilon^d}(\iota(X)-\iota(x_k))(f(X)-f(x_k))\chi_{{B}^{\mathbb{R}^p}_\epsilon(x_k)}(\iota(X))  \nonumber \\
&+  \mathbb{E}\frac{1}{\epsilon^d}(\iota(X)-\iota(x_k))\chi_{{B}^{\mathbb{R}^p}_\epsilon(x_k)}(\iota(X))^\top \mathcal{E}_3 \mathbb{E}\frac{1}{\epsilon^d}(\iota(X)-\iota(x_k))(f(X)-f(x_k))\chi_{{B}^{\mathbb{R}^p}_\epsilon(x_k)}(\iota(X)) \nonumber \\
&+ \mathbb{E}\frac{1}{\epsilon^d}(\iota(X)-\iota(x_k))\chi_{{B}^{\mathbb{R}^p}_\epsilon(x_k)}(\iota(X))^\top UI_{p,r}(\frac{\Lambda}{\epsilon^d}+\epsilon^{3} I_{p\times p})^{-1}U^\top \mathcal{E}_2 +\mbox{higher order terms} \nonumber \,.
\end{align}
Note that 
\begin{align}
&\mathbb{E}\frac{1}{\epsilon^d}(\iota(X)-\iota(x_k))\chi_{{B}^{\mathbb{R}^p}_\epsilon(x_k)}(\iota(X))^\top UI_{p,r}(\frac{\Lambda}{\epsilon^d}+\epsilon^{3} I_{p\times p})^{-1}U^\top \mathcal{E}_2= \mathbf{T}_{\iota(x_k)} \mathcal{E}_2\,.
\end{align}
When $x \in M_{\epsilon}$
\begin{align}
\mathbf{T}_{\iota(x_k)} \mathcal{E}_2 =& [\![O(\epsilon^{-1}), O(\epsilon^{-1}) ]\!] \cdot \Big[\!\!\!\Big[O\Big(\frac{\sqrt{\log (n)}}{n^{1/2}\epsilon^{d/2-2}}\Big),O\Big(\frac{\sqrt{\log (n)}}{n^{1/2}\epsilon^{d/2-3}}\Big) \Big]\!\!\!\Big]=O\Big(\frac{\sqrt{\log (n)}}{n^{1/2}\epsilon^{d/2-1}}\Big).\nonumber
\end{align}
When $x \not \in M_{\epsilon}$
\begin{align}
\mathbf{T}_{\iota(x_k)} \mathcal{E}_2 =& [\![O(1), O(\epsilon^{-1}) ]\!] \cdot \Big[\!\!\!\Big[O\Big(\frac{\sqrt{\log (n)}}{n^{1/2}\epsilon^{d/2-2}}\Big),O\Big(\frac{\sqrt{\log (n)}}{n^{1/2}\epsilon^{d/2-3}}\Big) \Big]\!\!\!\Big]=O\Big(\frac{\sqrt{\log (n)}}{n^{1/2}\epsilon^{d/2-2}}\Big).\nonumber
\end{align}
Moreover, when $x_k \in M_{\epsilon}$ or $x_k \in M \setminus M_{\epsilon}$ by a similar calculation as in Lemma \ref{Lemma:8}, $ UI_{p,r}(\frac{\Lambda}{\epsilon^d}+\epsilon^{3} I_{p\times p})^{-1}U^\top  \mathbb{E}\frac{1}{\epsilon^d}(\iota(X)-\iota(x_k))(f(X)-f(x_k))\chi_{{B}^{\mathbb{R}^p}_\epsilon(x_k)}(\iota(X))=[\![O(1), O(1) ]\!] $. Hence,
\begin{align}
\mathcal{E}_1^\top UI_{p,r}(\frac{\Lambda}{\epsilon^d}+\epsilon^{3} I_{p\times p})^{-1}U^\top  \mathbb{E}\frac{1}{\epsilon^d}(\iota(X)-\iota(x_k))(f(X)-f(x_k))\chi_{{B}^{\mathbb{R}^p}_\epsilon(x_k)}(\iota(X))=O\Big(\frac{\sqrt{\log (n)}}{n^{1/2}\epsilon^{d/2-1}}\Big).\nonumber
\end{align}
Next, we calculate $ \mathbb{E}\frac{1}{\epsilon^d}(\iota(X)-\iota(x_k))\chi_{{B}^{\mathbb{R}^p}_\epsilon(x_k)}(\iota(X))^\top \mathcal{E}_3 \mathbb{E}\frac{1}{\epsilon^d}(\iota(X)-\iota(x_k))(f(X)-f(x_k))\chi_{{B}^{\mathbb{R}^p}_\epsilon(x_k)}(\iota(X))$. By a straightforward calculation, we can show that it is dominated by 
\begin{align}
O\big(\frac{\sqrt{\log(n)}}{n^{1/2}\epsilon^{d/2+2}}\big) \mathbb{E}\frac{1}{\epsilon^d}(\iota(X)-\iota(x_k))\chi_{{B}^{\mathbb{R}^p}_\epsilon(x_k)}(\iota(X)) ^\top \mathbb{E}\frac{1}{\epsilon^d}(\iota(X)-\iota(x_k))(f(X)-f(x_k))\chi_{{B}^{\mathbb{R}^p}_\epsilon(x_k)}(\iota(X)).\nonumber
\end{align}
Hence, when $x_k \in M_{\epsilon}$, 
\begin{equation}
\mathbb{E}\frac{1}{\epsilon^d}(\iota(X)-\iota(x_k))\chi_{{B}^{\mathbb{R}^p}_\epsilon(x_k)}(\iota(X)) \mathcal{E}_3 \mathbb{E}\frac{1}{\epsilon^d}(\iota(X)-\iota(x_k))(f(X)-f(x_k))\chi_{{B}^{\mathbb{R}^p}_\epsilon(x_k)}(\iota(X))= O\Big(\frac{\sqrt{\log(n)}}{n^{1/2}\epsilon^{d/2-1}}\Big).\nonumber
\end{equation}
When $x_k \not \in M_{\epsilon}$,
\begin{equation}
\mathbb{E}\frac{1}{\epsilon^d}(\iota(X)-\iota(x_k))\chi_{{B}^{\mathbb{R}^p}_\epsilon(x_k)}(\iota(X)) \mathcal{E}_3 \mathbb{E}\frac{1}{\epsilon^d}(\iota(X)-\iota(x_k))(f(X)-f(x_k))\chi_{{B}^{\mathbb{R}^p}_\epsilon(x_k)}(\iota(X))= O\Big(\frac{\sqrt{\log(n)}}{n^{1/2}\epsilon^{d/2-2}}\Big).\nonumber
\end{equation}
In conclusion for $k=1,\cdots,n $, we have
\begin{align} 
& \Big[\frac{1}{n\epsilon^d}\sum_{j=1}^N(x_{k,j}-x_k)\Big]^{\top}U_n I_{p,r_n}\Big(\frac{\Lambda_n}{n\epsilon^d}+\epsilon^{3} I_{p\times p}\Big)^{-1}U_n^\top \Big[\frac{1}{n\epsilon^d}\sum_{j=1}^N(x_{k,j}-x_k)(f(x_{k,j})-f(x_k))\Big]\nonumber \\
=\,& \mathbb{E}\frac{1}{\epsilon^d}(\iota(X)-\iota(x_k))\chi_{{B}^{\mathbb{R}^p}_\epsilon(x_k)}(\iota(X))^\top\Big[ UI_{p,r}(\frac{\Lambda}{\epsilon^d}+\epsilon^{3} I_{p\times p})^{-1}U^\top\Big]\nonumber\\
&\qquad\times \mathbb{E}\frac{1}{\epsilon^d}(\iota(X)-\iota(x_k))(f(X)-f(x_k))\chi_{{B}^{\mathbb{R}^p}_\epsilon(x_k)}(\iota(X)) +O\Big(\frac{\sqrt{\log(n)}}{n^{1/2}\epsilon^{d/2-1}}\Big).\nonumber 
\end{align}
A similar argument shows that for $k=1,\cdots,n $,
\begin{align} 
& \Big[\frac{1}{n\epsilon^d}\sum_{j=1}^N(x_{k,j}-x_k)\Big]^{\top}U_n I_{p,r_n}\Big(\frac{\Lambda_n}{n\epsilon^d}+\epsilon^{3} I_{p\times p}\Big)^{-1}U_n^\top  \Big[\frac{1}{n\epsilon^d}\sum_{j=1}^N(x_{k,j}-x_k)\Big] \\
=&\, \mathbb{E}\frac{1}{\epsilon^d}(\iota(X)-\iota(x_k))\chi_{{B}^{\mathbb{R}^p}_\epsilon(x_k)}(\iota(X))^\top\Big[ UI_{p,r}\Big(\frac{\Lambda}{\epsilon^d}+\epsilon^{3} I_{p\times p}\Big)^{-1}U^\top\Big] \mathbb{E}\frac{1}{\epsilon^d}(\iota(X)-\iota(x_k))\chi_{{B}^{\mathbb{R}^p}_\epsilon(x_k)}(\iota(X)) \nonumber \\ 
& +O\Big(\frac{\sqrt{\log(n)}}{n^{1/2}\epsilon^{d/2}}\Big).\nonumber 
\end{align}
By Theorem \ref{Theorem:t1}, $g_1$ has order $O(\epsilon^2)$ and $g_2$ has order $1$. Hence, we have
\begin{align}
\sum_{j=1}^n [W-I_{n \times n}]_{kj}f(x_{k,j})&=\frac{g_1+O\big(\frac{\sqrt{\log(n)}}{n^{1/2}\epsilon^{d/2-1}}\big)}{g_2+O\big(\frac{\sqrt{\log(n)}}{n^{1/2}\epsilon^{d/2}}\big)}=Q_{\epsilon}f(x_k)+O\Big(\frac{\sqrt{\log(n)}}{n^{1/2}\epsilon^{d/2-1}}\Big).
\end{align}
\end{proof}

\section{Detailed calculation of \eqref{theorem (W-I)T(W-I)}}\label{proof (W-I)T(W-I)}
{

We show that the operator $\bar{Q}_\epsilon f(x)$ may not approximate the Laplace-Beltrami operator in the interior of $\iota(M)=[-1,1] \subset \mathbb{R}$. Let $x,y\in M$ such that $\iota(x) \in [-1+\epsilon, 1-\epsilon]$. Since $M$ is 1-dimensional and is embedded in $\mathbb{R}$, the local covariance matrix $C_x$ has rank $1$. Suppose $P \in C^2(M)$ and  $\epsilon>0$ is small enough. By \cite[Lemma SI.6, Case 0]{Wu_Wu:2017}  and the fact that $K(x,y)=0$ when $|\iota(x)-\iota(y)|\geq\epsilon$, we have
\begin{align}
K(x,y)=\left(1-\frac{P'(x)}{P(x)}(\iota(y)-\iota(x))+O(\epsilon^2)\right)\chi_{B_{\epsilon}^{\mathbb{R}}(\iota(x))}.
\end{align}
Hence,
\begin{align}
K(y,x)=\left(1-\frac{P'(y)}{P(y)}(\iota(x)-\iota(y))+O(\epsilon^2)\right)\chi_{B_{\epsilon}^{\mathbb{R}}(\iota(x))}.
\end{align}
Without loss of generality, consider $\iota(x)=0$. Let $\gamma(t):[-1,1] \rightarrow M$ be the arclength parametrization of $M$ with $\gamma(0)=x$.  Suppose $f \in C^3(M)$. Then, by a straightforward expansion in the parametrization $\gamma(t)$, we have
\begin{align}
\bar{Q}_\epsilon f(x) :=&\, \frac{\mathbb{E}[K_{\epsilon}(X,x)[f(X)-f(x)]]}{\mathbb{E}K_{\epsilon}(X,x)} \nonumber \\
=&\frac{\int_{-\epsilon}^\epsilon \big(1+\frac{(P\circ \gamma)'(t)}{(P\circ \gamma)(t)}t+O(\epsilon^2)\big)\big((f \circ \gamma)'(0)t+\frac{(f \circ \gamma)''(0)}{2}t^2+O(t^3)\big)(P\circ \gamma)(t) dt}{\int_{-\epsilon}^\epsilon \big(1+\frac{(P\circ \gamma)'(t)}{(P\circ \gamma)(t)}t+O(\epsilon^2)\big)(P\circ \gamma)(t) dt} \label{operator Qbar expansion}
\end{align}
The numerator of \eqref{operator Qbar expansion} can be expanded as
\begin{align}
& \int_{-\epsilon}^\epsilon \big(1+\frac{(P\circ \gamma)'(t)}{(P\circ \gamma)(t)}t+O(\epsilon^2)\big)\big((f \circ \gamma)'(0)t+\frac{(f \circ \gamma)''(0)}{2}t^2+O(t^3)\big)(P\circ \gamma)(t) dt \nonumber \\
=&  (f \circ \gamma)'(0) \int_{-\epsilon}^\epsilon (P\circ \gamma)(t) t dt + (f \circ \gamma)'(0) \int_{-\epsilon}^\epsilon \frac{(P\circ \gamma)'(t)}{(P\circ \gamma)(t)}(P\circ \gamma)(t)t^2 dt \nonumber  \\
&  \quad \quad \quad \quad + \frac{(f \circ \gamma)''(0)}{2} \int_{-\epsilon}^\epsilon (P\circ \gamma)(t) t^2 dt +O(\epsilon^4) \nonumber \\
=&  (f \circ \gamma)'(0) \int_{-\epsilon}^\epsilon (P\circ \gamma)(t) t dt + (f \circ \gamma)'(0) \int_{-\epsilon}^\epsilon (P\circ \gamma)'(t)t^2 dt \nonumber  \\
&  \quad \quad \quad \quad + \frac{(f \circ \gamma)''(0)}{2} \int_{-\epsilon}^\epsilon (P\circ \gamma)(t) t^2 dt +O(\epsilon^4) \nonumber \\
=&  (f \circ \gamma)'(0) \int_{-\epsilon}^\epsilon \bigg((P\circ \gamma)(0) t +(P\circ \gamma)'(0)t^2+O(t^3) \bigg)dt + (f \circ \gamma)'(0) \int_{-\epsilon}^\epsilon \bigg((P\circ \gamma)'(0)t^2+O(t^3)\bigg) dt \nonumber  \\
&  \quad \quad \quad \quad + \frac{(f \circ \gamma)''(0)}{2} \int_{-\epsilon}^\epsilon((P\circ \gamma)(0) t^2 O(t^3) \bigg) dt +O(\epsilon^4) \nonumber\\
=&  \frac{2}{3}(f \circ \gamma)'(0) (P\circ \gamma)'(0) \epsilon^3 + \frac{2}{3}(f \circ \gamma)'(0) (P\circ \gamma)'(0) \epsilon^3  + \frac{1}{3}(f \circ \gamma)''(0) (P\circ \gamma)(0) \epsilon^3 +O(\epsilon^4) \nonumber  \\
=& \Big( \frac{4}{3}(f \circ \gamma)'(0) (P\circ \gamma)'(0) + \frac{1}{3}(f \circ \gamma)''(0) (P\circ \gamma)(0) \Big) \epsilon^3 +O(\epsilon^4) .\nonumber  
\end{align}
The denominator of \eqref{operator Qbar expansion} can be expanded as
\begin{align}
\int_{-\epsilon}^\epsilon \Big(1+\frac{(P\circ \gamma)'(t)}{(P\circ \gamma)(t)}t+O(\epsilon^2)\Big)(P\circ \gamma)(t) dt =2(P\circ \gamma)(0) \epsilon+O(\epsilon^2). \nonumber
\end{align}
Hence, we have
\begin{align}
\bar{Q}_\epsilon f(x) = & \frac{\Big[ \frac{4}{3}(f \circ \gamma)'(0) (P\circ \gamma)'(0) + \frac{1}{3}(f \circ \gamma)''(0) (P\circ \gamma)(0) \Big] \epsilon^3 +O(\epsilon^4)}{2(P\circ \gamma)(0) \epsilon+O(\epsilon^2)} \\
=&\left( \frac{2}{3}\frac{(f \circ \gamma)'(0) (P\circ \gamma)'(0)}{P\circ \gamma)(0)}+ \frac{1}{6}(f \circ \gamma)''(0) \right) \epsilon^2 +O(\epsilon^3) \nonumber \\
=& \left(\frac{1}{6}f ''(x)+\frac{2}{3}\frac{f '(x) P'(x)}{P(x)} \right)\epsilon^2 +O(\epsilon^3) \nonumber 
\end{align}
Recall that $Q_\epsilon f(x)=\frac{1}{6}f ''(x)\epsilon^2 +O(\epsilon^3)$. However, since the kernel is not symmetric, $\bar{Q}_\epsilon$ does not approximate the Laplace-Beltrami operator $\Delta$. Moreover, when $f \in C^5(M)$,  we conclude that
$$
\bar{Q}_\epsilon Q_\epsilon f(x)=\frac{1}{9}\left(\frac{1}{4}f ''''(x) +\frac{f '''(x) P'(x)}{P(x)}\right)\epsilon^4 +O(\epsilon^5).
$$
}

\bibliographystyle{plain}
\bibliography{boundaryLLE}   

\begin{thebibliography}{10}

\bibitem{atkinson176141discrete}
F.V. Atkinson.
\newblock Discrete and continuous boundary value problems.
\newblock {\em Acad. Press., New York. MR}, 176141, 1964.

\bibitem{bates2014embedding}
Jonathan Bates.
\newblock The embedding dimension of {Laplacian} eigenfunction maps.
\newblock {\em Applied and Computational Harmonic Analysis}, 37(3):516--530,
  2014.

\bibitem{Belkin_Niyogi:2003}
M.~Belkin and P.~Niyogi.
\newblock Laplacian {Eigenmaps} for dimensionality reduction and data
  representation.
\newblock {\em Neural. Comput.}, 15(6):1373--1396, 2003.

\bibitem{Belkin_Niyogi:2007}
M.~Belkin and P.~Niyogi.
\newblock {Convergence of {L}aplacian eigenmaps}.
\newblock In {\em Adv. Neur. In.: Proceedings of the 2006 Conference},
  volume~19, page 129. The MIT Press, 2007.

\bibitem{berard1994embedding}
Pierre B{\'e}rard, G{\'e}rard Besson, and Sylvain Gallot.
\newblock Embedding {Riemannian} manifolds by their heat kernel.
\newblock {\em Geometric \& Functional Analysis GAFA}, 4(4):373--398, 1994.

\bibitem{berry2017density}
T.~Berry and T.~Sauer.
\newblock Density estimation on manifolds with boundary.
\newblock {\em Computational Statistics and Data Analysis}, 107:1--17, 2017.

\bibitem{calder2020lipschitz}
Jeff Calder, Nicolas~Garcia Trillos, and Marta Lewicka.
\newblock Lipschitz regularity of graph {Laplacians} on random data clouds.
\newblock {\em SIAM Journal on Mathematical Analysis}, 54(1):1169--1222, 2022.

\bibitem{chang1999regularity}
Sun-Yung~A Chang, Lihe Wang, and Paul~C Yang.
\newblock A regularity theory of biharmonic maps.
\newblock {\em Communications on Pure and Applied Mathematics},
  52(9):1113--1137, 1999.

\bibitem{Cheng_Wu:2013}
M.-Y. Cheng and H.-T. Wu.
\newblock Local linear regression on manifolds and its geometric
  interpretation.
\newblock {\em J. Am. Stat. Assoc.}, 108:1421--1434, 2013.

\bibitem{cheng2022eigen}
Xiuyuan Cheng and Nan Wu.
\newblock Eigen-convergence of {Gaussian kernelized graph Laplacian} by
  manifold heat interpolation.
\newblock {\em Applied and Computational Harmonic Analysis}, 61:132--190, 2022.

\bibitem{Coifman_Lafon:2006}
R.~R. Coifman and S.~Lafon.
\newblock Diffusion maps.
\newblock {\em Appl. Comput. Harmon. Anal.}, 21(1):5--30, 2006.

\bibitem{cuccu2009maximization}
Fabrizio Cuccu and Giovanni Porru.
\newblock Maximization of the first eigenvalue in problems involving the
  {bi-Laplacian}.
\newblock {\em Nonlinear Analysis: Theory, Methods \& Applications},
  71(12):e800--e809, 2009.

\bibitem{ding2020impact}
Xiucai Ding and Hau-Tieng Wu.
\newblock Impact of signal-to-noise ratio and bandwidth on graph {Laplacian}
  spectrum from high-dimensional noisy point cloud.
\newblock {\em IEEE Transactions on Information Theory}, 2022.

\bibitem{Donoho_Grimes:2003}
D.~L. Donoho and C.~Grimes.
\newblock {Hessian eigenmaps: Locally linear embedding techniques for
  high-dimensional data}.
\newblock {\em P. Natl. Acad. Sci. USA}, 100(10):5591--5596, 2003.

\bibitem{dunson2021spectral}
David~B Dunson, Hau-Tieng Wu, and Nan Wu.
\newblock Spectral convergence of graph {Laplacian and heat kernel
  reconstruction in $L^\infty$ from random samples}.
\newblock {\em Applied and Computational Harmonic Analysis}, 55:282--336, 2021.

\bibitem{dunson2021inferring}
David~B Dunson and Nan Wu.
\newblock Inferring manifolds from noisy data using {Gaussian} processes.
\newblock {\em arXiv preprint arXiv:2110.07478}, 2021.

\bibitem{epstein2013degenerate}
C.~L. Epstein and R.~Mazzeo.
\newblock {\em Degenerate diffusion operators arising in population biology}.
\newblock Number 185. Princeton University Press, 2013.

\bibitem{Gao:2016}
Tingran Gao.
\newblock The diffusion geometry of fibre bundles: Horizontal diffusion maps.
\newblock {\em Applied and Computational Harmonic Analysis}, 50:147--215, 2021.

\bibitem{georgiou2017exploration}
A.~S. Georgiou, J.~M. Bello-Rivas, C.~W. Gear, H.-T. Wu, E.~Chiavazzo, and
  I.~G. Kevrekidis.
\newblock An exploration algorithm for stochastic simulators driven by energy
  gradients.
\newblock {\em Entropy}, 19(7):294, 2017.

\bibitem{kaslovsky2014non}
D.~N. Kaslovsky and F.~G. Meyer.
\newblock Non-asymptotic analysis of tangent space perturbation.
\newblock {\em Information and Inference: a Journal of the IMA}, 3(2):134--187,
  2014.

\bibitem{kershaw1983some}
D.~Kershaw.
\newblock Some extensions of {W. Gautschi's} inequalities for the gamma
  function.
\newblock {\em Mathematics of Computation}, pages 607--611, 1983.

\bibitem{lee2012smooth}
John~M Lee.
\newblock {\em Smooth manifolds}.
\newblock Springer, 2012.

\bibitem{Lin_Minasian_Wu:2016}
Chen-Yun Lin, Arin Minasian, Xin~Jessica Qi, and Hau-Tieng Wu.
\newblock Manifold learning via the principle bundle approach.
\newblock {\em Frontiers in Applied Mathematics and Statistics}, 4:21, 2018.

\bibitem{lin2017extrinsic}
L.~Lin, B.~{St. Thomas}, H.~Zhu, and D.~B. Dunson.
\newblock Extrinsic local regression on manifold-valued data.
\newblock {\em Journal of the American Statistical Association},
  112(519):1261--1273, 2017.

\bibitem{little2017multiscale}
A.~V. Little, M.~Maggioni, and L.~Rosasco.
\newblock Multiscale geometric methods for data sets {I}: Multiscale {SVD},
  noise and curvature.
\newblock {\em Applied and Computational Harmonic Analysis}, 43(3):504--567,
  2017.

\bibitem{2018arXiv180402811M}
J.~{Malik}, C.~{Shen}, H.-T. {Wu}, and N.~{Wu}.
\newblock Connecting dots -- from local covariance to empirical intrinsic
  geometry and locally linear embedding.
\newblock {\em Pure and Applied Analysis}, 1(4):515 -- 542, 2019.

\bibitem{mukherjee2010learning}
S.~Mukherjee, Q.~Wu, and D.-X. Zhou.
\newblock Learning gradients on manifolds.
\newblock {\em Bernoulli}, 16(1):181--207, 2010.

\bibitem{nadler2008finite}
B.~Nadler.
\newblock Finite sample approximation results for principal component analysis:
  A matrix perturbation approach.
\newblock {\em The Annals of Statistics}, 36(6):2791--2817, 2008.

\bibitem{naimark1967linear}
M.~A. Na{\"\i}mark.
\newblock {\em Linear differential operators}.
\newblock F. Ungar Publishing Company, 1967.

\bibitem{osher2017low}
S.~Osher, Z.~Shi, and W.~Zhu.
\newblock Low dimensional manifold model for image processing.
\newblock {\em SIAM Journal on Imaging Sciences}, 10(4):1669--1690, 2017.

\bibitem{peoples2021spectral}
J~Wilson Peoples and John Harlim.
\newblock Spectral convergence of symmetrized graph {Laplacian} on manifolds
  with boundary.
\newblock {\em arXiv preprint arXiv:2110.06988}, 2021.

\bibitem{portegies2016embeddings}
Jacobus~W Portegies.
\newblock Embeddings of {Riemannian} manifolds with heat kernels and
  eigenfunctions.
\newblock {\em Communications on Pure and Applied Mathematics}, 69(3):478--518,
  2016.

\bibitem{Roweis_Saul:2000}
S.~T. Roweis and L.~K. Saul.
\newblock Nonlinear dimensionality reduction by locally linear embedding.
\newblock {\em Science}, 290(5500):2323--2326, 2000.

\bibitem{shnitzer2018recovering}
T.~Shnitzer, M.~Ben-Chen, L.~Guibas, R.~Talmon, and H.-T. Wu.
\newblock Recovering hidden components in multimodal data with composite
  diffusion operators.
\newblock {\em arXiv preprint arXiv:1808.07312}, 2018.

\bibitem{Singer_Wu:2012}
A.~Singer and H.-T. Wu.
\newblock Vector diffusion maps and the connection {Laplacian}.
\newblock {\em Communications on Pure and Applied Mathematics},
  65(8):1067--1144, 2012.

\bibitem{Singer_Wu:2016}
A.~Singer and H.-T. Wu.
\newblock Spectral convergence of the connection {Laplacian} from random
  samples.
\newblock {\em Information and Inference: A Journal of the IMA}, 6(1):58--123,
  2017.

\bibitem{Tenenbaum_deSilva_Langford:2000}
J.~B. Tenenbaum, V.~{de Silva}, and J.~C. Langford.
\newblock A global geometric framework for nonlinear dimensionality reduction.
\newblock {\em Science}, 290(5500):2319--2323, 2000.

\bibitem{trillos2018error}
N.~G. Trillos, M.~Gerlach, M.~Hein, and D.~Slepcev.
\newblock Error estimates for spectral convergence of the graph {Laplacian} on
  random geometric graphs toward the {Laplace--Beltrami} operator.
\newblock {\em Foundations of Computational Mathematics}, 20(4):827--887, 2020.

\bibitem{tyagi2013tangent}
H.~Tyagi, E.~Vural, and P.~Frossard.
\newblock Tangent space estimation for smooth embeddings of {Riemannian}
  manifolds.
\newblock {\em Information and Inference: A Journal of the IMA}, 2(1):69--114,
  2013.

\bibitem{VanderMaaten_Hinton:2008}
L.~{van der Maaten} and G.~Hinton.
\newblock {Visualizing data using t-SNE}.
\newblock {\em Journal of Machine Learning Research}, 9:2579--2605, 2008.

\bibitem{vaughn2019diffusion}
R.~Vaughn, T.~Berry, and H.~Antil.
\newblock Diffusion maps for embedded manifolds with boundary with applications
  to pdes.
\newblock {\em arXiv preprint arXiv:1912.01391}, 2019.

\bibitem{vaughn2020diffusion}
Ryan Vaughn.
\newblock {\em DIffusion maps for manifolds with boundary}.
\newblock PhD thesis, George Mason University, 2020.

\bibitem{von2008consistency}
U.~{Von Luxburg}, M.~Belkin, and O.~Bousquet.
\newblock Consistency of spectral clustering.
\newblock {\em The Annals of Statistics}, pages 555--586, 2008.

\bibitem{Weinberger_Saul:2006}
K.Q. Weinberger and L.K. Saul.
\newblock {An introduction to nonlinear dimensionality reduction by maximum
  variance unfolding}.
\newblock {\em Aaai}, pages 1683--1686, 2006.

\bibitem{whitney1992analytic}
Hassler Whitney.
\newblock Analytic extensions of differentiable functions defined in closed
  sets.
\newblock In {\em Hassler Whitney Collected Papers}, pages 228--254. Springer,
  1992.

\bibitem{wong2008extension}
J.~Wong.
\newblock An extension procedure for manifolds with boundary.
\newblock {\em Pacific Journal of Mathematics}, 235(1):173--199, 2008.

\bibitem{wormell2020spectral}
Caroline~L Wormell and Sebastian Reich.
\newblock Spectral convergence of diffusion maps: Improved error bounds and an
  alternative normalization.
\newblock {\em SIAM Journal on Numerical Analysis}, 59(3):1687--1734, 2021.

\bibitem{Wu_Wu:2017}
H.-T. Wu and N~Wu.
\newblock Think globally, fit locally under the manifold setup: Asymptotic
  analysis of locally linear embedding.
\newblock {\em Annals of Statistics}, 46(6B):3805--3837, 2018.

\bibitem{wu2022strong}
Hau-Tieng Wu and Nan Wu.
\newblock Strong uniform consistency with rates for kernel density estimators
  with general kernels on manifolds.
\newblock {\em Information and Inference: A Journal of the IMA},
  11(2):781--799, 2022.

\bibitem{yang2016bayesian}
Y.~Yang and D.~B. Dunson.
\newblock Bayesian manifold regression.
\newblock {\em The Annals of Statistics}, 44(2):876--905, 2016.

\end{thebibliography}

\end{document}